\documentclass[reqno,11pt]{amsart}
\usepackage[latin1]{inputenc}
\usepackage[T1]{fontenc}
\usepackage{lmodern}
\usepackage{a4wide}
\usepackage{amssymb}
\usepackage{enumitem}

\newcommand\NN{\mathbb{N}}
\newcommand\RR{\mathbb{R}}
\newcommand\ZZ{\mathbb{Z}}

\newcommand\E{\mathcal{E}}
\newcommand\Y{\mathcal{Y}}
\newcommand\Z[1]{\mathcal{Z}_{#1}}
\newcommand\Cinfini{\mathcal{C}^{\infty}(\RR,\RR)}

\newcommand\wt{\widetilde}
\newcommand\Qb{\wt{Q}_b}
\newcommand\Psib{\wt{\Psi}_b}
\newcommand\charfunc[1]{{\mathbf{1}}_{#1}}

\newcommand\e{\varepsilon}
\newcommand\g{\gamma}
\renewcommand\leq{\leqslant}
\renewcommand\geq{\geqslant}

\newcommand\m{\mbox}

\newcommand\be{\begin{equation}}
\newcommand\ee{\end{equation}}

\theoremstyle{plain}
\newtheorem{theorem}{Theorem}[section]
\newtheorem{lemma}[theorem]{Lemma}
\newtheorem{proposition}[theorem]{Proposition}

\newtheorem{corollary}[theorem]{Corollary}
\theoremstyle{definition}

\newtheorem{remark}[theorem]{Remark}
\newtheorem{notation}[theorem]{Notation}
\numberwithin{equation}{section}

\title[Minimal mass solution for gKdV]{Sharp asymptotics for the minimal mass blow up solution of critical gKdV equation}

\author[V. Combet]{Vianney Combet}
\address{Univ.~Lille, CNRS, UMR 8524 - Laboratoire Paul Painlevé, F-59000 Lille, France}
\email{vianney.combet@math.univ-lille1.fr}

\author[Y. Martel]{Yvan Martel}
\address{CMLS, École Polytechnique, CNRS, Université Paris-Saclay, 91128 Palaiseau, France}
\email{yvan.martel@polytechnique.edu}

\begin{document}

\begin{abstract}
Let $S$ be a minimal mass blow up solution of the critical generalized KdV equation as constructed in~\cite{MMR2}.
We prove both time and space sharp asymptotics for $S$ close to the blow up time.
Let $Q$ be the unique ground state of~(gKdV), satisfying $Q''+Q^5=Q$.

First, we show that there exist
universal smooth profiles $Q_k\in\mathcal{S}(\RR)$ (with $Q_0=Q$) and a constant $c_0\in\RR$ such that,
fixing the blow up time at $t=0$ and appropriate scaling and translation parameters, $S$ satisfies,
for any $m\geq 0$,
\[
\partial_x^m S(t) - \sum_{k=0}^{[m/2]} \frac 1{t^{\frac 12+m-2k}} Q_k^{(m-k)}\left(\frac{\cdot+ \frac1t}{t}+c_0\right)\to 0\quad
\m{ in }\ L^2 \m{ as }\ t\downarrow 0.
\]
Second, we prove that, for $0<t\ll 1$, $x\leq -\frac 1t -1$,
\[
S(t,x) \sim - \frac 12 \|Q\|_{L^1} |x|^{-3/2},
\]
and related bounds for the derivatives of $S(t)$ of any order.
We also prove $\int_{\RR} S(t,x)\,dx=0$.
\end{abstract}

\maketitle

\section{Introduction}

\subsection{Blow up and dynamics around the soliton for critical (gKdV)}

We consider the mass critical generalized Korteweg--de Vries equation:
\be \label{kdv}
\mathrm{(gKdV)}\quad\left\{
\begin{alignedat}{2}
&u_t + (u_{xx} + u^5)_x =0, \quad && (t,x)\in [0,T)\times\RR, \\
&u(0,x)= u_0(x), && x\in\RR.
\end{alignedat}
\right.
\ee

We first recall a few well-known facts about equation~\eqref{kdv}.
The Cauchy problem is locally well-posed in the energy space $H^1(\RR)$ from~\cite{KPV,KPV2}:
given $u_0 \in H^1$, there exists a unique (in a certain sense) maximal solution $u$ of~\eqref{kdv} in $\mathcal{C}([0,T),H^1)$ and
\be \label{blowucifoi}
T<+\infty \quad \m{implies} \quad \lim_{t\uparrow T} \|u_x(t)\|_{L^2} = +\infty.
\ee
Moreover, such $H^1$ solutions satisfy the conservation of mass and energy:
\[
M(u(t))=\int u^2(t)= M(u_0), \quad E(u(t))= \frac12 \int u_x^2(t) - \frac16 \int u^6(t)= E(u_0).
\]
Equation~\eqref{kdv} is invariant under scaling and translation: if $u$ is a solution of~\eqref{kdv}
then $u_{\lambda,x_0}$, defined by
\[
u_{\lambda,x_0}(t,x)=\lambda^{\frac12} u(\lambda^3t,\lambda x+x_0), \quad \lambda>0, \ x_0\in \RR,
\]
is also a solution of~\eqref{kdv}.

Recall that solitons are solutions of~\eqref{kdv} of the form $u(t,x)=Q(x-t)$,
where $Q$ is the ground state solitary wave
\[
Q(x) = \left(\frac3{\cosh^2\left(2x\right)}\right)^{\frac14},\quad Q''+Q^5=Q,
\]
which attains the best constant in the sharp Gagliardo--Nirenberg inequality, as proved in~\cite{W1983}:
\[
\forall v\in H^1,\quad \int v^6 \leq 3 \int v_x^2 \left(\frac{\int v^2}{\int Q^2}\right)^2.
\]
The conservation of mass and energy and the blow up criterion~\eqref{blowucifoi}
ensure that $H^1$ initial data with \emph{subcritical mass},
\emph{i.e.} $\|u_0\|_{L^2}<\|Q\|_{L^2}$, generate global in time solutions.

\medskip

The existence and properties of blow up solutions for~\eqref{kdv} were studied in two series of works
for $H^1$ initial data satisfying the following \emph{slightly supercritical mass} constraint:
\be \label{in:1}
\|Q\|_{L^2} \leq \|u_0\|_{L^2} \leq (1+\delta_0) \|Q\|_{L^2} \quad \m{where }\ 0<\delta_0\ll 1.
\ee

In the first series of works~\cite{MMjmpa,MMgafa,MMannals,MMjams,Mjams}, the main results can be summarized as follows:
(1) proof of blow up in finite or infinite time for general $H^1$ initial data with negative energy;
(2) proof of blow up in finite time with a lower bound on the blow up rate under an additional decay assumption on the initial data.

The second series of works~\cite{MMNR,MMR1,MMR2,MMR3} provides a better understanding of the blow up phenomenon under assumption~\eqref{in:1},
and more generally a classification of the (gKdV) flow for initial data close to $Q$ in a topology stronger than $H^1$.
Indeed, it is proved in~\cite{MMR1} that for initial data in
\[
\mathcal{A} = \left\{ u_0=Q+\e_0 \m{ with } \|\e_0\|_{H^1}<\alpha_0 \m{ and }
\int_{y>0} y^{10}\e_0^2< 1 \right\},\quad 0<\alpha_0\ll 1,
\]
only three possible behaviors can occur:
\begin{description}
\item[(Blowup)] The solution blows up in finite time $T>0$ with blow up rate $\frac 1{T-t}$.
\item[(Soliton)] The solution is global and locally converges to a soliton in large time.
\item[(Exit)] The solution defocuses and eventually exits any small $L^2$ neighborhood of the soliton.
\end{description}
It is also proved in~\cite{MMNR} that the (Soliton) case corresponds
to a codimension one manifold of initial data in $\mathcal{A}$ that separates the cases (Blowup) and (Exit).
The necessity of a decay assumption in the definition of the set $\mathcal{A}$ is also clarified:
blow up solutions with various blow up rates (in finite or infinite time) are constructed in~\cite{MMR3}
for initial data arbitrarily close to~$Q$ in the energy space
(these initial data obviously do not belong to $\mathcal{A}$).

Finally, the question of the minimal mass blow up dynamics is addressed in~\cite{MMR2};
we recall the following existence and uniqueness result.

\begin{theorem}[Existence and uniqueness of the minimal mass solution~\cite{MMR2}] \label{thm:previous}
\leavevmode
\begin{enumerate}[label=\emph{(\roman*)}]
\item \label{existsS} \emph{Existence.}
There exist a solution $S\in \mathcal{C}((0,+\infty),H^1)$ to~\eqref{kdv}
and universal constants $c_0\in\RR,C_0>0$ such that $\|S(t)\|_{L^2}=\|Q\|_{L^2}$ for all $t>0$ and
\begin{gather}
\|\partial_x S(t)\|_{L^2} \sim \frac{C_0}t \quad \m{ as }\ t\downarrow 0, \nonumber \\
S(t)-\frac1{t^{\frac12}} Q\left(\frac{\cdot+ \frac1t}t + c_0\right)\to 0\quad \m{ in }\ L^2 \m{ as }\ t\downarrow 0. \label{d:S}
\end{gather}
\item \emph{Regularity and exponential decay on the right.}
For all $m\geq 1$, $S\in \mathcal{C}((0,+\infty),H^m)$.
There exist $C(t),\g(t)>0$ such that, for all $t\in (0,1]$, $x\geq -\frac 1t$,
\[
|S(t,x)|\leq C(t)\exp\left(-\g(t)\left({x+\frac 1t}\right)\right).
\]
\item \label{uniqueS} \emph{Uniqueness.}
Let $u_0\in H^1$ with $\|u_0\|_{L^2}=\|Q\|_{L^2}$
and assume that the corresponding solution $u(t)$ to~\eqref{kdv} blows up in finite time.
Then, $u\equiv S$ up to the symmetries of the flow.
\end{enumerate}
\end{theorem}

The solution $S$ is interesting in itself, since it is the only solution (up to symmetries) that blows up with the critical mass $\|Q\|_{L^2}$.
It is also important in the general dynamics of (gKdV)
since it appears as a universal profile of all solutions in the (Exit) case.
In this context, an important open problem is the description of the behavior of $S(t)$ as $t\to +\infty$;
in particular, whether or not $S(t)$ scatters, \emph{i.e.}~behaves in $L^2$ as a solution of the linear Airy equation $u_t+u_{xxx}=0$,
would have important consequences for all solutions in the (Exit) case. We refer to~\cite{MMR2} on these questions.

\subsection{Previous minimal mass blow up results for critical (NLS)}

Recall that the first results on minimal mass blow up for nonlinear dispersive models concern
the mass critical nonlinear Schr\"odinger equation (in dimension $N\geq 1$),
\[
\mathrm{(NLS)} \quad i \partial_t u+\Delta u +|u|^{\frac4N} u=0, \quad (t,x)\in \RR\times \RR^N.
\]
For (NLS), it is well-known that the pseudo conformal symmetry generates an \emph{explicit} minimal mass blow up solution
\be \label{defSNLS}
S_\mathrm{NLS}(t,x)=\frac1 {t^{\frac N2}} e^{- i\frac{|x|^2}{4t}- \frac i t} Q_\mathrm{NLS}\left(\frac x t\right),
\ee
defined for all $t>0$ and blowing up as $t\downarrow 0$, where $Q_\mathrm{NLS}$ is the ground state of (NLS), solution to
\[
\Delta Q_\mathrm{NLS}-Q_\mathrm{NLS}+Q_\mathrm{NLS}^{1+\frac4N}=0, \quad Q>0, \quad Q\in H^1(\RR^N).
\]
Using the pseudo conformal symmetry, it was proved in~\cite{Mduke} that $S_\mathrm{NLS}$ is the unique
(up to the symmetries of the equation) minimal mass blow up element in the energy space.

For the inhomogeneous mass critical (NLS) in dimension~2,
\[
i \partial_t u+\Delta u +k(x)|u|^2 u=0,
\]
while~\cite{Me} derived sufficient conditions on the function $k(x)>0$ to ensure
the \emph{nonexistence} of minimal elements,~\cite{RS2010} introduced a new approach to obtain existence and uniqueness
of a minimal blow up solution under a necessary and sufficient condition on $k(x)$.
This work shows that the minimal blow up dynamics is not directly related to the pseudo conformal symmetry,
even if this symmetry heavily simplifies the analysis for critical (NLS).
For other constructions of minimal mass solutions for (NLS) type equations, we refer to~\cite{BCD,BW,CG,KLR,LMR}.
For related works on energy critical models, see~\cite{DM,DR}.

In the above presentation, we have restricted ourselves to minimal mass blow up.
For a review on general blow up results for $L^2$ critical (NLS), we refer to \cite{Ca03,MMRS,MRinvent,Rzurich} and to references therein.

\subsection{Main results}

The objective of this paper is to describe sharply the behavior of the minimal mass solution $S$ defined in Theorem~\ref{thm:previous},
both in time and space, close to the blow up time.

First, the following result generalizes the asymptotic result as $t\downarrow 0$ given by~\eqref{d:S} to any order of derivative.

\begin{theorem}[Time asymptotics] \label{thm:maintime}
Let $S$ and $c_0$ be defined as in~\ref{existsS} of Theorem~\ref{thm:previous}. Then,
\[
S(t) - \frac 1{t^{\frac 12}} Q\left(\frac{\cdot+\frac1t}t+c_0\right) \to 0\quad
\m{ in }\ H^1 \m{ as }\ t\downarrow 0.
\]
More generally, there exist functions $\{Q_k\}_{k\geq 0}\subset \mathcal{S}(\RR)$ (with $Q_0=Q$) such that, for all $m\geq 0$,
\be \label{th:timem}
\partial_x^m S(t) - \sum_{k=0}^{[m/2]} \frac 1{t^{\frac 12+m-2k}} Q_k^{(m-k)}\left(\frac{\cdot+ \frac1t}{t}+c_0\right)\to 0\quad
\m{ in }\ L^2 \m{ as }\ t\downarrow 0.
\ee
\end{theorem}

Second, we give new sharp pointwise asymptotics of the solution $S(t)$ and of all its space derivatives close to the blow up time.

\begin{theorem}[Space asymptotics] \label{thm:mainspace}
Let $S$ be defined as in~\ref{existsS} of Theorem~\ref{thm:previous}.
Then there exists $T_0>0$ such that the following holds.
\begin{enumerate}[label=\emph{(\roman*)}]
\item \label{item:ptleft} \emph{Pointwise asymptotics on the left.}
For all $t\in (0,T_0]$, $x\leq -\frac 1t - 1$,
\be \label{th:ptwise0}
S(t,x) = -\frac 12 {\|Q\|_{L^1}} |x|^{-\frac 32} + O\left(|x|^{-\frac 32-\frac 1{21}}\right).
\ee
For all $m\geq 1$, there exists $C_m>0$ such that,
for all $t\in (0,T_0]$, $x\leq -\frac 1t -1$,
\be \label{th:ptwisem}
\left|\partial_x^m S(t,x)\right|\leq C_m |x|^{-\frac 32-m}.
\ee
\item \emph{Pointwise bounds on the right.}
For all $m\geq 0$, there exist $C_m'>0$ and $\g_m>0$ such that,
for all $t\in (0,T_0]$, $x\geq -\frac 1t$,
\be \label{th:ptwiser}
|\partial_x^mS(t,x)|\leq \frac{C_m'}{t^{\frac 12+m}} \exp\left(-\g_m {\frac { x+\frac 1t} t}\right).
\ee
\item \emph{Integral.}
For all $t\in (0,T_0]$, $S(t)\in L^1(\RR)$ and
\[
\int_\RR S(t,x)\,dx =0.
\]
\end{enumerate}
\end{theorem}

\medskip

\noindent\emph{Comments on the results.}

\smallskip

\emph{1.} Theorem~\ref{thm:maintime} answers questions left open in~\cite{MMR2} on the asymptotic behavior of $S(t)$ as~$t$ goes to $0$,
in particular the interesting question of the convergence in $H^1$ and the value of the energy of $S$.
Indeed, we deduce from the proof that $E(S)=\frac{1}{16} \|Q\|_{L^1}^2$.

The expansion~\eqref{th:timem} may seem similar to what is easily obtained for $S_\mathrm{NLS}(t)$ from the explicit formula~\eqref{defSNLS}
and a Taylor expansion of $e^{-i\frac t4|y|^2}$ with $y=\frac xt$, which gives, for all $m\geq0$ (and for instance in dimension $N=1$),
\[
S_\mathrm{NLS}(t) - e^{-\frac it} \sum_{k=0}^m \frac1{t^{\frac12-k}} Q_k^\mathrm{NLS}\left(\frac\cdot t \right)
\to 0\quad \m{ in }\ H^m \m{ as }\ t\downarrow 0,
\]
where, for $k\geq 0$, $Q_k^\mathrm{NLS}\in \mathcal{S}(\RR)$ is a smooth profile defined by
$Q_k^\mathrm{NLS}(y) = \frac{(-1)^k}{k!}\left(\frac i4\right)^k y^{2k}Q(y)$.

However, there is a subtle difference, which is related to the $|x|^{-3/2}$ tail for $S(t)$ displayed in Theorem~\ref{thm:mainspace}.
Indeed, all functions $Q_k$ in~\eqref{th:timem} are also Schwartz functions (their decay at infinity is actually similar to the one of $Q$),
but we cannot replace $Q_k^{(m-k)}$ by the space derivative of order~$m$ of a Schwartz function.
For example, for $m=2$, it holds
\[
\partial_x^2 S(t)- \frac 1{t^{\frac 52}} Q''\left(\frac{\cdot+\frac1t}t+c_0\right)
-\frac 1{t^{\frac 12}} Q_1'\left(\frac{\cdot+\frac1t}t+c_0\right)
\to 0\quad\m{ in }\ L^2 \m{ as }\ t\downarrow 0,
\]
where $Q_1\in \mathcal{S}$ is not the derivative of a Schwartz function --- see~\eqref{Qun} for a formula for $Q_1$.
This is the reason why the statement in~\eqref{th:timem} is given for each level of spatial derivative.
We refer to Proposition~\ref{prop:SHm} for a more technical statement concerning the behavior of $S(t)$ as $t$ goes to $0$.

\smallskip

\emph{2.} Estimates~\eqref{th:ptwise0}--\eqref{th:ptwisem} in Theorem~\ref{thm:mainspace}~\emph{\ref{item:ptleft}}
say that $S(t,x)$ behaves exactly as the fixed power function $-c|x|^{-\frac 32}$, with gain of decay by space differentiation.
Such low decay for $S(t)$ was expected from~\cite{MMduke,MMR2}.
It also has analogies with decay of blow up profiles for (NLS) and (gKdV) in~\cite{MRcmp} and~\cite{MMR1}.
However, it is a surprise to obtain such a precise description of the tail of the solution.
Observe that this explicit tail, unlike the concentrating soliton, does not translate.
Note also from the proof that the error term $O(|x|^{-\frac 32-\frac 1{21}})$ in~\eqref{th:ptwise0} is uniform in $t$,
and is certainly not optimal.

Moreover, for a given $t\in (0,T_0]$, the condition $x\leq -\frac 1t- 1$
corresponds to points on the left at distance $|x+\frac 1t|\geq 1$ to the center of the soliton (located at $x\sim-\frac 1t$).
This means that, for $t$ close enough to $0$, we have a precise description of the solution up to a finite distance to the center of the soliton.

Also, it is clear from~\eqref{kdv} that one obtains bounds similar to~\eqref{th:ptwisem} for time derivatives of~$S$ (or mixed time-space derivatives).

\smallskip

\emph{3.} We briefly tell our motivations for studying in details the minimal mass solution.

First, we believe that these asymptotics have interest in themselves since they concern a very specific solution of (gKdV) equation.
At first sight, the solution $S$ may seem similar to $S_\mathrm{NLS}$.
But, unlike for the (NLS) equation, $S$ is not explicit and has a power-like tail on the left that makes it non symmetric,
and thus qualitatively quite different.
In a situation where numerically it does not even seem clear how to plot this solution,
Theorems~\ref{thm:maintime} and~\ref{thm:mainspace} give real new insight on the minimal mass blow up for (gKdV).
To our knowledge, obtaining precise space asymptotics such as~\eqref{th:ptwise0}--\eqref{th:ptwisem} for a truly time-dependent solution
(not derived from an elliptic problem) of a nonlinear dispersive equation is new.
Note that general blow up solutions with supercritical mass are expected to have quite different behavior, with oscillations, see \emph{e.g.}~\cite[Fig.~3]{KP}.
Also, the behavior~\eqref{th:ptwise0}--\eqref{th:ptwisem} has nothing to do with the decay properties of the Airy function, which is strongly oscillating.

Second, there is a clear connection between these asymptotics and the construction of exotic blow up solutions with supercritical mass in~\cite{MMR3}.
Indeed, the key idea in~\cite{MMR3} to construct exotic blow up rates for solutions of~\eqref{kdv} close to the soliton
is to consider a soliton-like solution on the background of a fixed explicit tail of the form $-|x|^{-\theta}$ for large $|x|$
(recall that $\theta\in (1,\frac{29}{18})$ in~\cite{MMR3}).
Under certain conditions, this tail drives the blow up and each choice of~$\theta$ selects a specific blow up rate.
The value $\theta=\frac32$ corresponds to the blow up rate $\frac1t$, as $S(t)$.
It is thus consistent to find such a tail for the minimal mass blow up solution,
even though the solutions constructed in~\cite{MMR3} have supercritical mass.

Third, we expect these asymptotics to be decisive when studying multi-soliton blow up, in the spirit of~\cite{Me0} for the (NLS) equation.
This subject is under investigation by the authors and will be the object of a forthcoming publication.

\subsection{Strategy and outline of the proofs of Theorems~\ref{thm:maintime} and~\ref{thm:mainspace}.}

The overall strategy and the tools used in this paper are largely inspired from~\cite{MMR1,MMR2},
but we push forward the analysis on various points, focusing on the question of the asymptotic behavior as $t\downarrow 0$ of
the minimal mass solution. Recall that the solution $S(t)$ is constructed in~\cite{MMR2} in the following explicit form,
for $t>0$ close to $0$:
\begin{gather*}
S(t,x)=\frac{1}{\lambda^{1/2}(t)}\left(Q_{b(t)}(y)+\e(t,y)\right),\ y = \frac{x-x(t)}{\lambda(t)}, \\
\lambda(t)\sim t,\ x(t)\sim -\frac 1t,\ b(t)\sim -t^2,\ \|\e(t)\|_{H^1}\ll 1 \quad \m{as } t\downarrow 0,
\end{gather*}
where $Q_{b(t)}$ is an explicit perturbation of $Q$.

The first improvement is of technical order and concerns the ansatz $Q_b$ used in the above decomposition.
In~\cite{MMR1,MMR2}, the ansatz is of the form $Q_b=Q+bP$, where $P$ is a fixed function
(up to a technical cut-off due to the fact that $P\not \in L^2$).
In Section~\ref{sec:Qb} of the present paper, we define an ansatz for the (gKdV) blow up at any order of $b$,
of the form $Q_b=Q+\sum_{k=1}^K b^k P_k$, for any fixed $K\geq 1$.
This refinement is essential in the sequel to understand the minimal blow up solution at any order of $t$ as $t\downarrow 0$,
\emph{i.e.}~to obtain $\e$ arbitrarily small around the soliton.
None of the functions~$P_k$ are in $L^2$, and so we also have to use a cut-off argument to obtain a profile in~$H^m$.
We refer to~\cite{Ko} for a more direct approach to the construction of exact blow up profile
in the case of the supercritical (gKdV) equation in search of self-similar blow up solutions.

In Section~\ref{sec:time-estimates}, we prove the main time estimates, \emph{i.e.}~Theorem~\ref{thm:maintime}.
First, we actually provide an alternative construction of the minimal mass solution $S(t)$ in $H^1$, as discussed in Remark~3.2 of~\cite{MMR2}.
Second, using special functionals at any level of Sobolev regularity, we are able to prove optimal decay estimates
for $\e(t,y)$ in Sobolev spaces and related weighted spaces.
This is somewhat in the spirit of some arguments in~\cite{Kato,Ma1,LM} but in the blow up context, which is more involved.
We refer to Proposition~\ref{prop:SHm} for results written in terms of~$\e(t)$ in various norms.
Theorem~\ref{thm:maintime} requires an additional Taylor expansion of the parameters $b(t)$, $\lambda(t)$ and $x(t)$ as $t\downarrow 0$,
which is obtained in Section~\ref{sec:proofmaintime}.

Section~\ref{sec:space-estimates} is probably the newest part of the paper.
We show how time estimates and the special properties of $S(t)$ as $t\downarrow 0$ provide sharp asymptotics in space on the left.
The estimates are first proved in Sobolev norms and then translated into pointwise estimates.
To prove the precise asymptotics~\eqref{th:ptwise0}, we use the $L^2$ norm conservation
and a resulting mechanism of exchange of mass between the two terms $Q_{b(t)}$ and $\e(t)$ in the decomposition of $S(t)$.
To find the correct sign in~\eqref{th:ptwise0}, we connect estimates from the region where $\e(t)$ is preponderant
in the expression of $S(t)$ to the region where $b(t) P_1$ (the second term in the series defining $Q_{b(t)}$) is preponderant.

\subsection{Notation}

For $f,g\in L^2(\RR)$, their $L^2$ scalar product is denoted as
\[
(f,g) = \int_\RR f(x)g(x)\,dx.
\]
The Schwartz space $\mathcal{S}$ is classically defined as
\[
\mathcal{S} = \mathcal{S}(\RR) = \{ f\in \Cinfini\ |\ \forall j\in\NN, \forall \alpha\in\NN, \exists C_{j,\alpha}\geq 0,
\forall y\in\RR, |y^\alpha f^{(j)}(y)|\leq C_{j,\alpha} \}.
\]
Let the generator of $L^2$ scaling be
\be \label{Lamb}
\Lambda\e = \frac12\e + y\e'.
\ee
The linearized operator close to $Q$ is defined by
\be \label{defL}
L\e = -\e'' + \e - 5 Q^4\e.
\ee
The following norms will be needed for the time estimates in Section~\ref{sec:time-estimates}:
\[
\|\e\|_{\dot H_B^m}^2 = \int (\partial_y^m \e)^2(y) e^{y/B}\, dy,\qquad
\|\e\|_{\dot H^m}^2   = \int (\partial_y^m \e)^2(y)\, dy.
\]
For $x\in\RR$, we denote $[x]$ its integer part, defined as the unique $n\in\ZZ$ such that $n\leq x<n+1$. \\
All numbers $C$ appearing in inequalities are real positive constants (with respect to the
context), which may change in each step of an inequality. \\
Finally, when $I\subset\RR$ is an interval, $\charfunc{I}$ denotes the characteristic function of $I$.

\subsection*{Acknowledgements}

The authors thank Herbert Koch, Frank Merle and Pierre Rapha\"el for enlightening discussions.
This material is based upon work supported by the National Science Foundation under Grant No.~0932078 000
while Y.M.~was in residence at the Mathematical Sciences Research Institute in Berkeley, California, during the Fall 2015 semester.
This work was also partly supported by the Labex CEMPI (ANR-11-LABX-0007-01) and by the project ERC 291214 BLOWDISOL.

\section{Nonlinear profiles at any order} \label{sec:Qb}

We start by refining the blow up profile and considering an approximation to the renormalized equation.
Looking for a solution to (gKdV) of the form
\[
u(t,x)\sim\frac1{\lambda^{1/2}(t)}\wt Q_{b(t)}\left(\frac{x-x(t)}{\lambda(t)}\right),\quad
\frac{ds}{dt}=\frac{1}{\lambda^3},\quad \frac{x_s}{\lambda}\sim 1,\quad b\sim-\frac{\lambda_s}{\lambda},
\]
leads to the following self-similar equation for $\Qb$:
\[
b_s\frac{\partial \Qb}{\partial b}+b\Lambda \Qb+(\Qb''-\Qb+\Qb^5)'\sim 0.
\]
Expanding the approximate modulated ansatz $\Qb$ and $b_s$ in powers of $b$,
\be \label{def:Qbt}
\Qb := Q+bP_1+b^2 P_2 +\cdots+b^KP_K,\quad b_s \sim -\beta_2b^2-\beta_3 b^3-\cdots -\beta_Kb^K=:-\theta(b),
\ee
where the unknowns are the functions $P_1,\ldots,P_K$ and the coefficients $\beta_2,\ldots,\beta_K$, the above equation writes
\be \label{eq:bs}
-\theta(b) \frac{\partial \Qb}{\partial b}+b\Lambda \Qb+(\Qb''-\Qb+\Qb^5)'\sim 0.
\ee
Let the linearized operator close to $Q$ be given by~\eqref{defL},
then the order $b^k$ in the expansion of~\eqref{eq:bs} leads to the equation
\be \label{eq:LPprime}
(LP_k)' =\Omega'_k+\Lambda P_{k-1}-\Theta_k,
\ee
where $\Omega_k$ and $\Theta_k$ are the coefficients of $b^k$ in the expansions of $\Qb^5$
(except for the nonlinear term involved in $LP_k$)
and $\theta(b) \frac{\partial \Qb}{\partial b}$ respectively.
More precisely, denoting $P_0:=Q$, we set for~$k\geq 1$:
\begin{align}
\Omega_k &:= \sum_{i_1+\cdots+i_5=k} P_{i_1}\cdots P_{i_5} -5Q^4P_k, \label{def:Omegak} \\
\Theta_k &:= \sum_{i=1}^{k-1} i\beta_{k+1-i}P_i = \beta_kP_1+2\beta_{k-1}P_2+\cdots+(k-1)\beta_2P_{k-1}. \label{def:Thetak}
\end{align}
Note that $\Omega_1=\Theta_1=0$, and so~\eqref{eq:LPprime} reduces to $(LP_1)'=\Lambda Q$ for $k=1$.

We will prove in Lemma~\ref{lemma:Qbt} below that~\eqref{eq:LPprime} admits solutions up to any order $K\geq 1$,
such that $\Qb$ defined by~\eqref{def:Qbt} is an approximate solution of~\eqref{eq:bs},
which generalizes the result obtained in~\cite{MMR1} for $K=1$.
As in~\cite{MMR1}, the functions $P_k$ are not in $L^2$ and thus $\Qb$ is a formal blow up profile.
A truncation of $\Qb$ is defined in Section~\ref{sec:locprofiles}.

\subsection{Construction of formal blow up profiles}

\begin{notation}
Define the following functional spaces
\begin{align*}
\Y &= \{ f\in \Cinfini\ |\ \forall j\in\NN, \exists C_j,r_j\geq 0, \forall y\in\RR, |f^{(j)}(y)|\leq C_j(1+|y|)^{r_j}e^{-|y|} \}, \\
\Y_- &= \{ f\in \Cinfini\ |\ \forall j\in\NN, \exists C_j,r_j\geq 0, \forall y<0, |f^{(j)}(y)|\leq C_j(1+|y|)^{r_j}e^{-|y|} \}.
\end{align*}
For $\ell\in\ZZ$, set
\begin{align*}
\Z\ell = \{ f\in \Cinfini\ |\
&\forall j\geq 0, \exists C_j^+,r_j^+\geq0, \forall y>0, |f^{(j)}(y)|\leq C_j^+(1+|y|)^{r_j^+} e^{-|y|} \\
&\forall j>\ell, \exists C_j^-,r_j^-\geq0, \forall y<0, |f^{(j)}(y)|\leq C_j^-(1+|y|)^{r_j^-} e^{-|y|} \\
&\forall 0\leq j\leq\ell, \exists D_j^-\geq0, \forall y<0, |f^{(j)}(y)|\leq D_j^- (1+|y|)^{\ell-j} \}.
\end{align*}
Note that $\Z\ell = \Y$ for $\ell<0$, and that $\Y\subset\mathcal{S}$.
\end{notation}

We now recall without proof the standard properties of the linearized operator $L$
that will be useful for our purpose (see \emph{e.g.}~\cite{MMgafa,W1985}),
and the properties of the function $P_1=P$ constructed in~\cite[Proposition~2.2]{MMR1},
which solves~\eqref{eq:LPprime} for $k=1$.

\begin{lemma}[Properties of $L$] \label{lemma:L}
The self-adjoint operator $L$ defined by~\eqref{defL} on $L^2$ satisfies:
\begin{enumerate}[label=\emph{(\roman*)}]
\item \label{kernel} Kernel: $\ker L = \{aQ' ; a\in\RR\}$;
\item \label{scaling} Scaling: $L(\Lambda Q)=-2Q$;
\item \label{inverse} Invertibility: for any function $h\in\Y$ orthogonal to $Q'$ for the $L^2$ scalar product,
there exists a unique function $f\in\Y$ orthogonal to $Q'$ such that $Lf=h$;
\item Coercivity: there exists $\mu_0>0$ such that, for all $f\in H^1$,
\be \label{coercivity}
(Lf,f) \geq \mu_0\|f\|_{H^1}^2 -\frac1{\mu_0} \left[ (f,Q)^2 +(f,\Lambda Q)^2 +(f,y\Lambda Q)^2 \right].
\ee
\end{enumerate}
\end{lemma}

\begin{lemma}[Properties of $P$] \label{lemma:P}
There exists a unique function $P\in\Z0$ such that
\begin{gather}
(LP)'=\Lambda Q,\quad \lim_{y\to -\infty} P(y) = \frac 12 \int Q,\quad \lim_{y\to +\infty} P(y) =0, \label{eq:LP} \\
(P,Q) = \frac 1{16} \left(\int Q\right)^2 > 0,\quad (P,Q')=0. \label{eq:PQ}
\end{gather}
\end{lemma}

We are now ready to prove the main result of this section, which is a generalization of Proposition~2.2 in~\cite{MMR1}
up to any fixed order $K\geq 1$.

\begin{lemma} \label{lemma:Qbt}
Let $K\geq1$. Then, for all $1\leq k\leq K$, there exist $P_k\in\Z{k-1}$ and $\beta_k\in\RR$ such that,
defining $\Qb$ and $\theta(b)$ by
\be \label{def:Qbthetab}
\Qb = Q+R_b\ \m{ where }\ R_b = \sum_{k=1}^K b^kP_k, \quad \theta(b)=\sum_{k=2}^K \beta_kb^k,
\ee
and $\Omega_k$ and $\Theta_k$ by~\eqref{def:Omegak} and~\eqref{def:Thetak}, then $(LP_k)' =\Omega'_k+\Lambda P_{k-1}-\Theta_k$.

Moreover, $\Qb$ is an approximate solution of~\eqref{eq:bs}, in the sense that the quantity $\Psib$ defined by
\be \label{def:Psibt}
-\Psib = (\Qb''-\Qb+\Qb^5)' +b\Lambda\Qb -\theta(b)\frac{\partial\Qb}{\partial b}
\ee
satisfies
\be \label{eq:Psibt}
-\Psib = b^{K+1} \left[ \sum_{k=1}^{K-1} b^{k-1}R_{K+k-2} + \sum_{k=K}^{4K} b^{k-1} R_{K+k-6} \right],
\ee
where $R_\ell\in\Z\ell$ for $\ell\in\ZZ$. Finally, note that $\beta_2=2$.
\end{lemma}

\begin{remark} \label{rem:Pk}
In the proof of Lemma~\ref{lemma:Qbt} that follows, we will actually show a more precise behavior of $P_k$ on the left
than the assertion $P_k\in\Z{k-1}$. Indeed, we will prove that, for all $1\leq k\leq K$, $P_k^{(k)}\in\Y$, and in particular
there exist $c_{1,k},\ldots,c_{k,k}\in\RR$ and $P_k^{\sharp}\in\Y_-$ such that, for all $y<0$,
\be \label{eq:Pkinf}
P_k(y) = c_{1,k}y^{k-1} +c_{2,k}y^{k-2}+\cdots+c_{k,k} +P_k^{\sharp}(y),
\ee
with the relations $c_{1,k} = -\frac{k-1/2}{k-1}\, c_{1,k-1}$ for $k\geq 2$, and $c_{1,1}=\frac12\int Q >0$ from~\eqref{eq:LP}.

As a consequence, we have $c_{1,k}\neq 0$ for all $k\geq 1$, which proves that the profiles constructed above
exhibit an \emph{actual} polynomial growth on the left, which will be handled by a suitable cut-off, see~\eqref{def:locprofile}.
Finally, since $c_{1,k}$ has the sign of $(-1)^{k-1}$, we can also affirm that $\lim_{y\to -\infty} P_k(y)=+\infty$ for all $k\geq 2$.
\end{remark}

\begin{remark} \label{rem:uniquePk}
In Lemma~\ref{lemma:Qbt}, if we constraint the functions $P_k$ to be orthogonal to $Q'$
(or to any other direction not orthogonal to $Q'$),
we obtain unique functions $P_k$ and unique values of $\beta_k$,
from the characterization of the kernel of $L$ given in Lemma~\ref{lemma:L}~\emph{\ref{kernel}}.
Note that it is easy to impose such orthogonality by adding $a_k Q'$ to each $P_k$, for a suitable $a_k$.
However, this constraint is not needed in the proof and any choice of $P_k$ leads to identical results
(see also Remark~\ref{rem:akPk}).
\end{remark}

\begin{proof}
The proof is done by induction on $K\geq 1$.
Note that the case $K=1$ is true by the Proposition~2.2 of~\cite{MMR1} if we denote $P_1=P$
(recall that we have set $P_0=Q$ and that $\Omega_1=\Theta_1=0$).
Now, with $K\geq 2$ given, we assume there exist $P_k$ and $\beta_k$ for $1\leq k\leq K-1$
which satisfy the lemma and~\eqref{eq:Pkinf}, and we prove the assertions for $k=K$.

\emph{Equation of $P_K$.} Let us solve the equation $(LP_K)' = \Omega'_K +\Lambda P_{K-1}-\Theta_K$ with $P_K$ under the form
\[
P_K = \wt{P}_K -\int_y^{+\infty} (\Lambda P_{K-1} -\Theta_K) - \sum_{k=1}^{K-2} d_{k,K}P_k,
\]
with $\wt{P}_K\in\Y$ and $d_{k,K}\in\RR$ for $1\leq k\leq K-2$ to be determined.
We also assume that the $P_k$ for $1\leq k\leq K-1$ were found under the same form,
which is true for $K=1$ by the proof of Proposition~2.2 in~\cite{MMR1}.
Then we have
\[
(L\wt{P}_K)' = (LP_K)' +\left[ L\int_y^{+\infty} (\Lambda P_{K-1} -\Theta_K)\right]' + \sum_{k=1}^{K-2} d_{k,K}(LP_k)' = S'_K,
\]
where
\begin{align*}
S_K &= \Omega_K + (\Lambda P_{K-1} -\Theta_K)' -5Q^4\int_y^{+\infty} (\Lambda P_{K-1} -\Theta_K) \\
&\quad + \sum_{k=1}^{K-2} d_{k,K}\Omega_k -\sum_{k=1}^{K-2} d_{k,K} \int_y^{+\infty} (\Lambda P_{k-1} -\Theta_k).
\end{align*}

\emph{Choice of $\beta_K$.} Since $LQ'=0$, we can ensure that $(S_K,Q')=0$ by a unique suitable choice of~$\beta_K$.
Indeed, since $(P,Q)\neq0$ from~\eqref{eq:PQ}, we have the following equivalences:
\begin{align*}
\int S_KQ' =0
&\Longleftrightarrow \int S'_KQ=0 \Longleftrightarrow \int [\Omega'_K +\Lambda P_{K-1}-\Theta_K] Q=0 \\
&\Longleftrightarrow \int [\Omega'_K +\Lambda P_{K-1} -\beta_K P -2\beta_{K-1}P_2-\cdots -(K-1)\beta_2 P_{K-1}]Q=0 \\
&\Longleftrightarrow \beta_K = \frac{1}{(P,Q)} \int [\Omega'_K +\Lambda P_{K-1} -2\beta_{K-1}P_2-\cdots -(K-1)\beta_2 P_{K-1}]Q.
\end{align*}
We also obtain, from~\eqref{eq:PQ} and the flux type computation~(2.43) of~\cite{MMR1}, the identity
\[
\beta_2 = \frac{1}{(P,Q)} \int (\Omega'_2 +\Lambda P)Q = \frac{(10(Q^3P^2)'+\Lambda P,Q)}{(P,Q)}
= \frac{\frac18 \|Q\|_{L^1}^2}{\frac{1}{16} \|Q\|_{L^1}^2} = 2.
\]

\emph{Choice of the $d_{k,K}$.} We now want to ensure that $S_K\in\Y$, so that we may apply~\emph{\ref{inverse}} of Lemma~\ref{lemma:L}
and solve $L\wt{P}_K=S_K$ with a unique solution $\wt{P}_K\in\Y$ such that $(\wt{P}_K,Q')=0$.
It is for this purpose that the parameters $d_{k,K}$ were introduced and will be fixed.
Indeed, since we assumed the construction of the previous $P_k$ was similar, we can write
\begin{align*}
S_K &= \Omega_K + \left[ \frac12 P_{K-1} +yP'_{K-1} -\beta_KP -\cdots -2(K-1)P_{K-1}\right]' +\sum_{k=1}^{K-2} d_{k,K}\Omega_k \\
&\quad -\sum_{k=1}^{K-2}d_{k,K} \left[ \wt{P}_k-P_k -\sum_{\ell=1}^{k-2} d_{\ell,k}P_\ell \right] -5Q^4\int_y^{+\infty} (\Lambda P_{K-1}-\Theta_K),
\end{align*}
with $\wt{P}_k\in\Y$ for $1\leq k\leq K-2$. Hence, we may write $S_K$ as
\begin{align*}
S_K &= \Omega_K -(2K-7/2)P'_{K-1} +yP''_{K-1}-\beta_KP'-\cdots -\beta_3(K-2)P'_{K-2} \\
&\quad +\sum_{k=1}^{K-2} d_{k,K}P_k +\sum_{k=1}^{K-2} d_{k,K}\Omega_k +\sum_{k=1}^{K-2} d_{k,K} \sum_{\ell=1}^{k-2} d_{\ell,k}P_\ell \\
&\quad -\sum_{k=1}^{K-2} d_{k,K}\wt{P}_k -5Q^4\int_y^{+\infty} (\Lambda P_{K-1}-\Theta_K) = A+B+C,
\end{align*}
where $A$ are the source terms, $B$ the terms involving the unknowns $d_{k,K}$, and $C\in\Y$.
Finally, since the only obstruction for having $S_K\in\Y$ is when $y\to -\infty$,
we use~\eqref{eq:Pkinf} to expand $S_K$ as series of powers of $y$ when $y\to -\infty$:
\begin{align*}
S_K &= y^{K-3} [d_{K-2,K}c_{1,K-2} -(2K-7/2)(K-2)c_{1,K-1} +(K-2)(K-3)c_{1,K-1}] \\
&\quad +y^{K-4} [d_{K-3,K}c_{1,K-3} +d_{K-2,K}c_{2,K-2}+\cdots] \\
&\quad +\cdots \\
&\quad + y[d_{2,K}c_{1,2}+\cdots] + [d_{1,K}c_{1,1}+\cdots] + \wt{C},
\end{align*}
with $\wt{C}\in\Y_-$.
So, to have $S_K\in\Y$, it is necessary and enough that the $(K-2)$ brackets cancel,
which is possible since this is a triangular system in the unknowns $d_{1,K},\ldots,d_{K-2,K}$ and that $c_{1,1},\ldots,c_{1,K-2}\neq 0$.
For instance, the first bracket gives the identity $d_{K-2,K}c_{1,K-2} = (K-2)(K-1/2)c_{1,K-1}$, which leads,
thanks to the relation between $c_{1,K-1}$ and $c_{1,K-2}$ given by the induction, to
\[
d_{K-2,K} = (K-2)(K-1/2)\frac{c_{1,K-1}}{c_{1,K-2}} = -(K-1/2)(K-3/2).
\]
Then, knowing the value of $d_{K-2,K}$, the second bracket gives $d_{K-3,K}$, and so on.

\emph{Properties of $P_K$.} Thanks to the parameters $\beta_K$ and $d_{1,K},\ldots,d_{K-2,K}$ fixed above,
we have ensured that $S_K\in\Y$ and $(S_K,Q')=0$. Hence, by~\emph{\ref{inverse}} of Lemma~\ref{lemma:L},
there exists a unique $\wt{P}_K\in\Y$ such that $(\wt{P}_K,Q')=0$ and $L\wt{P}_K=S_K$,
and \emph{a fortiori} $(LP_K)' = \Omega'_K +\Lambda P_{K-1} -\Theta_K$ as expected.
To check~\eqref{eq:Pkinf} for $P_K$, we first compute
\begin{align*}
P'_K &= \wt{P}'_K +\Lambda P_{K-1} -\Theta_K -\sum_{k=1}^{K-2} d_{k,K} P'_k \\
&= \wt{P}'_K -(2K-5/2)P_{K-1} +yP'_{K-1} -\beta_K P -\cdots -\beta_3(K-2)P_{K-2} -\sum_{k=1}^{K-2} d_{k,K} P'_k.
\end{align*}
Hence, $P_K^{(K)}\in\Y$, which proves that $P_K$ has an expansion as in~\eqref{eq:Pkinf} when $y\to -\infty$,
with coefficients $c_{1,K},\ldots,c_{K,K}$ (and in particular $P_K\in\Z{K-1}$).
By identification of the terms in~$y^{K-2}$ in the previous equality, we obtain the identity
\[
(K-1)c_{1,K} = -(2K-5/2)c_{1,K-1}+(K-2)c_{1,K-1} = -c_{1,K-1}(K-1/2),
\]
which leads to $c_{1,K} = -\frac{K-1/2}{K-1}\, c_{1,K-1}$ as expected.

\emph{Expression of $\Psib$.} From the definitions~\eqref{def:Omegak} and~\eqref{def:Thetak} of $\Omega_k$ and $\Theta_k$ respectively, we have
\begin{align*}
-\Psib &= \left[ Q'' + \sum_{k=1}^K b^k P''_k -Q -\sum_{k=1}^K b^kP_k +Q^5 +\sum_{k=1}^{5K} b^k(5Q^4P_k) +\sum_{k=1}^{5K} b^k\Omega_k \right]' \\
&\quad +b\sum_{k=0}^K b^k \Lambda P_k - \sum_{k=1}^{2K-1} b^k \Theta_k.
\end{align*}
Since $Q''+Q^5=Q$, the terms of order $0$ cancel and we obtain
\[
-\Psib = \sum_{k=1}^K b^k[ (-LP_k)' +\Omega'_k +\Lambda P_{k-1}-\Theta_k]
+b^{K+1}\Lambda P_K - \sum_{k=K+1}^{2K-1} b^k \Theta_k + \sum_{k=K+1}^{5K} b^k(\Omega_k+5Q^4P_k)'.
\]
From the equations satisfied by the $P_k$, the first sum cancels and we get
\[
-\Psib = b^{K+1}\Lambda P_K - \sum_{k=K+1}^{2K-1} b^k \Theta_k + \sum_{k=K+1}^{5K} b^k(\Omega_k+5Q^4P_k)'.
\]
Finally, since $P_K\in\Z{K-1}$ and $\Omega'_k\in\Z{k-6}$, we may rewrite $-\Psib$ as
\[
-\Psib = \sum_{k=K+1}^{2K-1} b^k R_{k-2} +\sum_{k=2K}^{5K} b^k R_{k-6}
= b^{K+1} \left[ \sum_{k=1}^{K-1} b^{k-1}R_{K+k-2} + \sum_{k=K}^{4K} b^{k-1} R_{K+k-6} \right],
\]
where $R_\ell\in\Z\ell$ for $\ell\in\ZZ$, which concludes the proof of the lemma.
\end{proof}

\subsection{Definition and estimates of the localized profiles} \label{sec:locprofiles}

We now proceed to a simple localization of the profile to avoid some artificial growth at~$-\infty$.
Let $\chi\in \Cinfini$ be such that $0\leq \chi \leq 1$, $\chi'\geq 0$ on $\RR$,
$\chi\equiv 1$ on $[-1,+\infty)$, $\chi\equiv 0$ on $(-\infty,-2]$.
We fix
\be \label{def:gamma}
\frac{17}{20} < \g \leq 1.
\ee
Now we define the localized profiles:
\be \label{def:locprofile}
\chi_b(y)= \chi(|b|^\g y),\quad Q_b(y) = Q(y) + R_b(y) \chi_b(y).
\ee

\begin{lemma}[Properties of localized profiles] \label{lemma:locprofile}
Let $K\geq 1$. The following holds for all $|b|\leq 1$.
\begin{enumerate}[label=\emph{(\roman*)}]
\item \emph{Estimates on $Q_b$:} For all $j\geq 0$ and all $y\in\RR$, we have
\begin{align}
|Q_b^{(j)}(y)| &\lesssim e^{-|y|/2} + |b|^{1+j}\charfunc{[-1,0]}(|b|^\g y) + |b|^{1+j\g} \charfunc{[-2,-1]}(|b|^\g y), \label{eq:Qbj} \\
|(Q_b-Q)^{(j)}(y)| &\lesssim |b|e^{-|y|/2} + |b|^{1+j}\charfunc{[-1,0]}(|b|^\g y) + |b|^{1+j\g} \charfunc{[-2,-1]}(|b|^\g y) \lesssim |b|, \label{eq:QbtoQ} \\
\left| \left( \frac{\partial Q_b}{\partial b} \right)^{(j)} (y) \right| &\lesssim e^{-|y|/2}
+ |b|^j\charfunc{[-1,0]}(|b|^\g y) + |b|^{j\g} \charfunc{[-2,-1]}(|b|^\g y). \label{eq:dQb}
\end{align}
\item \emph{Equation of $Q_b$:} Let
\be \label{def:Psib}
-\Psi_b=(Q_b''- Q_b+ Q_b^5)'+b \Lambda Q_b -\theta(b)\frac{\partial Q_b}{\partial b}.
\ee
Then, for all $0\leq j\leq K-1$ and for all $y\in\RR$,
\be \label{eq:Psibj}
|\Psi_b^{(j)}(y)| \lesssim |b|^{K+1}e^{-|y|/2} + |b|^{K+1}|y|^{K-1-j}\charfunc{[-1,0]}(|b|^\g y) + |b|^{1+(j+1)\g} \charfunc{[-2,-1]}(|b|^\g y).
\ee
\item \emph{Mass and energy of $Q_b$:}
\begin{gather}
\left|\int Q_b^2 - \int Q^2 - 2 b \int PQ \right| \lesssim |b|^{2-\g}, \label{QbL2} \\
\left|E(Q_b)+ b \int PQ \right| \lesssim |b|^2. \label{QbEn}
\end{gather}
\end{enumerate}
\end{lemma}

\begin{remark}
To prove Theorem~\ref{thm:mainspace}, we will need two different values of $\g$ in the interval~\eqref{def:gamma}.
Indeed, on the one hand, to obtain~\eqref{th:ptwisem} for any $m\geq 1$, one needs to use the value $\g=1$;
on the other hand, to prove~\eqref{th:ptwise0}, we need $\g<1$ arbitrarily close to $\frac{17}{20}$
--- see for instance the proof of Lemma~\ref{le:atR} in Section~\ref{sec:space-estimates}.
At this point, one remarks from~\eqref{QbL2} that $\g=1$ creates an error term in $L^2$ of the size of
$\int Q_b^2- \int Q^2$, which is too rough to obtain~\eqref{th:ptwise0} --- see also~\eqref{minmass}.
\end{remark}

\begin{proof}[Proof of Lemma~\ref{lemma:locprofile}]
(i) First notice that~\eqref{eq:Qbj} is a direct consequence of~\eqref{eq:QbtoQ} and the exponential decay of $Q$ and its derivatives.
To prove~\eqref{eq:QbtoQ}, we first write
\[
(Q_b-Q)^{(j)}(y) = \sum_{k=1}^K b^kP_k^{(j)}(y)\chi_b(y) + \sum_{p=1}^j \sum_{k=1}^K \binom j p b^k P_k^{(j-p)}(y) \chi_b^{(p)}(y).
\]
Now recall that $P_k\in\Z{k-1}$ by construction,
which implies in particular that $|P_k^{(j)}(y)| \lesssim e^{-|y|/2}$ for all $y>0$, and also for $y<0$ in the case $k\leq j$.
When $k\geq j+1$, we just have the polynomial control $|P_k^{(j)}(y)| \lesssim (1+|y|)^{k-1-j}$.
Now, by distinguishing the three regions $y>0$, $-1<y<0$ and $-2|b|^{-\g}<y<-1$, we may estimate
\begin{align*}
|(Q_b-Q)^{(j)}(y)| &\lesssim |b|e^{-|y|/2} + \sum_{k=j+1}^K |b|^k |y|^{k-1-j} \charfunc{[-2,0]}(|b|^\g y) \\
&\quad + \sum_{p=1}^j \sum_{k=k_p}^K |b|^{k+\g p} |y|^{k+p-1-j} \charfunc{[-2,-1]}(|b|^\g y),
\end{align*}
where $k_p=\max(1,j-p+1)$.
Since $|b|^\g |y|\leq 2$ with $\g\leq1$ when $|b|^\g y \in [-2,0]$, we may estimate the terms in the first sum, where $j+1\leq k\leq K$, as
\[
|b|^k |y|^{k-1-j} \lesssim |b|^k |b|^{-\g k +(1+j)\g} = |b|^{(1-\g)k +(1+j)\g}
\lesssim |b|^{(1-\g)(1+j)+(1+j)\g} = |b|^{1+j}.
\]
Similarly, in the second sum, where $k\geq k_p\geq 1$,
\[
|b|^{k+\g p} |y|^{k+p-1-j} \lesssim |b|^{(1-\g)k+(1+j)\g} \lesssim |b|^{1+j\g}.
\]
Hence, we find as expected
\[
|(Q_b-Q)^{(j)}(y)| \lesssim |b|e^{-|y|/2} + |b|^{1+j}\charfunc{[-2,0]}(|b|^\g y) + |b|^{1+j\g} \charfunc{[-2,-1]}(|b|^\g y),
\]
which leads to~\eqref{eq:QbtoQ} since $\g\leq 1$.

Finally, to prove~\eqref{eq:dQb}, we first compute, from the definition~\eqref{def:locprofile} of~$\chi_b$,
\be \label{eq:chib}
b\,\frac{\partial\chi_b}{\partial b}(y) = \g y|b|^\g \chi'(|b|^\g y) = \g y \chi'_b(y).
\ee
Therefore, we find
\begin{align*}
\frac{\partial Q_b}{\partial b} &= \frac{\partial R_b}{\partial b}\chi_b +R_b \frac{\partial \chi_b}{\partial b}
= \sum_{k=1}^K kb^{k-1} P_k(y)\chi_b(y) + \sum_{k=1}^K \g y b^{k-1} P_k(y)\chi'_b(y) \\
&= \sum_{k=1}^K b^{k-1} R_{k-1}(y)\chi_b(y) + \sum_{k=1}^K b^{k-1} R_k(y)\chi'_b(y),
\end{align*}
with the convention $R_k\in\Z k$.
Proceeding as above for the proof of~\eqref{eq:QbtoQ}, it leads to~\eqref{eq:dQb}.

\medskip

(ii) From the definition~\eqref{def:locprofile} of $Q_b$, the expression~\eqref{def:Psib} of $\Psi_b$ may be written as
\begin{align*}
-\Psi_b &= \left[ Q''+R''_b\chi_b +2R'_b\chi'_b +R_b\chi''_b -Q-R_b\chi_b +Q_b^5 \right]' \\
&\quad +b\Lambda Q +b\Lambda R_b\chi_b +byR_b\chi'_b
-\theta(b) \left[ \frac{\partial R_b}{\partial b}\chi_b +R_b\frac{\partial\chi_b}{\partial b} \right].
\end{align*}
From the definitions~\eqref{def:Psibt} of $\Psib$ and~\eqref{def:Qbthetab} of $\Qb$,
using the expression $Q_b=\Qb\chi_b +(1-\chi_b)Q$ to expand $Q_b^5$ and also identity~\eqref{eq:chib} above,
we may rewrite the previous expression as
\[
-\Psi_b = -\Psib\chi_b + C_b +D_b,
\]
with
\[
C_b = (3R_b''-R_b+\Qb^5)\chi'_b +3R_b'\chi''_b +R_b\chi'''_b
+\left[b-\g\theta(b)b^{-1}\right] yR_b\chi'_b + \left[\Qb^5(\chi_b^5-\chi_b)\right]'
\]
and
\[
D_b= \sum_{k=1}^5 \binom 5k \left[ (1-\chi_b)^kQ^k\Qb^{5-k}\chi_b^{5-k} \right]'
-5(1-\chi_b)Q'Q^4 +b(1-\chi_b)\Lambda Q.
\]
Hence, we have to estimate the four terms
\[
-\Psi_b^{(j)} = -\Psib^{(j)}\chi_b + \sum_{p=1}^j \binom jp (-\Psib)^{(j-p)}\chi_b^{(p)} + C_b^{(j)} + D_b^{(j)}
= A_j+B_j+C_j+D_j.
\]

\emph{Estimate of $A_j$.}
From the expression~\eqref{eq:Psibt} of $\Psib$, we get
\[
-\Psib^{(j)}\chi_b = b^{K+1} \left[ \sum_{k=1}^{K-1} b^{k-1}R_{K+k-2}^{(j)}
+ \sum_{k=K}^{4K} b^{k-1} R_{K+k-6}^{(j)} \right]\chi_b,
\]
with the convention $R_\ell\in\Z\ell$ for $\ell\in\ZZ$.
Therefore, from the definition of $\Z\ell$, we are led to set $k_j = \max(1,j+2-K)$ and $k'_j = \max(K,j+6-K)$.
Considering the cases $y>0$, $-1<y<0$ and $-2|b|^{-\g}<y<-1$ as in the (i) of the proof, we find the estimate
\[
|A_j(y)| \lesssim |b|^{K+1}e^{-|y|/2} + |b|^{K+1} |y|^{K-1-j}
\! \left( \sum_{k=k_j}^{K-1} |b|^{k-1}|y|^{k-1} +\sum_{k=k'_j}^{4K} |b|^{k-1}|y|^{k-5}\right) \! \charfunc{[-2|b|^{-\g},-1]}(y).
\]
For the first sum, since $k\geq k_j\geq 1$ and $1\leq |y|\leq 2|b|^{-\g}$, we find
\[
|b|^{k-1}|y|^{k-1} = (|b||y|)^{k-1} \leq (|b|^\g |y|)^{k-1} \lesssim 1.
\]
Similarly, for the second sum, since $k\geq k'_j\geq K\geq 1$, we find
\[
|b|^{k-1}|y|^{k-5} = (|b||y|)^{k-1} |y|^{-4} \lesssim |y|^{-4} \lesssim 1.
\]
Finally, we found
\[
|A_j(y)| \lesssim |b|^{K+1}e^{-|y|/2} + |b|^{K+1}|y|^{K-1-j} \charfunc{[-2,0]}(|b|^\g y)
\]
and \emph{a fortiori}
\[
|A_j(y)| \lesssim |b|^{K+1}e^{-|y|/2} + |b|^{K+1}|y|^{K-1-j}\charfunc{[-1,0]}(|b|^\g y) + |b|^{1+(j+1)\g} \charfunc{[-2,-1]}(|b|^\g y).
\]

\emph{Estimate of $B_j$.}
Similarly to $A_j$, we are led to consider $k_{j,p}=\max(1,j+2-p-K)$ and $k'_{j,p} = \max(K,j+6-p-K)$ for every $1\leq p\leq j$,
and we find
\begin{multline*}
|B_j(y)| \lesssim |b|^{K+1}e^{-|y|/2} \\
+ \sum_{p=1}^j (|b|^{\g}|y|)^p |b|^{K+1} |y|^{K-1-j}
\left( \sum_{k=k_{j,p}}^{K-1} |b|^{k-1}|y|^{k-1} + \sum_{k=k'_{j,p}}^{4K} |b|^{k-1}|y|^{k-5} \right) \charfunc{[-2,-1]}(|b|^\g y).
\end{multline*}
As in the estimate for $A_j$, we obtain
\begin{align*}
|B_j(y)| &\lesssim |b|^{K+1}e^{-|y|/2} + |b|^{K+1}|y|^{K-1-j} \charfunc{[-2,-1]}(|b|^\g y) \\
&\lesssim |b|^{K+1}e^{-|y|/2} + |b|^{1+(j+1)\g} \charfunc{[-2,-1]}(|b|^\g y).
\end{align*}

\emph{Estimate of $C_j$.}
First notice from the definition~\eqref{def:Qbthetab} of $\theta(b)$ that $\theta(b)b^{-1}$ is not singular, and more precisely
\[
\left|b-\g\theta(b)b^{-1} \right| \lesssim |b|.
\]
Then, using again the convention $R_\ell\in\Z\ell$, we may rewrite $C_b$ as
\begin{align*}
C_b &= \left( \sum_{k=1}^K b^k R_{k-1} + \sum_{k=0}^{5K} b^k R_{k-5} \right) \chi'_b + \Big(b-\g\theta(b)b^{-1}\Big)\sum_{k=1}^K b^kR_k\chi'_b \\
&\quad +\sum_{k=1}^K b^k R_{k-2}\chi''_b +\sum_{k=1}^K b^k R_{k-1}\chi'''_b + \sum_{k=0}^{5K} b^k[R_{k-5}(\chi_b^5-\chi_b)]',
\end{align*}
and so
\begin{align*}
C_b^{(j)} &= \sum_{p=0}^j \sum_{k=1}^K \binom jp b^k \left[ R_{k-1}^{(j-p)}\chi_b^{(p+1)} +\left(b-\g\theta(b)b^{-1}\right)R_k^{(j-p)}\chi_b^{(p+1)} \right. \\
&\hspace{6cm} \left. {}+ R_{k-2}^{(j-p)}\chi_b^{(p+2)} +R_{k-1}^{(j-p)} \chi_b^{(p+3)} \right] \\
&\quad + \sum_{p=0}^j \sum_{k=0}^{5K} \binom jp b^k R_{k-5}^{(j-p)}\chi_b^{(p+1)}
+ \sum_{p=0}^{j+1} \sum_{k=0}^{5K} \binom{j+1}p b^k R_{k-5}^{(j+1-p)}(\chi_b^5-\chi_b)^{(p)}.
\end{align*}
From the definition of $\chi_b$, we now notice that, for all $p\geq 0$,
\begin{align*}
|(\chi_b^5-\chi_b)^{(p)}(y)| &\lesssim |b|^{p\g}\charfunc{[-2,-1]}(|b|^\g y), \\
|\chi_b^{(p+1)}(y)| &\lesssim |b|^{(p+1)\g}\charfunc{[-2,-1]}(|b|^\g y).
\end{align*}
Then, following the above estimates of $A_j$ and $B_j$, we denote $k_{1,p}=\max(1,j+1-p)$ and $k_{6,p}=\max(6,j+6-p)$,
so that we may write
\begin{multline*}
|C_j(y)| \lesssim \sum_{p=0}^j \sum_{k=k_{1,p}}^K |b|^k |y|^{k-1-j+p} |b|^{(p+1)\g} \charfunc{[-2,-1]}(|b|^\g y) \\
+\sum_{p=0}^{j+1} \sum_{k=k_{6,p}}^{5K} |b|^k |y|^{k-6-j+p}|b|^{p\g} \charfunc{[-2,-1]}(|b|^\g y) + e^{-|y|/2} \charfunc{[-2,-1]}(|b|^\g y).
\end{multline*}
For the first sum we estimate, since $k\geq k_{1,p}\geq 1$,
\[
|b|^k |y|^{k-1-j+p} |b|^{(p+1)\g} \lesssim |b|^{(1-\g)k+(j+1)\g+\g} \lesssim |b|^{1+(j+1)\g}.
\]
As before, we also find the same estimate for the second sum.
Finally, when $|b|^\g y\in [-2,-1]$, we have
\[
e^{-|y|/2} \leq e^{-\frac12 |b|^{-\g}} \lesssim |b|^{1+(j+1)\g}.
\]
Hence we found
\[
|C_j(y)| \lesssim |b|^{1+(j+1)\g} \charfunc{[-2,-1]}(|b|^\g y).
\]

\emph{Estimate of $D_j$.}
First recall that $\Qb\chi_b = Q_b -(1-\chi_b)Q$, so~\eqref{eq:Qbj} gives that $(\Qb\chi_b)^{5-k}$
and all its derivatives are bounded for $1\leq k\leq 5$. Thus, we may rewrite $D_b^{(j)}$ as
\[
D_b^{(j)} = (1-\chi_b)D_{0,b} + \sum_{p=1}^{j+1} D_{p,b} \chi_b^{(p)},
\]
where $D_{p,b}$ satisfy $|D_{p,b}(y)| \lesssim e^{-\frac34|y|}$ for all $0\leq p\leq j+1$.
To estimate the terms of the sum, we proceed as above and find, for all $1\leq p\leq j+1$,
\[
|D_{p,b}(y) \chi_b^{(p)}(y)| \lesssim e^{-\frac34|y|} \charfunc{[-2,-1]}(|b|^\g y)
\lesssim e^{-\frac34|b|^{-\g}} \charfunc{[-2,-1]}(|b|^\g y)
\lesssim |b|^{1+(j+1)\g} \charfunc{[-2,-1]}(|b|^\g y).
\]
To estimate the first term, we rely on the fact that $\chi\equiv 1$ on $[-1,+\infty)$, and so
\[
|(1-\chi_b(y))D_{0,b}(y)| \lesssim e^{-\frac34|y|} \charfunc{(-\infty,-1]}(|b|^\g y)
\lesssim e^{-\frac14|b|^{-\g}} e^{-|y|/2} \lesssim |b|^{K+1} e^{-|y|/2}.
\]
Finally, we found
\[
|D_j(y)| \lesssim |b|^{K+1} e^{-|y|/2} + |b|^{1+(j+1)\g} \charfunc{[-2,-1]}(|b|^\g y).
\]

\medskip

(iii) To prove~\eqref{QbL2}, we first compute
\[
\int Q_b^2 = \int \left( Q+ R_b\chi_b \right)^2 =\int Q^2 + 2 \int Q R_b\chi_b + \int (R_b\chi_b)^2.
\]
But, from one hand, we have
\[
2 \int Q R_b\chi_b = 2 b \int QP - 2 b \int Q P(1-\chi_b) + 2 \sum_{k=2}^K b^k \int Q P_k\chi_b,
\]
with
\[
\left|\int QP(1-\chi_b)\right|\lesssim\int_{y<-|b|^{-\g}} Q \lesssim |b|^{10}
\quad\m{and}\quad
\left| \sum_{k=2}^K b^k \int Q P_k\chi_b \right| \lesssim |b|^2.
\]
On the other hand, from~\eqref{eq:QbtoQ}, we find
\[
\int (R_b\chi_b)^2 = \int (Q_b-Q)^2 \lesssim |b|^{2-\g}.
\]
Gathering the above estimates, we obtain~\eqref{QbL2}.

To prove~\eqref{QbEn}, we write similarly
\begin{align*}
E(Q_b) 
&= E(Q) - \int (Q''+Q^5) R_b\chi_b + \frac12 \int \left(\partial_y \left(R_b\chi_b\right)\right)^2 \\
& \quad - \frac16 \int \left[(Q+R_b\chi_b)^6 -Q^6 -6 Q^5 \left(R_b\chi_b \right) \right]
= - b \int PQ + O(|b|^2),
\end{align*}
since $E(Q)=0$,
\[
\int (Q''+Q^5) R_b\chi_b = \int Q R_b\chi_b = b \int PQ + O(|b|^2),
\]
and, again from~\eqref{eq:QbtoQ},
\begin{gather*}
\int \left(\partial_y\left(R_b\chi_b\right)\right)^2 = \int \left( \partial_y(Q_b-Q) \right)^2 \lesssim |b|^2, \\
\int \left| (Q+R_b\chi_b)^6 -Q^6 -6 Q^5 \left(R_b\chi_b \right) \right| \lesssim \int (Q_b-Q)^6 + Q^4(Q_b-Q)^2 \lesssim |b|^2.
\end{gather*}
This concludes the proof of the lemma.
\end{proof}

\subsection{Modulation around the profile}

Let $u\in\mathcal{C}(I,H^1)$ be a solution of~\eqref{kdv} on a time interval~$I$ (with $0\in I$),
close in $H^1$ to the family of solitons,
\emph{i.e.}~assume that there exist $(\lambda_1(t),x_1(t))\in (0,+\infty)\times \RR$ and $\e_1(t)$
such that, for all $t\in I$,
\be \label{hypeprochien}
u(t,x)=\frac{1}{\lambda^{1/2}_1(t)}[Q+\e_1(t)]\left(\frac{x-x_1(t)}{\lambda_1(t)}\right),
\quad \|\e_1(t)\|_{H^1}\leq \delta_0,
\ee
for some small enough universal constant $\delta_0>0$.
As in~\cite{MMR1,MMR2}, this decomposition is refined using the $Q_b$ profiles and a standard modulation argument.

\begin{lemma}[Decomposition around the refined profiles] \label{lemma:decomposition}
Let $K\geq1$. Then there exists $\delta_0>0$ such that, assuming~\eqref{hypeprochien}, the following holds.
\begin{enumerate}[label=\emph{(\roman*)}]
\item \emph{Decomposition:} There exist unique $\mathcal{C}^1$ functions $(\lambda, x,b) : I\to (0,+\infty)\times \RR^2$ such that
\be \label{defofeps}
\e(t,y)=\lambda^{1/2}(t) u(t,\lambda(t) y + x(t)) - Q_{b(t)}(y),
\ee
defined for all $t\in I$, satisfies the orthogonality conditions
\be \label{ortho1}
(\e(t), y \Lambda Q ) = (\e(t),\Lambda Q) = (\e(t),Q) = 0
\ee
and
\be \label{controle}
|b(t)| + \|\e(t)\|_{H^1} \lesssim \delta_0.
\ee
\item \emph{Equation of $\e$:}
Let
\be \label{rescaledtime}
s = s(t)= \int_0^t \frac{dt'}{\lambda^3(t')}\quad \m{and}\quad y=\frac{x-x(t)}{\lambda(t)}.
\ee
Let $J=s(I)$. Then, on $J$, there holds
\begin{align}
\e_s + \left(\e_{yy}-\e + \left( Q_b+\e\right)^5 - Q_b^5 \right)_y
&= \frac{\lambda_s}{\lambda} \Lambda \e + \left(\frac{\lambda_s}{\lambda}+b\right){\Lambda} Q_b
+ \left(\frac{x_s}{\lambda} -1\right) (Q_b + \e)_y \nonumber \\
&\quad - (b_s +\theta(b)) \frac{\partial Q_b}{\partial b} + \Psi_{b}, \label{eqofeps}
\end{align}
where $\Psi_b$ is defined in~\eqref{def:Psib}.
\item \emph{Estimate induced by the conservation of the mass:}
for all $s\in J$, there holds
\be \label{twobound}
\|\e(s)\|^2_{L^2} \lesssim |b(s)|+\left|\int u_0^2-\int Q^2\right|.
\ee
\item \emph{Modulation equations:} on $J$, there hold
\begin{gather}
\left|\frac{\lambda_s}{\lambda} + b\right| + \left| \frac{x_s}{\lambda} - 1 \right| \lesssim
\left(\int \e^2 {e^{-|y|/2}} \right)^{\frac12} + |b|^{K+1}, \label{eq:2002} \\
|b_s+\theta(b)| \lesssim \int \e^2 {e^{-|y|/2}} +
|b|\left( \int \e^2 {e^{-|y|/2}} \right)^{\frac12}+|b|^{K+1}. \label{eq:2003}
\end{gather}
\item \emph{Minimal mass:}
if in addition $\|u_0\|_{L^2}=\|Q\|_{L^2}$, then $E(u_0)\geq 0$ and, on $J$,
\begin{gather}
\left| \int \e^2 + 2 b \int PQ \right|\lesssim |b|^{\frac12 (3-\g)}, \label{minmass} \\
\left| \int \e_y^2 - 2 \lambda^2 E(u_0) - 2b \int PQ \right| \lesssim |b|^2 +\int \e^2 e^{-|y|/2}. \label{energie}
\end{gather}
\end{enumerate}
\end{lemma}

\begin{proof}
(i) Let $\mathcal{M}=H^1\times (0,+\infty)\times \RR^2$.
For $M=(u,\lambda,x,b)\in\mathcal{M}$, we define the function $\e_M(y)=\lambda^{1/2}u(\lambda y+x)-Q_b(y)$
and consider the functional $\Phi: \mathcal{M} \to \RR^3$ defined by
\[
\Phi(M) = ((\e_M,y\Lambda Q),(\e_M,\Lambda Q),(\e_M,Q)).
\]
Now let $M_0=(Q,1,0,0)$, and note that $\Phi(M_0)=0$ since $\e_{M_0}\equiv 0$.
Moreover, the Jacobian of $\Phi$ in the direction of $(\lambda,x,b)$ at the point $M_0$ reads as
\begin{align*}
-
\begin{vmatrix}
(\Lambda Q,y\Lambda Q)&(Q',y\Lambda Q)&(P,y\Lambda Q) \\
(\Lambda Q,\Lambda Q)&(Q',\Lambda Q)&(P,\Lambda Q) \\
(\Lambda Q,Q)&(Q',Q)&(P,Q)
\end{vmatrix}
&= -
\begin{vmatrix}
0&\frac12\|yQ\|_{L^2}^2&(P,y\Lambda Q) \\
\|\Lambda Q\|_{L^2}^2&0&(P,\Lambda Q) \\
0&0&\frac{1}{16}\|Q\|_{L^1}^2
\end{vmatrix} \\
&= \frac{1}{32} \|Q\|_{L^1}^2 \|\Lambda Q\|_{L^2}^2 \|yQ\|_{L^2}^2 \neq 0.
\end{align*}
The implicit function theorem then gives the existence of a neighbourhood of $Q$ in $H^1$,
denoted $V_Q = \{ u\in H^1\ |\ \|u-Q\|_{H^1}\leq\bar\delta \}$ for some $\bar\delta>0$ (independent of $K$),
a neighbourhood of $(1,0,0)$ in~$\RR^3$, denoted $V_{(1,0,0)}$,
and a unique $\mathcal{C}^1$ map $(\lambda,x,b) : V_Q\to V_{(1,0,0)}$ such that $\e_u$,
defined by $\e_u(y) = \lambda^{1/2}(u)u[\lambda(u)y+x(u)] -Q_{b(u)}(y)$,
satisfies~\eqref{ortho1} and
\[
|\lambda(u)-1| +|x(u)| + |b(u)| + \|\e_u\|_{H^1} \lesssim \|u-Q\|_{H^1}.
\]
Assuming now $\delta_0\leq\bar\delta$, we may apply from~\eqref{hypeprochien} the above result to $u_0(t)$ for all $t\in I$,
where $u_0$ is defined by $u_0(t,y) = \lambda_1^{1/2}(t)u(t,\lambda_1(t)y+x_1(t))$.
Denoting $\lambda_0(t)=\lambda(u_0(t))$, $x_0(t)=x(u_0(t))$, $b_0(t)=b(u_0(t))$ and $\e_0(t)=\e_{u_0(t)}$,
and finally
\[
\lambda(t)=\lambda_0(t)\lambda_1(t),\quad x(t)=\lambda_1(t)x_0(t)+x_1(t),\quad b(t)=b_0(t),\quad \e(t)=\e_0(t),
\]
we obtain~\eqref{defofeps},~\eqref{ortho1} and~\eqref{controle} as expected.
Assuming $\delta_0$ even smaller,
we notice that such a decomposition is unique under the condition~\eqref{controle}.
Note that the smallness of $\delta_0$ to ensure the existence and uniqueness of the decomposition
is independent of the value of $K$.

\medskip

(ii) The equation of $\e$ comes from a straightforward calculation,
by plugging the expression of $u(t,x)$ given by~\eqref{defofeps} in~\eqref{kdv},
with the changes of variables~\eqref{rescaledtime}.

\medskip

(iii) We first use the conservation of the mass to get
\[
\int (Q_b+\e)^2 = \int Q_b^2 + 2\int \e Q_b + \int \e^2 = \int u_0^2
\]
and thus, from~\eqref{eq:QbtoQ},~\eqref{QbL2} and the orthogonality $(\e,Q)=0$, we obtain
\begin{align*}
\|\e\|_{L^2}^2
&\leq \left| \int u_0^2 -\int Q_b^2 \right| + 2 \left| \int \e(Q_b-Q) \right| \\
&\leq \left| \int u_0^2 -\int Q^2 \right| + \left| \int Q_b^2 -\int Q^2\right| +2\|\e\|_{L^2}\|Q_b-Q\|_{L^2} \\
&\leq \left| \int u_0^2 -\int Q^2 \right| + C|b| +C|b|^{2-\g} +\frac12 \|\e\|_{L^2}^2
\end{align*}
with $C>0$, which leads to~\eqref{twobound} as expected.

\medskip

(iv) First note that, from~\eqref{eq:Psibj}, we have
\[
\int |\Psi_b(y)| e^{-|y|/2}\, dy \lesssim |b|^{K+1}.
\]
Thus, differentiating the orthogonality conditions $(\e,\Lambda Q) = (\e,y \Lambda Q)= 0$
and using the equation of~$\e$, together with the cancellations $(Q',\Lambda Q)=(y\Lambda Q,\Lambda Q)=0$
and the non degeneracies $(\Lambda Q,\Lambda Q)\neq 0$ and $(Q',y\Lambda Q)\neq 0$
(as already stated in the expression of the Jacobian above), we find
\begin{align*}
\left| \frac{\lambda_s}{\lambda} + b\right| + \left| \frac{x_s}{\lambda} - 1 \right|
&\lesssim \left(\int \e^2 e^{-|y|/2}\right)^{\frac12}
+\left(\left| \frac{\lambda_s}{\lambda} + b\right| + \left| \frac{x_s}{\lambda} - 1 \right|\right)
\left[ \left(\int \e^2 e^{-|y|/2}\right)^{\frac12} +|b| \right] \\
&\quad + |b_s + \theta(b)| + |b|^{K+1}.
\end{align*}
Similarly, differentiating $(\e,Q)=0$, using $(P,Q)\neq 0$, $LQ'=0$, $(Q,\Lambda Q)=0$ and $(\e,\Lambda Q)=0$, we find
\begin{align*}
|b_s + \theta(b)| &\lesssim \int \e^2 e^{-|y|/2}
+ \left(\left| \frac{\lambda_s}{\lambda} + b\right| + \left| \frac{x_s}{\lambda} - 1 \right|\right)
\left[\left(\int \e^2 e^{-|y|/2}\right)^{\frac12} + |b|\right] \\
&\quad +|b||b_s+\theta(b)| + |b|\left(\int \e^2 e^{-|y|/2}\right)^{\frac12} + |b|^{K+1}.
\end{align*}
Combining these estimates with~\eqref{controle} and assuming $\delta_0$ small enough, we obtain~\eqref{eq:2002} and~\eqref{eq:2003}.

\medskip

(v) First note that, in the minimal mass case $\|u_0\|_{L^2}=\|Q\|_{L^2}$,
estimate~\eqref{twobound} above reduces simply to $\|\e\|_{L^2}^2 \lesssim |b|$.

To prove the more precise estimate~\eqref{minmass}, we start as in the proof of~(iii)
and use again the minimal mass assumption to get
\[
\int (Q_b+\e)^2 = \int Q_b^2 + 2\int \e Q_b + \int \e^2 = \int Q^2
\]
and thus, from~\eqref{QbL2} and the orthogonality $(\e,Q)=0$,
\[
2 b\int PQ + 2 \int \e (Q_b-Q) + \int \e^2 = O(|b|^{2-\g}).
\]
From $\|\e\|_{L^2}^2 \lesssim |b|$ and~\eqref{eq:QbtoQ}, we also get
\[
\left|\int \e (Q_b-Q)\right| \leq \|\e\|_{L^2} \|Q_b-Q\|_{L^2}
\lesssim |b|^{\frac12} |b|^{1-\frac\g2} = |b|^{\frac{3-\g}2},
\]
and~\eqref{minmass} follows.

To prove~\eqref{energie}, we first compute, using~\eqref{QbEn} and again the orthogonality $(\e,Q)=0$,
\begin{align*}
2\lambda^2 E(u_0) = 2E(Q_b+\e)
&= 2 E(Q_b) - 2 \int \e (Q_b)_{yy} + \int \e_y^2 - \frac13 \int \left[(Q_b+\e)^6 - Q_b^6\right] \\
&= - 2b \int PQ +\int \e_y^2 - 2 \int \e \left[ (Q_b-Q)_{yy} + (Q_b^5-Q^5)\right] \\
&\quad - \frac13 \int \left[(Q_b+\e)^6-Q_b^6-6Q_b^5 \e\right] + O(|b|^2).
\end{align*}
By~\eqref{eq:QbtoQ} and $\int_{-2 |b|^{-\g} <y<0} |\e|\lesssim \|\e\|_{L^2} |b|^{-\frac\g 2} \lesssim |b|^{\frac12(1-\g)}$, we get
\[
\left| \int \e (Q_b-Q)_{yy}\right|
\lesssim |b| \left(\int \e^2 e^{-|y|/2} \right)^{\frac12} + |b|^{1+2\g} \int_{-2 |b|^{-\g} <y<0} |\e|
\lesssim |b| \left(\int \e^2 e^{-|y|/2} \right)^{\frac12} + |b|^{\frac32 (1+\g)}.
\]
Next, by expanding $Q_b^5$ as $[Q+(Q_b-Q)]^5$, using the exponential decay of $Q$ and~\eqref{eq:QbtoQ}, we find, for all $y\in\RR$,
\[
|(Q_b^5-Q^5)(y)| \lesssim |b|e^{-|y|} + |(Q_b-Q)^5(y)| \lesssim |b|e^{-|y|} + |b|^5 \charfunc{[-2,0]}(|b|^\g y).
\]
So, by similar computations,
\[
\left| \int \e (Q_b^5-Q^5) \right|
\lesssim |b| \left(\int \e^2 e^{-|y|/2} \right)^{\frac12} + |b|^5 \int_{-2 |b|^{-\g} <y<0} |\e|
\lesssim |b| \left(\int \e^2 e^{-|y|/2} \right)^{\frac12} + |b|^{5+\frac12(1-\g)}.
\]
Finally, from~\eqref{eq:Qbj},~\eqref{controle} and $\|\e\|_{L^2}^2 \lesssim |b|$, we obtain
\begin{align*}
\int \left|(Q_b+\e)^6 - Q_b^6-6Q^5_b \e\right|
&\lesssim \int \e^6 + \int Q_b^4\e^2 \lesssim \|\e\|_{L^2}^4 \|\e_y\|_{L^2}^2 + \int \e^2 e^{-|y|/2} + |b|^4 \int \e^2 \\
&\lesssim |b|^2 + \int \e^2 e^{-|y|/2}.
\end{align*}
Gathering these estimates, we obtain~\eqref{energie} provided that $\g\geq\frac13$,
which is granted by the choice~\eqref{def:gamma} of $\g$.
\end{proof}

\subsection{Weak $H^1$ stability of the decomposition}

The decomposition of Lemma~\ref{lemma:decomposition} is stable by weak $H^1$ limit.
We refer to Lemma~17 and Appendix~D of~\cite{MMjmpa} for a similar statement.

\begin{lemma}[Weak $H^1$ stability and convergence of the parameters~\cite{MMjmpa}] \label{le:weak}
Let $(u_{0,n})$ be a sequence of $H^1$ initial data such that
\[
u_{0,n}\rightharpoonup u_0\quad \m{in } H^1\ \m{as }\ n\to +\infty.
\]
Assume that, for some $C_1, T_1>0$, for all $n$ large, the corresponding solution $u_n(t)$ of~\eqref{kdv}
exists on $[-T_1,0]$ and satisfies $\max_{t\in [-T_1,0]}\|u_n(t)\|_{H^1}\leq C_1$.
Let $u(t)$ be the solution of~\eqref{kdv} corresponding to $u_0$.
Then, $u(t)$ exists on $[-T_1,0]$ and
\be \label{we:1}
\forall t\in [-T_1,0], \quad u_n(t) \rightharpoonup u(t)\quad \m{in } H^1\ \m{as }\ n\to +\infty.
\ee
Assume further that $u_n(t)$ satisfies~\eqref{hypeprochien} on $[-T_1,0]$ and that the parameters of the decomposition
$(\lambda_n,x_n,b_n,\e_n)$ of~$u_n$ given by Lemma~\ref{lemma:decomposition} satisfy, for all $n$ large,
\be \label{hypoweak}
\forall t\in [-T_1,0],\quad 0<c\leq \lambda_n(t)<C ,\quad \lambda_n(0)=1,\quad x_n(0)=0.
\ee
Then, $u(t)$ satisfies~\eqref{hypeprochien} on $[-T_1,0]$ and its decomposition $(\lambda,x,b,\e)$ satisfies, as $n\to +\infty$,
\[
\forall t\in [-T_1,0],\quad \e_n(t)\rightharpoonup \e(t)\quad \m{in } H^1,\quad
\lambda_n(t)\to \lambda(t),\quad x_n(t)\to x(t), \quad b_n(t)\to b(t).
\]
\end{lemma}

\begin{proof}[Sketch of proof]
The first part of the lemma, \emph{i.e.}~the weak convergence of $u_n(t)$ to $u(t)$, is the object of Lemma~30 in~\cite{MMjmpa}.
This result is stated with a special sequence of initial data $u_n(t,x)=v(t+t_n,x)$,
where $v$ is a (global) solution of ~\eqref{kdv} and $t_n\to +\infty$,
but the same proof applies for any weakly converging sequence.
Note that this part of the lemma is not related to closeness to soliton nor to any particular choice of decomposition close to the soliton.
For a different proof, we also refer the reader to Theorem~5 in~\cite{KMa} (in the case of the Benjamin--Ono equation)
and to Lemma~3.4 in~\cite{C} for the mass supercritical (gKdV) equation.

The second part of the lemma depends on the decomposition, but the general scheme of the proof of Lemma~17 given in~\cite{MMjmpa} applies.
We also refer to~\cite{MRinvent}, page 599, for a more detailed argument.
The first step of the proof is to note that estimates~\eqref{eq:2002}--\eqref{eq:2003}
provide uniform bounds on the time derivatives of the geometric parameters $(\lambda_n(t),x_n(t),b_n(t))$ on $[-T_1,0]$.
Therefore, by Ascoli's theorem, up to the extraction of a subsequence,
\[
(\lambda_n(t),x_n(t),b_n(t))\to (\wt\lambda(t),\wt x(t),\wt b(t)) \quad \m{on } [-T_1,0],
\]
for some functions $(\wt\lambda(t),\wt x(t),\wt b(t))$.
Writing the orthogonality conditions~\eqref{ortho1} in terms of $u_n(t)$ and $(\lambda_n(t),x_n(t),b_n(t))$,
using~\eqref{we:1} and passing to the limit as $n\to +\infty$,
we see that $u(t)$ and the limiting parameters $(\wt\lambda(t),\wt x(t),\wt b(t))$
satisfy the same orthogonality relations.
In particular, they correspond to the unique parameters $(\lambda(t),x(t),b(t))$ given by Lemma~\ref{lemma:decomposition}.
This uniqueness statement proves by a standard argument that, for the whole sequence,
$(\lambda_n(t),x_n(t),b_n(t))$ converges to $(\lambda(t),x(t),b(t))$ on $[-T_1,0]$ as $n\to +\infty$.
It follows in particular that $\e_n(t)\rightharpoonup \e(t)$ in $H^1$ as $n\to +\infty$.
\end{proof}

\section{Time estimates} \label{sec:time-estimates}

In this section, we prove refined estimates on minimal mass blow up solutions close to the blow up time.

\begin{proposition} \label{prop:SHm}
Let $S$ be a solution of~\eqref{kdv} which blows up as $t\downarrow 0$ and such that $\|S(t)\|_{L^2}=\|Q\|_{L^2}$.
Then there exist $0<T_0<1$, $\epsilon = \pm1$, $\ell_0>0$ and $x_0\in \RR$
such that $\epsilon S(t)$ satisfies~\eqref{hypeprochien} on $(0,T_0]$
and such that the following holds.

Let $K>20$ and $m_K = [K/2]-1$.
Then the decomposition $(\lambda,x,b,\e)$ of $\epsilon S(t)$ given by Lemma~\ref{lemma:decomposition},
\be \label{eq:decomposition}
\epsilon S(t,x) = \frac{1}{\lambda^{1/2}(t)}\left[Q_{b(t)}+\e(t)\right]\left(\frac{x-x(t)}{\lambda(t)}\right),
\ee
satisfies the following estimates:
for all $0<t\leq T_0$, for all $0\leq m\leq m_K$,
\begin{gather}
|\lambda(t) - \ell_0 t|\lesssim t^3,\quad
|x(t) + \ell_0^{-2} t^{-1} - x_0 |\lesssim t,\quad
|b(t) + \ell_0^3 t^2 |\lesssim t^4, \label{Spar} \\
\|\e(t)\|_{L^2} \lesssim t,\quad
\|\e(t)\|_{\dot H^m} \lesssim t^{\frac12(1+\g) + 2\g m},\quad
\|\e(t)\|_{\dot H^m_B} \lesssim t^{2K-m}, \label{SHm}
\end{gather}
for some constant $B>1$.
Moreover, for all $0<t\leq T_0$,
\be \label{Sloc}
\int_{y>-t^{-\frac 85}} \e^2(t,y)\, dy \lesssim t^7.
\ee
\end{proposition}

Recall that minimal mass solutions are constructed in~\cite{MMR2} using a compactness argument on a sequence of suitable global solutions of~\eqref{kdv}.
The proof of~\eqref{Spar}--\eqref{Sloc} below also relies on the use of a sequence of solutions $(u_n(t))$ of~\eqref{kdv} approximating $S(t)$
and provides an alternative construction of $S(t)$, as discussed in Remark~3.2 of~\cite{MMR2}.
Note also that~\eqref{Spar} and the estimate on the $L^2$ norm of $\e(t)$ in~\eqref{SHm} are already contained in~\cite{MMR2},
but estimates in~\eqref{SHm} are new for $m\geq 1$ (compare~\eqref{SHm} with $m=1$ and $\frac{17}{20} <\g \leq 1$
and estimate~(4.6) in Proposition~4.1 of~\cite{MMR2}).
Such refined estimates in homogeneous Sobolev norms of $\e$ are crucial in the proof of decay estimates in space in Section~\ref{sec:space-estimates}.

\subsection{The bootstrap setting}

Let $K>20$ and $m_K = [K/2]-1$.
Let $T_n = \frac1{\sqrt{n}}$ for $n>1$ large.
Let $u_n(t)$ be the solution of~\eqref{kdv} corresponding to the following initial data at $t=T_n$:
\[
u_n(T_n,x)=\frac1{\lambda_n^{1/2}(T_n)} Q_{b_n(T_n)} \left(\frac{x-x_n(T_n)}{\lambda_n(T_n)}\right),
\]
where
\be \label{inittrois}
\lambda_n(T_n)=\left(n-\frac{\beta_3}2 \log\left(\frac n2\right)\right)^{-\frac12},
\quad b_n(T_n)=-\lambda_n^2(T_n) +\frac{\beta_3}{2}\lambda_n^4(T_n),\quad x_n(T_n)=-\sqrt{n}.
\ee
For $t\geq T_n$, as long as the solution $u_n(t)$ exists and satisfies~\eqref{hypeprochien},
we consider its decomposition $(\lambda_n,x_n,b_n,\e_n)$ from Lemma~\ref{lemma:decomposition}.
At $T_n$, this decomposition satisfies~\eqref{inittrois} and $\e_n(T_n)\equiv 0$.

Set
\be \label{eq:time}
s=s(t)=S_n + \int_{T_n}^t \frac{dt'}{\lambda_n^3(t')},\quad S_n= s(T_n):= -\frac n2.
\ee
We consider all time dependent functions indifferently as functions of $t$ or $s$.
In particular, in the rescaled time variable $s$,~\eqref{inittrois} rewrites as
\be \label{initdeux}
\lambda_n(S_n) = \left(2|S_n| - \frac{\beta_3}2 \log|S_n|\right)^{-\frac12},
\quad b_n(S_n) = -\lambda_n^2(S_n)+ \frac{\beta_3}2 \lambda_n^4(S_n),
\quad x_n(S_n) = -\sqrt{2|S_n|}.
\ee
Note in particular that
\be \label{initquatre}
\left| b_n(S_n) - \frac{\beta_3}2 b_n^2(S_n) + \lambda_n^2(S_n)\right| =
\left| \frac{\beta_3}2 \right| \left| b_n^2(S_n) - \lambda_n^4(S_n)\right| \lesssim |S_n|^{-3}.
\ee
Let $B>1$ and $C^*>1$ to be chosen large enough.
We work with the following bootstrap estimates, for $0\leq m\leq m_K$:
\be \label{eq:BS}
\left.
\begin{aligned}
\left|\lambda_n(s)-\frac1{\sqrt{2|s|}}\right|\leq |s|^{-1},\quad
\left|b_n(s)+\frac1{2|s|}\right|\leq |s|^{-\frac32},\quad
\left|x_n(s) + \sqrt{2|s|}\right|\leq 1 & \\
\|\e_n(s)\|_{\dot H^m}\leq C^* |s|^{-\frac14 (1+\g) - \g m},\quad
\|\e_n(s)\|_{\dot H^m_B} \leq (C^*)^{\frac{m}2} |s|^{-K+\frac m2} &
\end{aligned}
\right\}
\ee

Note that, from~\eqref{initdeux} and the continuity of $t\mapsto u_n(t)$ in $H^1$,
there exists $\tau_n>0$ such that $u_n$ exists and satisfies~\eqref{hypeprochien} on $[S_n,S_n+\tau_n]$.
Moreover, by possibly taking a smaller $\tau_n>0$,~\eqref{eq:BS} is satisfied on $[S_n,S_n+\tau_n]$.
Thus, for $S_0<-1$ to be chosen later, we may set
\[
S^*_n = \sup\{S_n<\tilde{s}<S_0 \m{ such that~\eqref{eq:BS} is satisfied for all } s\in [S_n,\tilde{s}]\}.
\]

Note for future reference that, as a direct consequence of~\eqref{eq:BS},~\eqref{QbL2},~\eqref{twobound} and~\eqref{inittrois},
we have, for $|S_0|>1$ large enough (possibly depending on $C^*$), for all $s\in [S_n,S_n^*]$,
\be \label{eq:BS2}
\|\e_n(s)\|_{L^2} \lesssim |s|^{-\frac12},\quad
0<\lambda_n(s)\lesssim |s|^{-\frac12},\quad
-|s|^{-1} \lesssim b_n(s) <0,\quad -|s|^{\frac12} \lesssim x_n(s)<0.
\ee

\begin{proposition} \label{proposition:uniforme}
There exist $C^*>1$, $B>1$ and $S_0<-1$, independent of $n$, such that, for all~$n$ large enough, $S^*_n=S_0$.
\end{proposition}

In Sections~\ref{sec:parameters} to~\ref{sec:closingbootstrap}, we prove Proposition~\ref{proposition:uniforme}
and, in Section~\ref{sec:intermediate}, we prove additional $L^2$~estimates for $\e_n$ in intermediate region.
From these results, we deduce Proposition~\ref{prop:SHm} in Section~\ref{sec:proofprop}.
Finally, we show how Proposition~\ref{prop:SHm} implies Theorem~\ref{thm:maintime} in Section~\ref{sec:proofmaintime}.

\subsection{Parameters estimates} \label{sec:parameters}

To strictly improve the estimates on the parameters in~\eqref{eq:BS}, we claim the following.

\begin{lemma}[Parameters bootstrap]
For all $s\in [S_n,S_n^*]$,
\be \label{par2}
\left|\lambda_n(s)-\frac1{\sqrt{2|s|}}\right|\leq \frac12 |s|^{-1},\quad
\left|b_n(s)+\frac1{2|s|}\right|\leq \frac12 |s|^{-\frac32},\quad
\left|x_n(s)+\sqrt{2|s|}\right|\leq \frac12.
\ee
\end{lemma}

\begin{proof}
In this proof, we denote $(\lambda_n,x_n,b_n,\e_n)$ simply by $(\lambda,x,b,\e)$.
From~\eqref{eq:2002},~\eqref{eq:2003} and the bootstrap estimate on $\|\e(s)\|_{\dot H^0_B}$ in~\eqref{eq:BS},
and since $K>20$ and we can assume $B$ large,
we have, for all $s\in [S_n,S_n^*]$,
\be \label{par1}
\left| \frac{\lambda_s}\lambda +b \right| +\left|\frac{x_s}{\lambda} - 1\right| + |b_s + 2b^2 + \beta_3 b^3| \lesssim |s|^{-4}.
\ee
In particular, from~\eqref{par1} and~\eqref{eq:BS},
\begin{align*}
\frac d{ds} \left( \frac{b-\frac{\beta_3}2 b^2}{\lambda^2}\right)
&= \frac1{\lambda^2} \left[b_s (1-\beta_3 b) - 2 \frac{\lambda_s}{\lambda} \left(b-\frac{\beta_3}2 b^2\right)\right] \\
&= \frac1{\lambda^2} \left[-(2b^2+ \beta_3 b^3) (1-\beta_3 b) + 2 b \left(b-\frac{\beta_3}2 b^2\right)\right]+ O(|s|^{-3})=O(|s|^{-3}).
\end{align*}
By integration on $[S_n,s]$ and then using~\eqref{initquatre}, we obtain
\[
\left(\frac{b-\frac{\beta_3}2 b^2}{\lambda^2}\right)(s) = \left(\frac{b-\frac{\beta_3}2 b^2}{\lambda^2}\right)(S_n)+ O(|s|^{-2})=-1+ O(|s|^{-2}).
\]
Note that, by~\eqref{eq:BS}, we have
\[
\left(\frac{b^2}{\lambda^2}\right)(s) = -\frac1{2s} + O(|s|^{-\frac32}).
\]
Thus, using~\eqref{par1} again,
\[
-\left(\frac{\lambda_s}{\lambda^3}\right)(s) +\frac{\beta_3}4 \frac1 s = -1+ O(|s|^{-\frac32}).
\]
Integrating on $[S_n,s]$ and then using~\eqref{initdeux}, we obtain
\begin{align*}
\frac1{2\lambda^2(s)} & = \left(\frac1{2\lambda^2(S_n)}+ \frac{\beta_3}4 \log|S_n| + S_n\right) - s - \frac{\beta_3}4 \log|s| + O(|s|^{-\frac12}) \\
& = - s - \frac{\beta_3}4 \log|s| + O(|s|^{-\frac12}).
\end{align*}
In particular,
\[
\left| \frac1{\lambda^2(s)} + 2s\right| \lesssim |s|^{\frac14}.
\]
Therefore, for $|S_0|$ large enough,
\[
\left| \lambda(s) - \frac1{\sqrt{2|s|}}\right| \lesssim |s|^{-\frac54} \leq \frac12 |s|^{-1}.
\]
From the previous estimates, we have
\[
\left( \frac b{\lambda^2} \right)(s) = - 1 + O(|s|^{-1}),
\]
and thus, for $|S_0|$ large enough,
\[
\left| b(s) -\frac1{2s}\right| \lesssim |s|^{-\frac74}\leq \frac12 |s|^{-\frac32}.
\]
Now, we turn to the estimate of $x(s)$.
From~\eqref{par1} and the previous estimate obtained on~$\lambda(s)$, we have
\[
\left| x_s(s) - \frac1{\sqrt{2|s|}} \right| \lesssim |s|^{-\frac54}.
\]
By integration and using~\eqref{initdeux}, we obtain, for $|S_0|$ large enough,
\[
\left|x(s)+\sqrt{2|s|} \right| \lesssim |s|^{-\frac14} \leq \frac12.
\]
This completes the proof of~\eqref{par2}.
\end{proof}

\subsection{$\dot H^0_B$ estimates}

Let $\psi\in\Cinfini$ be a nondecreasing function such that
\[
\psi(y) = \left\{
\begin{aligned}
e^y \quad & \m{for } y<-1, \\
1 \quad & \m{for } y >-\frac12,
\end{aligned}
\right.
\]
and $\psi(y)\geq e^y$ for $y<0$.

Let $\psi_B(y) = \psi\left(\frac y B\right)$, for $B> 1$ large to be chosen.
Note that, directly from the definition of~$\psi$, we have the estimates
\be \label{psi-estimates}
\left\{
\begin{aligned}
\psi_B(y)+B\psi'_B(y)+B^3|\psi'''_B(y)| \lesssim e^{y/B}, &\quad\m{for all } y\in\RR, \\
e^{y/B} \leq \psi_B(y), &\quad\m{for all } y<0.
\end{aligned}
\right.
\ee
Let
\[
F_{0,n} = \frac{1}{\lambda_n^2}
\left[ \int (\e_n)_y^2 \psi_B + \int \e_n^2 e^{y/B}
- \frac13 \int \left( (Q_{b_n}+\e_n)^6 - Q_{b_n}^6 - 6 Q_{b_n}^5 \e_n\right) \psi_B\right].
\]

\begin{lemma}[Local energy estimates] \label{le:enerloc}
There exist $B>100$, $\mu>0$, and $\kappa(B)>0$ a constant depending only on $B$ such that, for all $s\in [S_n,S_n^*]$, the following holds.
\begin{enumerate}[label=\emph{(\roman*)}]
\item \emph{Coercivity of $F_{0,n}$:}
\be \label{el1}
F_{0,n} \geq \frac{\mu}{\lambda_n^2} \left[\int (\e_n)_y^2 \psi_B + \int \e_n^2 e^{y/B}\right].
\ee
\item \emph{Time variation of $F_{0,n}$:}
\be \label{el2}
\frac{d F_{0,n}}{ds} \leq \kappa(B) |s|^{-2K-1}.
\ee
\end{enumerate}
\end{lemma}

\begin{proof}
As before, we denote $(\lambda_n,x_n,b_n,\e_n)$ and $F_{0,n}$ simply by $(\lambda,x,b,\e)$ and $F_0$.

(i) We first decompose $F_0$ as
\begin{align*}
\lambda^2 F_0 &= \int \left( \e_y^2\psi_B +\e^2e^{y/B} -5Q^4\e^2\psi_B \right) -5\int (Q_b^4-Q^4)\e^2\psi_B \\
&\quad -\frac13 \int \left((Q_b+\e)^6-Q_b^6-6Q_b^5\e-15Q_b^4\e^2\right) \psi_B.
\end{align*}
To estimate the first term, we rely on the coercivity of the linearized energy~\eqref{coercivity},
together with the choice of orthogonality conditions~\eqref{ortho1} and standard localization arguments.
Proceeding for instance as in the Appendix~A of~\cite{MMannals} or as in the proof of Lemma~\ref{lem:virialcut} below,
we obtain, for some $\bar\mu>0$ and for $B$ large enough,
\[
\int \left( \e_y^2\psi_B +\e^2e^{y/B} -5Q^4\e^2\psi_B \right) \geq \bar\mu \int (\e_y^2\psi_B+\e^2e^{y/B}).
\]
To estimate the second term, we use~\eqref{eq:QbtoQ},~\eqref{psi-estimates} and the bootstrap estimates~\eqref{eq:BS}
to obtain, for $|S_0|$ large enough,
\[
5\left|\int (Q_b^4-Q^4)\e^2\psi_B \right| \lesssim |b|\int \e^2e^{y/B} \leq \frac{\bar\mu}3 \int \e^2e^{y/B}.
\]
Finally, using also~\eqref{eq:Qbj}, the nonlinear term is estimated as
\begin{align*}
\frac13 \left| \int \left((Q_b+\e)^6-Q_b^6-6Q_b^5\e-15Q_b^4\e^2\right) \psi_B \right|
&\lesssim \int (|Q_b|^3|\e|^3 +|\e|^6)\psi_B \lesssim \|\e\|_{L^\infty} \int \e^2e^{y/B} \\
&\leq \frac{\bar\mu}3 \int \e^2e^{y/B},
\end{align*}
by choosing again $|S_0|$ large enough.
Gathering the above estimates and letting $\mu=\frac{\bar\mu}3>0$, we obtain~\eqref{el1}.

\medskip

(ii) An integration by parts first gives
\begin{align*}
\frac{dF_0}{ds} &= \frac{2}{\lambda^2} \int \e_s\left\{ -\psi'_B\e_y -\psi_B\e_{yy} +\e e^{y/B} -\psi_B[(Q_b+\e)^5-Q_b^5] \right\} \\
&\quad -2\frac{\lambda_s}{\lambda}F_0 -\frac{2}{\lambda^2} \int (Q_b)_s \left[(Q_b+\e)^5-Q_b^5-5Q_b^4\e\right]\psi_B.
\end{align*}
Denoting $G_B(\e) = \Big\{ -\psi'_B\e_y -\psi_B\e_{yy} +\e e^{y/B} -\psi_B[(Q_b+\e)^5-Q_b^5] \Big\}$
and using the equation~\eqref{eqofeps} satisfied by $\e$, we find
\[
\lambda^2 \frac{dF_0}{ds} = f_1+f_2+f_3+f_4+f_5,
\]
where
\begin{align*}
f_1 &= 2\int \left( -\e_{yy} +\e -[(Q_b+\e)^5-Q_b^5]\right)_y G_B(\e), \\
f_2 &= 2\frac{\lambda_s}{\lambda} \int \Lambda\e\, G_B(\e) - \e_y^2\psi_B -\e^2 e^{y/B} +\frac13\left( (Q_b+\e)^6-Q_b^6-6Q_b^5\e\right)\psi_B, \\
f_3 &= 2\left(\frac{\lambda_s}{\lambda}+b\right)\int \Lambda Q_b G_B(\e) + 2\left(\frac{x_s}{\lambda}-1\right) \int (Q_b+\e)_y G_B(\e)
     - 2(b_s+\theta(b))\int \frac{\partial Q_b}{\partial b} G_B(\e), \\
f_4 &= 2\int \Psi_b G_B(\e), \\
f_5 &= -2\int (Q_b)_s \left[(Q_b+\e)^5-Q_b^5-5Q_b^4\e \right]\psi_B.
\end{align*}
We now estimate each of these five terms separately.

\medskip

\textbf{Control of $f_1$.}
From multiple integrations by parts, we may rewrite $f_1$ as
\begin{align*}
f_1 &= -\int \left[ 3\psi'_B\e_{yy}^2 +\left( \frac3B e^{y/B}+\psi'_B-\psi'''_B\right)\e_y^2 +\frac1B\left(1-\frac{1}{B^2}\right) e^{y/B} \e^2 \right] \\
&\quad -2\int \left[ \frac{(Q_b+\e)^6}{6} -\frac{Q_b^6}{6} -Q_b^5\e -\left((Q_b+\e)^5-Q_b^5\right)\e \right] \left(\frac1B e^{y/B}-\psi'_B\right) \\
&\quad +2\int (Q_b)_y \left[ (Q_b+\e)^5-Q_b^5-5Q_b^4\e \right] \left( \psi_B -e^{y/B} \right) \\
&\quad +10\int \psi'_B \e_y \left[ (Q_b)_y\left((Q_b+\e)^4-Q_b^4\right) +\e_y(Q_b+\e)^4 \right] \\
&\quad -\int \psi'_B \left\{ \left[ -\e_{yy}+\e -\left( (Q_b+\e)^5-Q_b^5\right) \right]^2 -[-\e_{yy}+\e]^2 \right\} \\
&= f_1^< + f_1^>,
\end{align*}
where $f_1^<$ and $f_1^>$ respectively correspond to integration on $y<-\frac B2$ and $y>-\frac B2$.

\emph{Estimate of $f_1^<$.}
From~\eqref{psi-estimates} we find, taking $B$ large enough,
\begin{align*}
Bf_1^< &\leq -3\int_{y<-B/2} B\psi'_B\e_{yy}^2 - \frac12 \int_{y<-B/2} (\e_y^2+\e^2)e^{y/B} \\
&\quad +C\int_{y<-B/2} e^{y/B} (\e^6 + Q_b^4\e^2) \\
&\quad +C\int_{y<-B/2} Be^{y/B} |(Q_b)_y| \left( |\e|^5 +|Q_b|^3\e^2 \right) \\
&\quad +C\int_{y<-B/2} e^{y/B} |\e_y| \left[ |(Q_b)_y| \left(|Q_b|^3|\e| +\e^4\right) +|\e_y|(Q_b^4+\e^4) \right] \\
&\quad +C\int_{y<-B/2} B\psi'_B \left(Q_b^4|\e|+|\e|^5\right) \left( |\e_{yy}|+|\e|+Q_b^4|\e|+|\e|^5 \right),
\end{align*}
where $C$ denotes various positive constants.

Using the estimates~\eqref{eq:Qbj} on $Q_b$ when $y<-B/2$, which give for instance
\[
Q_b^4 \lesssim e^{-B} +|b|^4,\quad |(Q_b)_y| \lesssim e^{-B/4} +|b|^{1+\g},
\]
we get
\begin{align*}
Bf_1^< &\leq -3\int_{y<-B/2} B\psi'_B\e_{yy}^2 - \frac12 \int_{y<-B/2} (\e_y^2+\e^2)e^{y/B} \\
&\quad +C_0(\|\e\|_{L^\infty}+Be^{-B/4}+B|b|)\int_{y<-B/2} (\e_y^2+\e^2)e^{y/B} \\
&\quad +C_1(\|\e\|_{L^\infty}+e^{-B}+|b|^4)\int_{y<-B/2} B\psi'_B\e_{yy}^2,
\end{align*}
for some $C_0,C_1>0$.
Now we choose $B$ large enough, then $|S_0|$ large enough (depending on~$B$) such that $|b|$ and $\|\e\|_{L^\infty}$
are small enough thanks to the bootstrap estimates~\eqref{eq:BS}, so that we have
\[
C_0(\|\e\|_{L^\infty}+Be^{-B/4}+B|b|) \leq \frac14\quad \m{and}\quad C_1(\|\e\|_{L^\infty}+e^{-B}+|b|^4)\leq 1.
\]
Hence we obtain
\[
Bf_1^< \leq -2\int_{y<-B/2} B\psi'_B\e_{yy}^2 - \frac14 \int_{y<-B/2} (\e_y^2+\e^2)e^{y/B},
\]
and \emph{a fortiori}
\[
Bf_1^< \leq - \frac14 \int_{y<-B/2} (\e_y^2+\e^2)e^{y/B}.
\]

\emph{Estimate of $f_1^>$.}
When $y>-\frac B2$, we have $\psi_B(y)=1$ and in particular $\psi'_B\equiv\psi'''_B\equiv 0$, so $f_1^>$ reduces to
\begin{align*}
Bf_1^> &= -\int_{y>-B/2} \left[3\e_y^2 + \left(1-\frac{1}{B^2}\right) \e^2 \right] e^{y/B} \\
&\quad -2\int_{y>-B/2} \left[ \frac{(Q_b+\e)^6}{6} -\frac{Q_b^6}{6} -Q_b^5\e -\left((Q_b+\e)^5-Q_b^5\right)\e \right] e^{y/B} \\
&\quad +2\int_{y>-B/2} (Q_b)_y \left[ (Q_b+\e)^5-Q_b^5-5Q_b^4\e \right]e^{y/B} B\left( e^{-y/B}-1 \right).
\end{align*}
Expanding the last expression, we find
\[
Bf_1^> = -\int_{y>-B/2} \left[ 3\e_y^2 +\e^2 -5Q^4\e^2 +20yQ'Q^3\e^2 \right] e^{y/B} +R_{\mathrm{Vir}}(\e),
\]
where
\begin{align*}
R_{\mathrm{Vir}}(\e) &= \frac{1}{B^2}\int_{y>-B/2} \e^2 e^{y/B} +\int_{y>-B/2} 5\e^2 e^{y/B}(Q_b^4-Q^4) \\
&\quad + \int_{y>-B/2} 20\e^2 e^{y/B} [y+B(e^{-y/B}-1)]Q'Q^3 \\
&\quad + \int_{y>-B/2} 20\e^2 e^{y/B} B(e^{-y/B}-1)Q'(Q_b^3-Q^3) \\
&\quad + \int_{y>-B/2} 20\e^2 e^{y/B} B(e^{-y/B}-1)Q_b^3(Q_b-Q)_y \\
&\quad + \int_{y>-B/2} \left[ \frac{40}{3} Q_b^3\e +15Q_b^2\e^2 +8Q_b\e^3 +\frac53\e^4 \right]\e^2 e^{y/B} \\
&\quad + \int_{y>-B/2} (Q_b)_y \e^2 e^{y/B} B(e^{-y/B}-1) \left[ 20Q_b^2\e +10Q_b\e^2 +2\e^3 \right].
\end{align*}
To estimate the first term in $Bf_1^>$, we rely on the following coercivity result under the orthogonality conditions~\eqref{ortho1},
which is a variant of~\cite[Lemma~3.4]{MMR1} adapted to our case.

\begin{lemma} \label{lem:virialcut}
There exist $B_0>1$ and $\mu_1>0$ such that, for all $B\geq B_0$,
\[
\int_{y>-B/2} \left[ 3\e_y^2 +\e^2 -5Q^4\e^2 +20yQ'Q^3\e^2 \right] e^{y/B} \geq \mu_1\int_{y>-B/2} (\e_y^2+\e^2)e^{y/B} -\frac{1}{B} \int \e^2 e^{-|y|/2}.
\]
\end{lemma}

\begin{proof}
It is a simple consequence of the following coercivity property of a virial quadratic form under suitable repulsivity properties, proved in~\cite{MMjmpa}.
It is stated in Proposition~4 therein that there exists $\mu_0>0$ such that, for all $v\in H^1(\RR)$,
\be \label{eq:virial}
\int (3v_y^2 + v^2 -5 Q^4v^2 +20y Q'Q^3v^2) \geq
\mu_0\int (v_y^2+v^2) -\frac{1}{\mu_0} \left(\int vy\Lambda Q\right)^2 -\frac{1}{\mu_0} \left( \int vQ\right)^2.
\ee
We now proceed to a suitable localization argument.
Let $\zeta$ be a smooth nondecreasing function such that
\[
\zeta(y) = \left\{
\begin{aligned}
0 &\quad\m{for } y<-\frac12, \\
e^{y/2} &\quad\m{for } y>-\frac14,
\end{aligned}
\right.
\]
and $\zeta(y)\leq e^{y/2}$ for all $y\in\RR$.
Set $\tilde\e(y) = \e(y)\zeta_B(y)$, where $\zeta_B(y) = \zeta(\frac{y}{B})$ and $B>1$ will be chosen large enough.
Applying~\eqref{eq:virial} to $\tilde\e$, we obtain
\begin{multline} \label{eq:virial-loc}
(3-\mu_0)\int \tilde\e_y^2 +(1-\mu_0)\int \tilde\e^2 -5\int Q^4\tilde\e^2 +20\int y Q'Q^3\tilde\e^2 \\
\geq -\frac{1}{\mu_0} \left(\int \tilde\e y\Lambda Q\right)^2 -\frac{1}{\mu_0} \left( \int \tilde\e Q \right)^2.
\end{multline}
To estimate the left-hand side of~\eqref{eq:virial-loc}, we first notice that
\[
\int \tilde\e^2 = \int \e^2 \zeta_B^2 \leq \int_{y>-B/2} \e^2e^{y/B},
\]
and
\begin{align*}
\int \tilde\e_y^2
&= \int \e_y^2\zeta_B^2 +\int \e^2(\zeta'_B)^2 -\frac12 \int \e^2(\zeta_B^2)'' \\
&\leq \int_{y>-B/2} \e_y^2 e^{y/B} -\frac{1}{4B^2}\int_{y>-B/4} \e^2 e^{y/B}
+ \int_{-B/2<y<-B/4} \left[ \e^2(\zeta'_B)^2 -\frac12 \e^2(\zeta_B^2)'' \right] \\
&\leq \int_{y>-B/2} \e_y^2 e^{y/B} +\frac{C}{B^2}\int_{-B/2<y<-B/4} \e^2 e^{y/B},
\end{align*}
with $C>0$. Then, from $yQ'<0$ and the exponential decay of $Q$ and $Q'$, we have
\begin{multline*}
-5\int Q^4\tilde\e^2 +20\int y Q'Q^3\tilde\e^2 \leq -5\int_{y>-B/4} Q^4\e^2e^{y/B} +20\int_{y>-B/4} y Q'Q^3\e^2e^{y/B} \\
\leq -5\int_{y>-B/2} Q^4\e^2e^{y/B} +20\int_{y>-B/2} y Q'Q^3\e^2e^{y/B} +Ce^{-\frac{B}{2}} \int_{-B/2<y<-B/4} \e^2 e^{y/B}.
\end{multline*}
Thus, for $B$ large and assuming $\mu_0<1$,
\begin{multline*}
(3-\mu_0)\int \tilde\e_y^2 +(1-\mu_0)\int \tilde\e^2 -5\int Q^4\tilde\e^2 +20\int y Q'Q^3\tilde\e^2 \leq (3-\mu_0)\int_{y>-B/2} \e_y^2e^{y/B} \\
+ \left(1-\frac{\mu_0}{2}\right)\int_{y>-B/2} \e^2e^{y/B} -5\int_{y>-B/2} Q^4\e^2e^{y/B} +20\int_{y>-B/2} y Q'Q^3\e^2e^{y/B}.
\end{multline*}
Finally, to estimate the right-hand side of~\eqref{eq:virial-loc}, we rely on the orthogonality conditions~\eqref{ortho1}
and obtain by the Cauchy--Schwarz inequality
\begin{align*}
\left(\int \tilde\e y\Lambda Q\right)^2
&= \left(\int \e(\zeta_B-1) y\Lambda Q\right)^2 \lesssim \left(\int |\e||\zeta_B-1| e^{-|y|/2} \right)^2 \\
&\lesssim \left( \int (\zeta_B-1)^2 e^{-|y|/2} \right) \left( \int \e^2 e^{-|y|/2} \right).
\end{align*}
From the inequality $|e^\frac{y}{2B}-1| \leq \frac{|y|}{2B}e^\frac{|y|}{2B}$, valid for all $y\in\RR$, we get
\[
\int (\zeta_B-1)^2 e^{-|y|/2} \leq \int_{y<-B/4} e^{y/2} +\frac{1}{4B^2} \int_{y>-B/4} y^2e^{|y|/B}e^{-|y|/2} \lesssim \frac{1}{B^2},
\]
and so, for $B$ large enough,
\[
\left(\int \tilde\e y\Lambda Q\right)^2 \leq \frac{\mu_0}{2B} \int \e^2 e^{-|y|/2}.
\]
Since we may get similarly the same estimate for the term $\int \tilde\e Q$, we finally reach the conclusion of Lemma~\ref{lem:virialcut}
by inserting all the above estimates in~\eqref{eq:virial-loc} and letting $\mu_1=\frac{\mu_0}{2}>0$.
\end{proof}

To estimate the term $R_{\mathrm{Vir}}(\e)$, we rely on~\eqref{eq:Qbj},~\eqref{eq:QbtoQ} and the inequalities
\[
|y+B(e^{-y/B}-1)|\leq \frac{y^2}{2B}e^{|y|/B},\qquad |B(e^{-y/B}-1)|\leq |y|e^{|y|/B},
\]
to obtain
\[
|R_{\mathrm{Vir}}(\e)| \lesssim \left( \frac1B +|b|+\|\e\|_{L^\infty}\right) \int_{y>-B/2} \e^2 e^{y/B}.
\]
Indeed, we have for instance the estimates, for all $y>-B/2$,
\begin{gather*}
|[y+B(e^{-y/B}-1)](Q'Q^3)(y)| \lesssim \frac{1}{B}y^2e^{|y|/B} e^{-4|y|} \lesssim \frac{1}{B}, \\
|B(e^{-y/B}-1)Q_b^3(y)| \lesssim |y|e^{|y|/B} \left[ e^{-3|y|/2} +|b|^3\charfunc{[-2,0]}(|b|^\g y) \right] \lesssim e^{-|y|} +|b|^2 \lesssim 1, \\
|B(e^{-y/B}-1)(Q_b)_y(y)| \lesssim |y|e^{|y|/B} \left[ e^{-|y|/2} +|b|^{1+\g}\charfunc{[-2,0]}(|b|^\g y) \right] \lesssim e^{-|y|/4} + |b| \lesssim 1.
\end{gather*}
Taking $B$ and $|S_0|$ large enough, we obtain
\[
Bf_1^> \leq -\frac{\mu_1}{2} \int_{y>-B/2} (\e_y^2+\e^2)e^{y/B} +\frac1B \int_{y<-B/2} \e^2 e^{y/B}.
\]

\emph{Estimate of $f_1$.}
Gathering the above estimates of $Bf_1^<$ and $Bf_1^>$ and assuming $B$ large enough and $\mu_1<\frac14$, we finally obtain,
with $\mu_2 = \frac{\mu_1}{2}>0$ independent of $B$,
\[
Bf_1 \leq -\mu_2 \int (\e_y^2 +\e^2)e^{y/B}.
\]

\textbf{Control of $f_2$.}
Using integrations by parts and the definition~\eqref{Lamb} of $\Lambda\e$, we obtain
\begin{align*}
-\int \Lambda\e(\psi_B\e_y)_y &= \int \psi_B\e_y^2 -\frac12 \int y\psi'_B\e_y^2, \\
\int \Lambda\e\,\e\, e^{y/B} &= -\frac{1}{2B}\int y \e^2 e^{y/B},
\intertext{and}
-\int \Lambda\e\,\psi_B \left[(Q_b+\e)^5-Q_b^5\right] &= \frac16\int [y\psi'_B-2\psi_B]\left[ (Q_b+\e)^6-Q_b^6-6Q_b^5\e\right] \\
&\quad +\int \psi_B\Lambda Q_b \left[ (Q_b+\e)^5-Q_b^5-5Q_b^4\e \right].
\end{align*}
From the definition of $f_2$, this leads to
\begin{multline*}
f_2 = \frac{2\lambda_s}{\lambda} \left\{ -\frac12 \int y\psi'_B\e_y^2 -\frac{1}{2B}\int y\e^2e^{y/B} -\int \e^2e^{y/B}
+\frac16\int y\psi'_B \left[ (Q_b+\e)^6-Q_b^6-6Q_b^5\e\right] \right. \\
\left. +\int \psi_B\Lambda Q_b \left[ (Q_b+\e)^5-Q_b^5-5Q_b^4\e \right] \right\}.
\end{multline*}
Now recall that $b<0$ from~\eqref{eq:BS2} and so, using~\eqref{eq:2002} and~\eqref{eq:BS}, $\frac{\lambda_s}{\lambda}>0$ (for $|S_0|$ large enough).
Since $\psi'_B(y)=0$ for $y>-\frac B2$, we thus obtain, using the control~\eqref{psi-estimates}, the \emph{inequality}
\begin{multline*}
Bf_2 \leq \frac{\lambda_s}{\lambda} \left\{ \int_{y<-\frac B2} |y|\e_y^2e^{y/B}
+\int_{y<0} |y|\e^2e^{y/B} +\frac13 \int_{y<-\frac B2} |y| \left| (Q_b+\e)^6-Q_b^6-6Q_b^5\e\right|e^{y/B} \right. \\
\left. {}+2\int B |\Lambda Q_b| \left| (Q_b+\e)^5-Q_b^5-5Q_b^4\e \right|e^{y/B} \right\}.
\end{multline*}
From the estimate~\eqref{eq:Qbj} of $Q_b$, we get
\[
Bf_2 \lesssim \left( \left|\frac{\lambda_s}{\lambda}+b\right| +|b|\right) \left( \int_{y<0} |y|(\e_y^2+\e^2)e^{y/B} +\int B\e^2e^{y/B} \right).
\]
But, from~\eqref{eq:2002} and~\eqref{eq:BS}, we have
\[
\left|\frac{\lambda_s}{\lambda}+b\right| +|b| \lesssim \|\e\|_{\dot H^0_B} +|b| \lesssim |s|^{-1}.
\]
Therefore, choosing $|S_0|$ large enough, we get
\[
Bf_2 \leq \frac{\mu_2}{10}\int \e^2 e^{y/B} + C_2|s|^{-1} \int_{y<0} |y|(\e_y^2+\e^2)e^{y/B},
\]
for some constant $C_2>0$.
To estimate the last integral, we use Hölder's inequality and~\eqref{eq:BS} to obtain,
with $C(B)>0$ a constant depending only on $B$,
\begin{align*}
C_2 \int_{y<0} |y|(\e_y^2+\e^2)e^{y/B}
&\leq C_2 \left( \int_{y<0} |y|^{4K} (\e_y^2+\e^2)e^{y/B} \right)^\frac{1}{4K} \left( \int_{y<0} (\e_y^2+\e^2)e^{y/B} \right)^{1-\frac{1}{4K}} \\
&\leq C_2 C(B) \left( \int \e_y^2+\e^2 \right)^\frac{1}{4K} \left( \int (\e_y^2+\e^2)e^{y/B} \right)^{1-\frac{1}{4K}} \\
&\leq \frac{\mu_2}{10} \left( \int (\e_y^2+\e^2)e^{y/B} \right)^{1-\frac{1}{4K}},
\end{align*}
by choosing $|S_0|$ large enough, as before depending on $B$.

Finally, using the inequality $ab\leq a^\frac{4K}{4K-1}+b^{4K}$ valid for all $a,b\geq 0$,
we get the estimate
\[
Bf_2 \leq \frac{2\mu_2}{10} \int (\e_y^2+\e^2)e^{y/B} + \frac{\mu_2}{10} |s|^{-4K}.
\]

\textbf{Control of $f_3$.}
Let us denote $f_{3,1}$, $f_{3,2}$ and $f_{3,3}$ the three terms involved in $f_3$ and estimate them separately.

\emph{Estimate of $f_{3,1}$.}
We first rewrite this term as
\begin{align*}
f_{3,1} &= 2\left(\frac{\lambda_s}{\lambda}+b\right) \int \Lambda Q(L\e) \\
&\quad + 2\left(\frac{\lambda_s}{\lambda}+b\right) \int \Lambda Q\e(e^{y/B}-1) \\
&\quad + 2\left(\frac{\lambda_s}{\lambda}+b\right) \int \Lambda Q \left\{ (1-\psi_B)\e_{yy} -\psi'_B\e_y +(1-\psi_B)\left[(Q_b+\e)^5-Q_b^5\right] \right\} \\
&\quad - 2\left(\frac{\lambda_s}{\lambda}+b\right) \int \Lambda Q \left[ (Q_b+\e)^5 -Q_b^5 -5Q^4\e \right] \\
&\quad + 2\left(\frac{\lambda_s}{\lambda}+b\right) \int \Lambda (Q_b-Q) \left\{ -\psi'_B\e_y -\psi_B\e_{yy} +\e e^{y/B} -\psi_B[(Q_b+\e)^5-Q_b^5] \right\} \\
&= f_{3,1,1} + f_{3,1,2} + f_{3,1,3} + f_{3,1,4} + f_{3,1,5}.
\end{align*}
To estimate the first term, we rely on the property~\emph{\ref{scaling}} of Lemma~\ref{lemma:L}
and on the orthogonality condition $(\e,Q)=0$ to obtain
\[
\int \Lambda Q(L\e) = (L\e,\Lambda Q) = (\e,L\Lambda Q) = -2(\e,Q) = 0,
\]
and so $f_{3,1,1} = 0$.
For the second term, using the orthogonality condition $(\e,y\Lambda Q)=0$,
together with~\eqref{eq:2002} and the inequality $|e^x-1-x|\leq \frac{x^2}{2}e^{|x|}$ valid for all $x\in\RR$, we get
\begin{align*}
|f_{3,1,2}| &= 2\left| \left(\frac{\lambda_s}{\lambda}+b\right) \int \Lambda Q\e\left(e^{y/B}-1-\frac{y}{B}\right) \right| \\
&\lesssim \frac{1}{B^2} \left[ \left(\int \e^2 {e^{-|y|/2}} \right)^{\frac12} + |b|^{K+1} \right] \int |\e| y^2 e^{-|y|/2} e^{|y|/B},
\end{align*}
and so
\[
|f_{3,1,2}| \lesssim \frac{1}{B^2} \left[ \left(\int \e^2 {e^{-|y|/2}} \right)^{\frac12} + |b|^{K+1} \right] \left(\int \e^2 e^{-|y|/2} \right)^{\frac12}
\lesssim \frac{1}{B^2} \left[ \int \e^2 e^{y/B} + |b|^{2K+2} \right],
\]
which gives, for $B$ chosen large enough and using~\eqref{eq:BS},
\[
B|f_{3,1,2}| \leq \frac{\mu_2}{50} \int \e^2 e^{y/B} + \frac{\mu_2}{50} |s|^{-2K-2}.
\]
For the next term $f_{3,1,3}$, we first integrate by parts to remove all derivatives on $\e$, then notice that $\psi_B\equiv 1$ on $[-B/2,+\infty)$,
and finally use the exponential decay of $Q$ and its derivatives to obtain, by choosing $B$ large enough as above,
\[
B|f_{3,1,3}| \lesssim Be^{-B/4} \left[ \left(\int \e^2 {e^{-|y|/2}} \right)^{\frac12} + |b|^{K+1} \right] \left(\int \e^2 e^{-|y|/2} \right)^{\frac12}
\leq \frac{\mu_2}{50} \int \e^2 e^{y/B} + \frac{\mu_2}{50} |s|^{-2K-2}.
\]
For the term $f_{3,1,4}$, using~\eqref{eq:QbtoQ} and the exponential decay of $\Lambda Q$, we obtain
\begin{align*}
\left| \int \Lambda Q \left[ (Q_b+\e)^5 -Q_b^5 -5Q^4\e \right] \right|
&\leq \left| \int \Lambda Q \left[ (Q_b+\e)^5 -Q_b^5 -5Q_b^4\e \right] \right| + 5\left| \int \Lambda Q \e (Q_b^4-Q^4) \right| \\
&\lesssim \int \e^2e^{-|y|/2} +|b| \left(\int \e^2 e^{-|y|/2} \right)^{\frac12} \\
&\lesssim (\|\e\|_{\dot H_B^0} +|b|) \left(\int \e^2 e^{-|y|/2} \right)^{\frac12}.
\end{align*}
Choosing $|S_0|$ large enough (depending on $B$ as before), we also get
\[
B|f_{3,1,4}| \leq \frac{\mu_2}{50} \int \e^2 e^{y/B} + \frac{\mu_2}{50} |s|^{-2K-2}.
\]
Finally, for the last term $f_{3,1,5}$, we first integrate by parts to remove all derivatives on $\e$ as before, and use~\eqref{eq:QbtoQ} to obtain
\[
B|f_{3,1,5}| \leq C(B)|b| \left| \frac{\lambda_s}{\lambda}+b\right| \int |y| \left[ e^{-|y|/2} + \charfunc{(-\infty,0]}(y) \right] |\e| e^{y/B},
\]
with $C(B)>0$.
But the Cauchy--Schwarz inequality gives
\[
\int_{y<0} |y||\e| e^{y/B} \leq \left(\int_{y<0} y^2e^{y/B}\right)^{\frac12}
\left(\int_{y<0} \e^2 e^{y/B} \right)^{\frac12} \leq C(B) \left(\int \e^2 e^{y/B} \right)^{\frac12},
\]
and similarly
\[
\int |y|e^{-|y|/2} |\e|e^{y/B} \leq \left(\int y^2e^{-|y|}e^{y/B}\right)^{\frac12}
\left(\int \e^2 e^{y/B} \right)^{\frac12} \lesssim \left(\int \e^2 e^{y/B} \right)^{\frac12}.
\]
Thus, we obtain
\[
B|f_{3,1,5}| \leq C(B)|b| \left| \frac{\lambda_s}{\lambda}+b\right| \left(\int \e^2 e^{y/B} \right)^{\frac12}
\leq \frac{\mu_2}{50} \int \e^2 e^{y/B} + \frac{\mu_2}{50} |s|^{-2K-2},
\]
by choosing $|S_0|$ large enough as above.

Gathering all the previous estimates obtained above, we get
\[
B|f_{3,1}| \leq \frac{\mu_2}{10} \int \e^2 e^{y/B} + \frac{\mu_2}{10} |s|^{-2K-2}.
\]

\emph{Estimate of $f_{3,2}$.}
Integrating by parts, we first find the identity
\[
\int \psi_B(Q_b+\e)_y \left[(Q_b+\e)^5-Q_b^5\right] = -\frac16 \int \psi'_B \left[(Q_b+\e)^6-Q_b^6-6Q_b^5\e\right] +5\int \psi_B(Q_b)_yQ_b^4\e.
\]
Hence, we may rewrite $f_{3,2}$ as
\begin{align*}
f_{3,2} &= 2\left( \frac{x_s}{\lambda}-1 \right) \int Q'[L\e +(e^{y/B}-1)\e] \\
&\quad +2\left( \frac{x_s}{\lambda}-1 \right) \int Q'[-\psi'_B\e_y+(1-\psi_B)\e_{yy}+5(1-\psi_B)Q^4\e] \\
&\quad +\frac13\left( \frac{x_s}{\lambda}-1 \right) \int \psi'_B \left[ (Q_b+\e)^6-Q_b^6-6Q_b^5\e\right] \\
&\quad -10\left( \frac{x_s}{\lambda}-1 \right) \int \e\psi_B[(Q_b)_yQ_b^4-Q'Q^4] \\
&\quad +2\left( \frac{x_s}{\lambda}-1 \right) \int (Q_b-Q)_y [-\psi'_B\e_y -\psi_B\e_{yy} +\e e^{y/B}] \\
&\quad +2\left( \frac{x_s}{\lambda}-1 \right) \int \e_y [-\psi'_B\e_y -\psi_B\e_{yy} +\e e^{y/B}] \\
&= f_{3,2,1} + f_{3,2,2} + f_{3,2,3} + f_{3,2,4} + f_{3,2,5} + f_{3,2,6}.
\end{align*}
To estimate these six terms, since $\left|\frac{x_s}{\lambda}-1\right|$ satisfies
the same control as $\left|\frac{\lambda_s}{\lambda}+b\right|$ by~\eqref{eq:2002},
we will closely follow the calculation done for $f_{3,1}$.

For instance, to estimate the first term, we rely on the cancelation $LQ'=0$
and the orthogonality conditions $(\e,Q)=(\e,\Lambda Q)=0$ which give $(\e,yQ')=0$ to obtain,
as for $f_{3,1,2}$ and by choosing $B$ large enough,
\begin{align*}
B|f_{3,2,1}| &= 2B \left|\frac{x_s}{\lambda}-1\right| \left| \int Q'\e \left(e^{y/B}-1-\frac{y}{B} \right) \right| \\
&\lesssim \frac{1}{B} \left|\frac{x_s}{\lambda}-1\right| \int y^2 e^{-|y|}|\e|e^{|y|/B}
\leq \frac{\mu_2}{50} \int \e^2 e^{y/B} + \frac{\mu_2}{50} |s|^{-2K-2}.
\end{align*}
The term $f_{3,2,2}$, handled as $f_{3,1,3}$, also satisfies the same estimate by choosing $B$ large enough,
and \emph{we definitely fix $B$ to this value}.

For the term $f_{3,2,3}$, we use~\eqref{eq:2002} and~\eqref{psi-estimates} to obtain
\[
B|f_{3,2,3}| \lesssim (\|\e\|_{\dot H_B^0} +|b|^{K+1}) \int \e^2 e^{y/B} \leq \frac{\mu_2}{50} \int \e^2 e^{y/B},
\]
by choosing $|S_0|$ large enough.

For the next term $f_{3,2,4}$, we rewrite $(Q_b)_yQ_b^4-Q'Q^4 = Q'(Q_b^4-Q^4)+Q_b^4(Q_b-Q)_y$ then estimate,
using~\eqref{eq:Qbj} and~\eqref{eq:QbtoQ},
\[
\left| \int \e\psi_B Q'(Q_b^4-Q^4) \right|
\lesssim |b| \left( \int \e^2e^{y/B} \right)^{\frac12}
\left( \int e^{y/B}e^{-2|y|} \right)^{\frac12} \lesssim |b| \left( \int \e^2e^{y/B} \right)^{\frac12},
\]
and
\begin{align*}
\left| \int \e\psi_B Q_b^4(Q_b-Q)_y \right|
&\lesssim |b| \left( \int \e^2e^{y/B} \right)^{\frac12}
\left( \int e^{y/B}\left[ e^{-4|y|} +|b|^8\charfunc{[-2,0]}(|b|^\g y) \right] \right)^{\frac12} \\
&\lesssim B^{\frac12} |b| \left( \int \e^2e^{y/B} \right)^{\frac12}.
\end{align*}
Together with~\eqref{eq:2002}, these estimates lead to, choosing $|S_0|$ large enough,
\[
B|f_{3,2,4}| \lesssim B^{\frac32} |b| \left[ \left( \int \e^2e^{y/B} \right)^{\frac12} +|b|^{K+1}\right] \left( \int \e^2e^{y/B} \right)^{\frac12}
\leq \frac{\mu_2}{50} \int \e^2 e^{y/B} + \frac{\mu_2}{50} |s|^{-2K-2}.
\]

For the term $f_{3,2,5}$, we proceed as in the estimate of $f_{3,1,5}$ and therefore obtain the same control.
Finally, for the last term $f_{3,2,6}$, we first integrate by parts to obtain
\[
\int \e_y [-\psi'_B\e_y -\psi_B\e_{yy} +\e e^{y/B}] = -\frac12 \int \psi'_B\e_y^2 -\frac{1}{2B} \int \e^2 e^{y/B}.
\]
Thus, from~\eqref{psi-estimates} and~\eqref{eq:2002}, we get
\[
B|f_{3,2,6}| \lesssim (\|\e\|_{\dot H_B^0} +|b|^{K+1}) \int (\e_y^2+\e^2)e^{y/B} \leq \frac{\mu_2}{50} \int (\e_y^2+\e^2)e^{y/B},
\]
by choosing $|S_0|$ large enough.

Gathering all the previous estimates obtained above, we get the estimate
\[
B|f_{3,2}| \leq \frac{2\mu_2}{10} \int (\e_y^2+\e^2) e^{y/B} + \frac{\mu_2}{10} |s|^{-2K-2}.
\]

\emph{Estimate of $f_{3,3}$.}
We first integrate by parts to remove all derivatives on $\e$ as before, and use~\eqref{eq:dQb} to obtain
\[
B|f_{3,3}| \leq C(B) |b_s+\theta(b)| \int \left[ e^{-|y|/2} +\charfunc{(-\infty,0]}(y) \right] |\e|e^{y/B},
\]
with $C(B)>0$.
The last integral is handled as for the term $f_{3,1,5}$ by the Cauchy--Schwarz inequality,
and from~\eqref{eq:2003} we get
\begin{align*}
B|f_{3,3}| &\leq C(B) \left[ \int \e^2e^{y/B} + |b|\left( \int \e^2e^{y/B} \right)^{\frac12} + |b|^{K+1} \right] \left( \int \e^2e^{y/B} \right)^{\frac12},
\intertext{and so}
B|f_{3,3}| &\leq C(B)(\|\e\|_{\dot H_B^0} +|b|) \int \e^2e^{y/B} + C(B) |b|^{K+1} \left( \int \e^2e^{y/B} \right)^{\frac12} \\
&\leq \frac{\mu_2}{10} \int \e^2 e^{y/B} +C(B) \left[ \frac{\mu_2}{10C(B)}\int \e^2e^{y/B} + \frac{10C(B)}{4\mu_2} |b|^{2K+2} \right],
\end{align*}
by choosing $|S_0|$ large enough and using the inequality $ab\leq \sigma a^2 +\frac{1}{4\sigma} b^2$ valid for all $a,b,\sigma>0$.
Thus, we obtain, by using~\eqref{eq:BS},
\[
B|f_{3,3}| \leq \frac{2\mu_2}{10} \int \e^2 e^{y/B} + C(B)|s|^{-2K-2}.
\]

\emph{Estimate of $f_3$.}
Gathering the above estimates of $f_{3,1}$, $f_{3,2}$ and $f_{3,3}$, we finally obtain, with $C_0(B)>0$ a constant depending only on $B$,
\[
B|f_3| \leq \frac{5\mu_2}{10} \int (\e_y^2+\e^2) e^{y/B} + C_0(B) |s|^{-2K-2}.
\]

\textbf{Control of $f_4$.}
We proceed as for the term $f_{3,3}$ above, by first integrating by parts to remove all derivatives on $\e$,
and use~\eqref{eq:Psibj} to obtain
\[
B|f_4| \leq C(B) \int \left[ |b|^{K+1}e^{-|y|/2} + |b|^{K+1}|y|^{K-1}\charfunc{[-1,0]}(|b|^\g y)
+ |b|^{1+\g} \charfunc{[-2,-1]}(|b|^\g y) \right] |\e| e^{y/B},
\]
with $C(B)>0$. Using the Cauchy--Schwarz inequality, we get
\begin{align*}
B|f_4| &\leq C(B)|b|^{K+1} \left( \int e^{-|y|}e^{y/B} \right)^{\frac12} \left( \int \e^2 e^{y/B} \right)^{\frac12} \\
&\quad + C(B)|b|^{K+1} \left( \int_{y<0} |y|^{2K-2} e^{y/B} \right)^{\frac12} \left( \int \e^2 e^{y/B} \right)^{\frac12}
+ C(B)|b|^{1+\g} \int_{y<-|b|^{-\g}} e^{y/B},
\end{align*}
and so
\[
B|f_4| \leq C(B)|b|^{K+1} \left( \int \e^2 e^{y/B} \right)^{\frac12} + C(B) |b|^{1+\g} e^{-|b|^{-\g}/B}.
\]
Proceeding as in the estimate of $f_{3,3}$, we finally obtain, for some $C_1(B)>0$,
\[
B|f_4| \leq \frac{\mu_2}{10} \int \e^2 e^{y/B} + C_1(B) |s|^{-2K-2}.
\]

\textbf{Control of $f_5$.}
Since $(Q_b)_s = b_s \frac{\partial Q_b}{\partial b}$ and
$\left\| \frac{\partial Q_b}{\partial b} \right\|_{L^\infty} \lesssim 1$ by~\eqref{eq:dQb}, we find,
using also~\eqref{eq:2003} and the definition~\eqref{def:Qbt} of~$\theta(b)$,
\[
B|f_5| \lesssim B|b_s| \int \e^2 e^{y/B} \lesssim B(|b_s+\theta(b)|+|\theta(b)|) \int \e^2 e^{y/B}
\lesssim B(|b|^2 +\|\e\|_{\dot H_B^0}^2) \int \e^2 e^{y/B}.
\]
Choosing $|S_0|$ large enough, again depending on $B$, we get
\[
B|f_5| \leq \frac{\mu_2}{10} \int \e^2 e^{y/B}.
\]

\textbf{Conclusion.}
Gathering all the previous estimates above, we get
\[
B\lambda^2 \frac{dF_0}{ds} \leq -\frac{\mu_2}{10} \int (\e_y^2+\e^2)e^{y/B} + C_2(B) |s|^{-2K-2},
\]
for some $C_2(B)>0$. For $|S_0|$ large enough, we also have $\lambda^{-2}(s) \leq 8|s|$ by~\eqref{eq:BS}.
Dividing the previous inequality by $B\lambda^2>0$ and setting $\kappa(B)=\frac{8C_2(B)}{B}$,
we finally obtain
\[
\frac{dF_0}{ds} \leq -\frac{\mu_2}{10B\lambda^2} \int (\e_y^2+\e^2)e^{y/B} + \kappa(B)|s|^{-2K-1},
\]
and \emph{a fortiori}~\eqref{el2}, which concludes the proof of Lemma~\ref{le:enerloc}.
\end{proof}

\subsection{$\dot H^m_B$ estimates}

Let $B>100$ be chosen as in Lemma~\ref{le:enerloc}.
For $m\geq 1$, we set
\[
F_{m,n}= \int (\partial_y^m \e_n)^2 e^{y/B} - \int (\partial_y^{m-1} \e_n)^2 h_{m,n},
\]
where
\[
h_{m,n}(s,y)=\frac53 \int_{-\infty}^y \left( (2m+1) \partial_y (Q_{b_n}^4) - \frac1B Q_{b_n}^4\right)(y') e^{y'/B}\, dy'.
\]
Note that, for all $j\geq 0$, all $S_n\leq s\leq S_n^*$ and all $y\in\RR$,
\be \label{bdfm}
|\partial_y^j h_{m,n}(s,y)| \lesssim e^{y/B},\quad |\partial_y^j \partial_s h_{m,n}(s,y)|\lesssim |s|^{-2} e^{y/B}.
\ee
For the second estimate, note that we have used $|(b_n)_s|\lesssim |s|^{-2}$ from~\eqref{eq:2003}.

\begin{lemma} \label{le:Hmloc}
Let $4\leq m \leq m_K$. There exists $\kappa_m>0$ such that, for all $S_n\leq s\leq S_n^*$,
\be \label{Hml1}
\frac{d F_{m,n}}{ds}\leq \kappa_m (C^*)^{m-1} |s|^{-2K+m-1}.
\ee
\end{lemma}

\begin{proof}
As before, we denote $(\lambda_n,x_n,b_n,\e_n)$, $h_{m,n}$ and $F_{m,n}$ simply by $(\lambda,x,b,\e)$, $h_m$ and $F_m$.

First, we have
\begin{align*}
\frac{dF_m}{ds}
&= 2 \int (\partial_y^m \e_s) (\partial_y^m \e) e^{y/B}
- 2 \int (\partial_y^{m-1} \e_s) (\partial_y^{m-1} \e) h_m - \int (\partial_y^{m-1} \e)^2 \partial_s h_m \\
&= - 2 \int (\partial_y^{m-1} \e_s) \left\{\partial_y [(\partial_y^{m} \e) e^{y/B}] + (\partial_y^{m-1} \e) h_m\right\}
- \int (\partial_y^{m-1}\e)^2 \partial_s h_m.
\end{align*}
Denoting $G_{m,B}= \Big\{ \partial_y [(\partial_y^{m} \e) e^{y/B}] + (\partial_y^{m-1} \e) h_m \Big\}$ and
using~\eqref{eqofeps}, we find
\[
\frac{dF_m}{ds} = f_{m,1}+f_{m,2}+f_{m,3}+f_{m,4}+f_{m,5},
\]
where
\begin{align*}
f_{m,1}& = 2 \int \partial_y^m\left(\e_{yy}- \e +5 Q_b^4 \e\right)G_{m,B}, \\
f_{m,2}& = -2 \frac{\lambda_s}{\lambda} \int (\partial_y^{m-1} \Lambda \e) G_{m,B}, \\
f_{m,3}& = 2 \int \partial_y^m\left( (Q_b+\e)^5 - Q_b^5 - 5 Q_b^4 \e\right) G_{m,B}, \\
f_{m,4}& = -2\left(\frac{\lambda_s}{\lambda}+b\right)\int (\partial_y^{m-1}\Lambda Q_b) G_{m,B} - 2\left(\frac{x_s}{\lambda}-1\right) \int (\partial_y^{m}(Q_b+\e)) G_{m,B} \\
       & \quad + 2(b_s+\theta(b))\int \left(\partial_y^{m-1}\frac{\partial Q_b}{\partial b}\right) G_{m,B}, \\
f_{m,5}& =- 2 \int (\partial_y^{m-1} \Psi_b) G_{m,B}- \int (\partial_y^{m-1}\e)^2 \partial_s h_m.
\end{align*}

\textbf{Control of $f_{m,1}$.}
First, by integration by parts,
\begin{align*}
2 \int \partial_y^m(\e_{yy}-\e) G_{m,B}
& = - \frac3B \int (\partial_y^{m+1} \e)^2 e^{y/B} - \frac1B\left(1-\frac1{B^2}\right) \int (\partial_y^m \e)^2 e^{y/B} \\
&\quad + 3 \int (\partial_y^m \e)^2 (h_m)_y + \int (\partial_y^{m-1} \e)^2 \left[(h_m)_y - (h_m)_{yyy}\right].
\end{align*}
By~\eqref{bdfm} and~\eqref{eq:BS},
\[
\left| \int (\partial_y^{m-1} \e)^2 \left[(h_m)_y - (h_m)_{yyy}\right]\right|
\lesssim \int (\partial_y^{m-1} \e)^2 e^{y/B}
\lesssim (C^*)^{m-1} |s|^{-2K + m-1}.
\]

Second, integrating by parts,
\begin{align*}
10\int \partial_y^m(Q_b^4\e)\partial_y[(\partial_y^m\e)e^{y/B}]
& = -10 \int \partial_y^{m+1} (Q_b^4\e) (\partial_y^m \e) e^{y/B} \\
& = -10 \int (\partial_y^{m+1} \e) (\partial_y^m \e) Q_b^4 e^{y/B} -10 (m+1) \int (\partial_y^m \e)^2 \partial_y (Q_b^4) e^{y/B} \\
&\quad + \sum_{j=0}^{m-1} c_{j,m} \int (\partial_y^j \e) (\partial_y^m \e) (\partial_y^{m+1-j} Q_b^4)e^{y/B},
\intertext{and so}
10\int \partial_y^m(Q_b^4\e)\partial_y[(\partial_y^m\e)e^{y/B}]
&= - 5 \int (\partial_y^m \e)^2 \left((2m+1) \partial_y (Q_b^4) - \frac1B Q_b^4\right) e^{y/B} \\
&\quad + c'_{m-1,m} \int (\partial_y^{m-1} \e)^2 \partial_y\left( (\partial_y^{2} Q_b^4)e^{y/B}\right) \\
&\quad +\sum_{j=0}^{m-2} c'_{j,m} \int (\partial_y^{m-1} \e) \partial_y\left((\partial_y^j \e) (\partial_y^{m+1-j} Q_b^4)e^{y/B}\right),
\end{align*}
where $c_{j,m}$ and $c'_{j,m}$ denote various constants depending only on $m$ and $j$.
Using~\eqref{eq:Qbj} and~\eqref{eq:BS}, we obtain
\begin{multline*}
\left|\int (\partial_y^{m-1} \e)^2 \partial_y\left( (\partial_y^{2} Q_b^4)e^{y/B} \right) \right|
+ \sum_{j=0}^{m-2} \left| \int (\partial_y^{m-1} \e) \partial_y\left( (\partial_y^j \e) (\partial_y^{m+1-j} Q_b^4) e^{y/B} \right) \right| \\
\lesssim \sum_{j=0}^{m-1} \int (\partial_y^j \e)^2 e^{y/B} \lesssim (C^*)^{m-1} |s|^{-2K+m-1}.
\end{multline*}
Due to the choice of $h_m$, we note that
\[
3 (h_m)_y = 5 \left((2m+1) \partial_y (Q_b^4) - \frac1B Q_b^4\right) e^{y/B}.
\]
This leads to the cancellation
\[
\left| 3 \int (\partial_y^m \e)^2 (h_m)_y + 10\int \partial_y^m(Q_b^4\e)\partial_y[(\partial_y^m\e)e^{y/B}] \right| \lesssim (C^*)^{m-1} |s|^{-2K+m-1}.
\]

Finally, proceeding as before using integration by parts,~\eqref{eq:Qbj},~\eqref{bdfm} and~\eqref{eq:BS}, we show
\[
\left|\int \partial_y^m (Q_b^4\e) (\partial_y^{m-1} \e) h_m \right| \lesssim (C^*)^{m-1} |s|^{-2K+m-1}.
\]

Therefore, gathering the information above and since $B>100$, we obtain the estimate
\[
f_{m,1} \leq - \frac1{2B} \int \left((\partial_y^{m+1} \e)^2+(\partial_y^{m} \e)^2\right) e^{y/B} +C_1 (C^*)^{m-1}|s|^{-2K+m-1},
\]
for some constant $C_1>0$.

\medskip

\textbf{Control of $f_{m,2}$.}
For this term, we use the same strategy as for the term $f_2$ in the proof of Lemma~\ref{le:enerloc}.
In particular, recall that $0<\frac{\lambda_s}\lambda\lesssim |s|^{-1}$.

First, note that
\begin{align*}
2 \int (\partial_y^{m} \Lambda \e) (\partial_y^m \e) e^{y/B}
&= \int \left ((2m+1) \partial_y^{m} \e +2 y(\partial_y^{m+1} \e) \right) (\partial_y^{m} \e) e^{y/B} \\
&= 2m \int (\partial_y^{m} \e)^2 e^{y/B} -\frac1B \int (\partial_y^{m} \e)^2 y e^{y/B}.
\end{align*}
For $|S_0|$ large enough, it follows by integration by parts that, for some $\wt C>0$,
\begin{align*}
- 2 \frac{\lambda_s}{\lambda} \int (\partial_y^{m-1} \Lambda \e) \partial_y[ (\partial_y^m\e)e^{y/B}]
&=2 \frac{\lambda_s}{\lambda} \int (\partial_y^{m} \Lambda \e) (\partial_y^m \e) e^{y/B} \\
&\leq \frac1{100 B} \int (\partial_y^{m} \e)^2 e^{y/B}+ \wt C |s|^{-1} \int_{y<0} (\partial_y^{m} \e)^2 |y| e^{y/B}.
\end{align*}
Moreover, using H\"older's inequality, as in the control of $f_2$ in the proof of Lemma~\ref{le:enerloc}, and then~\eqref{eq:BS},
we find, with $C(B)>0$ only depending on $B$ and then for $|S_0|$ large enough,
\begin{align*}
\int_{y<0} (\partial_y^{m} \e)^2 |y| e^{y/B}
&\lesssim C(B) \left(\int (\partial_y^{m} \e)^2 \right)^{\frac1{4K}}\left(\int (\partial_y^{m} \e)^2 e^{y/B}\right)^{1-\frac1{4K}} \\
&\leq \frac{1}{100\wt CB} \left(\int (\partial_y^{m} \e)^2 e^{y/B}\right)^{1-\frac1{4K}}.
\end{align*}
We thus obtain
\[
- 2 \frac{\lambda_s}{\lambda} \int (\partial_y^{m-1} \Lambda \e) \partial_y[ (\partial_y^m\e)e^{y/B}]
\leq \frac1{50 B} \int (\partial_y^{m} \e)^2 e^{y/B} + |s|^{-4K}.
\]

Second, we have
\begin{align*}
-2 \frac{\lambda_s}{\lambda} \int (\partial_y^{m-1} \Lambda \e) (\partial_y^{m-1} \e) h_m
&= -\frac{\lambda_s}{\lambda} \int \left ((2m-1) \partial_y^{m-1} \e +2 y(\partial_y^{m} \e) \right) (\partial_y^{m-1} \e) h_m \\
&= -2(m-1) \frac{\lambda_s}{\lambda} \int (\partial_y^{m-1} \e)^2 h_m +\frac{\lambda_s}{\lambda}\int (\partial_y^{m-1} \e)^2 y (h_m)_y.
\end{align*}
By~\eqref{bdfm}, the definition of $h_m$,~\eqref{eq:Qbj} (for $y>0$) and~\eqref{eq:BS},
\[
\left|\int (\partial_y^{m-1} \e)^2 h_m\right| + \left|\int_{y>0} (\partial_y^{m-1} \e)^2 y (h_m)_y \right|\lesssim (C^*)^{m-1} |s|^{-2K + m-1}.
\]
Moreover, by H\"older's inequality as before, and~\eqref{eq:BS},
\begin{align*}
|s|^{-1}\int_{y<0} (\partial_y^{m-1} \e)^2 |y| e^{y/B}
&\lesssim |s|^{-1}\left(\int (\partial_y^{m-1} \e)^2 \right)^{\frac1{4K}}\left(\int (\partial_y^{m-1} \e)^2 e^{y/B}\right)^{1-\frac1{4K}} \\
&\lesssim |s|^{-1}\left((C^*)^{m-1} |s|^{-2K + m-1} \right)^{1-\frac1{4K}}\lesssim (C^*)^{m-1} |s|^{-2K + m-1}.
\end{align*}

In conclusion for this term, we obtain, for some constant $C_2>0$,
\[
f_{m,2} \leq \frac1{50 B} \int (\partial_y^{m} \e)^2 e^{y/B} + C_2(C^*)^{m-1}|s|^{-2K + m-1}.
\]

\textbf{Control of $f_{m,3}$.}
First, we see using~\eqref{bdfm} that
\be \label{bdGm}
|G_{m,B}|\lesssim \left(|\partial_y^{m+1} \e| + |\partial_y^{m} \e|+ |\partial_y^{m-1} \e|\right)e^{y/B}.
\ee
Note that $(Q_b+\e)^5-Q_b^5-5Q_b^4\e$ is quadratic order and higher in $\e$.
It is thus suitable to use $L^\infty$ estimates for this term.
From~\eqref{eq:BS}, for $0\leq m' \leq m-1$, we observe that
\be \label{linf}
\|\partial_y^{m'} \e \|_{L^\infty} \lesssim \|\e\|_{\dot H^{m'+1}}^{\frac12}\|\e\|_{\dot H^{m'}}^{\frac12}
\lesssim C^* |s|^{-\frac14 -\frac\g 4- \g (m'+\frac12)}\lesssim |s|^{-\frac14},
\ee
for $|S_0|$ large enough (depending on $C^*$).
In particular, using~\eqref{eq:Qbj} and~\eqref{linf}, we obtain the pointwise bound
\begin{align*}
\left|\partial_y^m\left((Q_b+\e)^5-Q_b^5-5Q_b^4\e\right)\right|
&\lesssim \left(\sum_{m'=0}^{m-1} ( \|\partial_y^{m'} \e\|_{L^\infty}+\|\partial_y^{m'} \e\|_{L^\infty}^4)\right)
\left(\sum_{m'=0}^{m} |\partial_y^{m'} \e| \right) \\
&\lesssim |s|^{-\frac14} \sum_{m'=0}^{m} |\partial_y^{m'} \e|.
\end{align*}
Therefore, using the Cauchy--Schwarz inequality and~\eqref{eq:BS}, we find
\begin{align*}
\left|f_{m,3}\right|
&\lesssim |s|^{-\frac14} \left(\sum_{m'=0}^{m+1}\int (\partial_y^{m'} \e)^2 e^{y/B} \right) \\
&\lesssim |s|^{-\frac14} \int \left((\partial_y^{m+1} \e)^2 +(\partial_y^{m} \e)^2 \right)e^{y/B} + (C^*)^{m-1} |s|^{-2K+m-1}.
\end{align*}
For $|S_0|$ large enough, we obtain, for some constant $C_3>0$,
\[
\left|f_{m,3}\right| \leq \frac1{50 B} \int \left((\partial_y^{m+1} \e)^2 +(\partial_y^m \e)^2 \right) e^{y/B} + C_3(C^*)^{m-1} |s|^{-2K+m-1}.
\]

\textbf{Control of $f_{m,4}$.}
Note from~\eqref{eq:2002},~\eqref{eq:2003} and~\eqref{eq:BS} that
\be \label{moduls}
\left|\frac{\lambda_s}{\lambda}+b\right|+\left|\frac{x_s}{\lambda}-1\right|\lesssim |s|^{-K},\quad
\left|b_s+\theta(b)\right|\lesssim |s|^{-K-1}.
\ee
Moreover, by~\eqref{eq:Qbj} and~\eqref{bdfm}, using integration by parts, we find
\[
\left| \int (\partial_y^{m-1}\Lambda Q_b) G_{m,B} \right| +\left|\int (\partial_y^{m}Q_b) G_{m,B} \right|
\lesssim \|\e\|_{\dot H^0_B} \lesssim |s|^{-K}.
\]
Also, using~\eqref{bdGm} and~\eqref{eq:BS},
\[
\left| \int (\partial_y^{m}\e) G_{m,B}\right|
\lesssim \int \left((\partial_y^{m+1} \e)^2 +(\partial_y^{m} \e)^2 \right) e^{y/B} + (C^*)^{m-1} |s|^{-2K+m-1}.
\]
Therefore, for some constant $C>0$ and for $|S_0|$ large enough,
\begin{multline*}
\left|\left(\frac{\lambda_s}{\lambda}+b\right)\int (\partial_y^{m-1}\Lambda Q_b) G_{m,B}\right|
+\left|\left(\frac{x_s}{\lambda}-1\right) \int (\partial_y^{m}(Q_b+\e)) G_{m,B}\right| \\
\leq \frac1{100 B} \int \left((\partial_y^{m+1} \e)^2 +(\partial_y^{m} \e)^2 \right) e^{y/B}+ C|s|^{-2K}.
\end{multline*}
Similarly, we obtain, using integration by parts,~\eqref{eq:dQb},~\eqref{moduls} and~\eqref{eq:BS},
\[
\left| (b_s+\theta(b))\int \left(\partial_y^{m-1}\frac{\partial Q_b}{\partial b}\right) G_{m,B}\right| \lesssim |s|^{-2K-1}.
\]
In conclusion, for $|S_0|$ large enough and for some constant $C_4>0$,
\[
|f_{m,4}|\leq \frac1{100 B} \int \left((\partial_y^{m+1} \e)^2 +(\partial_y^{m} \e)^2 \right) e^{y/B} + C_4 |s|^{-2K}.
\]

\textbf{Control of $f_{m,5}$.}
As a consequence of~\eqref{eq:Psibj} and~\eqref{eq:BS}, first note that there holds, for all~$ 0\leq j \leq K-1$,
\[
\int (\partial_y^j \Psi_b)^2(y) e^{y/B} \,dy \lesssim |s|^{-2K-2}.
\]
Therefore, integrating by parts (note that $2m\leq K-1$ since $m\leq m_K$), using the Cauchy--Schwarz inequality and~\eqref{bdfm},
\[
\left| \int (\partial_y^{m-1} \Psi_b) G_{m,B} \right| \lesssim |s|^{-K-1} \|\e\|_{\dot H^0_B} \lesssim |s|^{-2K-1}.
\]
Finally, by~\eqref{bdfm} and~\eqref{eq:BS},
\[
\left|\int (\partial_y^{m-1}\e)^2 \partial_s h_m\right| \lesssim |s|^{-2} (C^*)^{m-1} |s|^{-2K + m-1},
\]
and so, for some constant $C_5>0$,
\[
|f_{m,5}|\leq C_5(C^*)^{m-1} |s|^{-2K + m-3}.
\]

\textbf{Conclusion.}
Gathering all the previous estimates above, we get
\[
\frac{d F_{m}}{ds}\leq -\frac1{4B} \int \left((\partial_y^{m+1} \e)^2+(\partial_y^m \e)^2\right) e^{y/B} + \kappa_m (C^*)^{m-1} |s|^{-2K+m-1},
\]
for some $\kappa_m>0$ depending on $m$ (but independent of $C^*$) and $|S_0|$ large enough (possibly depending on $C^*$),
and \emph{a fortiori}~\eqref{Hml1}, which concludes the proof of Lemma~\ref{le:Hmloc}.
\end{proof}

\subsection{$\dot H^m$ estimates}

For $m\geq 1$, we set
\[
G_{m,n} = \frac1{\lambda^{2m}} \int (\partial_y^m \e_n)^2.
\]

\begin{lemma} \label{le:Hm}
Let $4\leq m \leq m_K$. There exists $\kappa_m'>0$ such that, for all $S_n\leq s\leq S_n^*$,
\be \label{Hm1}
\frac{d G_{m,n}}{ds}\leq \kappa_m' |s|^{-\frac32 - \frac\g 2 - (2\g-1)m}.
\ee
\end{lemma}

\begin{proof}
As before, we denote $(\lambda_n,x_n,b_n,\e_n)$ and $G_{m,n}$ simply by $(\lambda,x,b,\e)$ and $G_m$.

First, we have, using~\eqref{eqofeps},
\begin{align*}
\lambda^{2m} \frac{dG_m}{ds}
&= 2 \int (\partial_y^m \e_s) (\partial_y^m \e)
- (2m) \frac{\lambda_s}{\lambda} \int (\partial_y^{m} \e)^2 \\
&= g_{m,1}+g_{m,2}+g_{m,3}+g_{m,4},
\end{align*}
where
\begin{align*}
g_{m,1}& = -2 \int \partial_y^{m+1}\left(\e_{yy}- \e +(Q_b+\e)^5 - Q_b^5 \right) (\partial_y^m \e), \\
g_{m,2}& = 2 \frac{\lambda_s}{\lambda} \int (\partial_y^{m} \Lambda \e) (\partial_y^m \e)
         - (2m) \frac{\lambda_s}{\lambda} \int (\partial_y^{m} \e)^2, \\
g_{m,3}& = 2 \left(\frac{\lambda_s}{\lambda}+b\right) \int (\partial_y^{m}\Lambda Q_b) (\partial_y^m \e)
          +2 \left(\frac{x_s}{\lambda}-1\right) \int (\partial_y^{m+1}(Q_b+\e))(\partial_y^m \e) \\
       &\quad - 2 (b_s+\theta(b))\int \left(\partial_y^{m}\frac{\partial Q_b}{\partial b}\right)(\partial_y^m \e), \\
g_{m,4}& = 2 \int (\partial_y^m \Psi_b) (\partial_y^m \e).
\end{align*}

\textbf{Control of $g_{m,1}$.}
Integrating by parts, we have
\[
g_{m,1}= -2 \int \partial_y^{m+1}\left((Q_b+\e)^5 - Q_b^5 \right) (\partial_y^m \e).
\]
For $1\leq j\leq 5$, we have by integration by parts
\begin{multline*}
2\int \partial_y^{m+1}\left(Q_b^{5-j} \e^j\right) (\partial_y^m \e) =
j(2m-1) \int (\partial_y^m \e)^2 \partial_y(Q_b^{5-j} \e^{j-1}) \\
+2 \int \left[\partial_y^{m+1}\left(Q_b^{5-j} \e^{j} \right) -j (\partial_y^{m+1} \e)(Q_b^{5-j} \e^{j-1})
-m j(\partial_y^m \e) \partial_y(Q_b^{5-j} \e^{j-1})\right] (\partial_y^m \e).
\end{multline*}

First, we address the first term on the right-hand side, using~\eqref{eq:BS} and~\eqref{eq:Qbj}:
for $j=1$,
\[
\left|\int (\partial_y^m \e)^2 \partial_y(Q_b^4 ) \right|
\lesssim \int (\partial_y^m \e)^2 e^{-|y|/2} + |s|^{-4-\g} \int (\partial_y^m \e)^2
\lesssim (C^*)^{m}|s|^{-2K+m}+ |s|^{-4} \int (\partial_y^m \e)^2;
\]
for $2\leq j\leq 4$, since by~\eqref{eq:BS} and the choice~\eqref{def:gamma} of $\g$ we have
\[
\|\e\|_{L^\infty}+\|\partial_y\e\|_{L^\infty}\lesssim
\|\e\|_{\dot H^1}^{\frac12}(\|\e\|_{L^2}^{\frac12}+\|\e\|_{\dot H^2}^{\frac12}) \lesssim |s|^{-\frac78},
\]
then we obtain
\begin{align*}
\left|\int (\partial_y^m \e)^2 \partial_y(Q_b^{5-j} \e^{j-1}) \right|
&\lesssim \left(\|\e\|_{L^\infty}+\|\partial_y\e\|_{L^\infty}\right)
\left(\int (\partial_y^m \e)^2 e^{-|y|/2} + |s|^{-1}\int (\partial_y^m \e)^2\right) \\
&\lesssim (C^*)^{m}|s|^{-2K+m}+ |s|^{-\frac{15}8} \int (\partial_y^m \e)^2;
\end{align*}
and, for $j=5$,
\[
\left|\int (\partial_y^m \e)^2 \partial_y(\e^4) \right|
\lesssim \left(\|\e\|_{L^\infty}^3\|\partial_y\e\|_{L^\infty}\right) \int (\partial_y^m \e)^2
\lesssim |s|^{-\frac72} \int (\partial_y^m \e)^2.
\]
Therefore, for this term, we obtain, for $|S_0|$ large enough,
\begin{align*}
\sum_{j=1}^5\left| \int (\partial_y^m \e)^2 \partial_y(Q_b^{5-j} \e^{j-1})\right|
&\lesssim (C^*)^{m}|s|^{-2K+m}+ |s|^{-\frac{15}8} \int (\partial_y^m \e)^2 \\
&\lesssim (C^*)^{m}|s|^{-2K+m}+ |s|^{-2-\frac\g 2 - 2\g m}.
\end{align*}

Second, for $1\leq j\leq 4$, let
\[
\Sigma_m^j =\left\{\bar m=(m_1,\ldots,m_j), \ |\bar m| = \sum_{k=1}^j m_k\leq m+1,\
0\leq m_1\leq \cdots \leq m_j \leq m-1\right\}.
\]
Then, for $1\leq j\leq 4$,
\begin{multline*}
\left|\partial_y^{m+1}\left(Q_b^{5-j} \e^{j} \right) -j (\partial_y^{m+1} \e)(Q_b^{5-j} \e^{j-1})
-m j(\partial_y^m \e) \partial_y(Q_b^{5-j} \e^{j-1}) \right| \\
\lesssim \sum_{\bar m \in \Sigma_m^j} \left( \prod_{k=1}^j|\partial_y^{m_k} \e|\right) \left|\partial_y^{m+1-|\bar m|} (Q_b^{5-j})\right|.
\end{multline*}
Using~\eqref{eq:Qbj}, for $1\leq j\leq 4$, we have
\[
\left|\partial_y^{m+1-|\bar m|} (Q_b^{5-j})\right| \lesssim e^{-|y|/2} + |s|^{-(5-j) - (m+1-|\bar m|) \g}.
\]
We also observe that, from~\eqref{eq:BS},
\[
\|\partial_y^{m_k}\e\|_{L^\infty}
\lesssim \|\e\|_{\dot H^{m_k}}^{\frac12} \|\e\|_{\dot H^{m_k+1}}^{\frac12}
\lesssim C^* |s|^{-\frac14 - \frac\g 4 - \g (m_k+\frac12)}.
\]
Thus, using again~\eqref{eq:BS},
\begin{align*}
&\left| \int \left[\partial_y^{m+1}\left(Q_b^{5-j} \e^{j} \right) -j (\partial_y^{m+1} \e)(Q_b^{5-j} \e^{j-1})
-mj(\partial_y^m \e) \partial_y(Q_b^{5-j} \e^{j-1})\right] (\partial_y^m \e)\right| \\
&\lesssim \sum_{m'=0}^{m} \|\e\|_{\dot H_B^{m'}}^2 +\sum_{\bar m\in \Sigma_m^j} |s|^{-(5-j)-(m+1-|\bar m|)\g}
\left(\prod_{k=1}^{j-1} \|\partial_y^{m_k}\e\|_{L^\infty} \right) \|\e\|_{\dot H^{m_j}} \|\e\|_{\dot H^m} \\
&\lesssim (C^*)^{m} |s|^{-2K+m} + (C^*)^{5} |s|^{-2-4\g-2\g m}.
\end{align*}
For $j=5$, we set
\[
\Sigma_m =\left\{\bar m=(m_1,\ldots,m_5), \ |\bar m| = \sum_{k=1}^5 m_k= m+1,\
0\leq m_1\leq \cdots \leq m_5 \leq m-1\right\}.
\]
Then,
\begin{multline*}
\left| \int \left[\partial_y^{m+1}(\e^{5}) -5 (\partial_y^{m+1} \e)(\e^4)
-5m(\partial_y^m \e) \partial_y(\e^{4})\right] (\partial_y^m \e)\right|
\lesssim \sum_{\bar m\in \Sigma_m} \int \left(\prod_{j=1}^5 |\partial_y^{m_j} \e| \right) |\partial_y^m \e| \\
\lesssim \sum_{\bar m\in \Sigma_m} \left(\prod_{j=1}^4 \|\partial_y^{m_j} \e\|_{L^\infty}\right) \|\e\|_{\dot H^{m_5}} \|\e\|_{\dot H^m}
\lesssim (C^*)^{6} |s|^{-\frac32-\frac92\g-2\g m},
\end{multline*}
by similar arguments.

In conclusion, for $|S_0|$ large enough, depending as before on $C^*$,
\[
|g_{m,1}| \lesssim (C^*)^{m}|s|^{-2K+m}+ |s|^{- \frac32 -\frac\g 2- 2 \g m}
\lesssim |s|^{-\frac32 -\frac\g 2- 2 \g m},
\]
since $m\leq m_K$ and so $2K-m \geq 2+\frac\g 2+2 \g m$.

\medskip

\textbf{Control of $g_{m,2}$.}
It is easy to see that $g_{m,2}=0$ since, by integration by parts,
\[
\int (\partial_y^m \Lambda \e) (\partial_y^m \e) = m \int (\partial_y^m \e)^2.
\]

\textbf{Control of $g_{m,3}$.}
First, we note that $\int (\partial_y^{m+1} \e)(\partial_y^m \e)=0$ by integration by parts.
Then, using~\eqref{moduls}, integration by parts,~\eqref{eq:Qbj},~\eqref{eq:BS} and~\eqref{eq:BS2},
\begin{align*}
&\left| \left(\frac{\lambda_s}{\lambda}+b\right) \int (\partial_y^{m}\Lambda Q_b) (\partial_y^m \e)\right|
+\left| \left(\frac{x_s}{\lambda}-1\right) \int (\partial_y^{m+1}Q_b)(\partial_y^m \e)\right| \\
&=\left| \left(\frac{\lambda_s}{\lambda}+b\right) \int (\partial_y^{2m}\Lambda Q_b) \e\right|
+\left| \left(\frac{x_s}{\lambda}-1\right) \int (\partial_y^{2m+1}Q_b) \e\right| \\
&\lesssim |s|^{-K}\left( \|\e\|_{\dot H^0_B} + |s|^{-1+\frac\g2-2\g m} \|\e\|_{L^2}\right)
\lesssim |s|^{-K-\frac32+\frac\g2-2\g m}.
\end{align*}
Note that we used again that $m\leq m_K\leq K/2-1$.
Arguing similarly for
\[
(b_s+\theta(b))\int \left(\partial_y^{m}\frac{\partial Q_b}{\partial b}\right)(\partial_y^m \e)
\]
and using~\eqref{eq:dQb}, we obtain, since $K>1\geq\g$,
\[
|g_{m,3}|\lesssim |s|^{-\frac32 - \frac\g 2 - 2\g m}.
\]

\textbf{Control of $g_{m,4}$.}
By integration by parts,~\eqref{eq:Psibj} (note that $2m\leq K-1$ since $m\leq m_K$) and~\eqref{eq:BS2},
we obtain
\begin{align*}
|g_{m,4}|
&= 2\left|\int (\partial_y^{2m} \Psi_b) \e \right|
\lesssim |s|^{-(K+1)} \|\e\|_{\dot H_B^0} + |s|^{-(1+(1+2m)\g)} \int_{-2|b|^{-\g}<y<0} |\e| \\
&\lesssim |s|^{-2K-1} +|s|^{-1-(1+2m)\g} |s|^{\frac\g 2} \|\e\|_{L^2}
\lesssim |s|^{-\frac32 - \frac\g 2 - 2\g m}.
\end{align*}

\textbf{Conclusion.}
Gathering all the above estimates and using $\lambda^{-2m}(s)\lesssim |s|^m$, we finally get
\[
\frac{d G_m}{ds}\lesssim |s|^{-\frac32 - \frac\g 2 - (2\g-1)m},
\]
which concludes the proof of Lemma~\ref{le:Hm}.
\end{proof}

\subsection{Closing the estimates for $\e_n$} \label{sec:closingbootstrap}

Now, we use Lemmas~\ref{le:enerloc},~\ref{le:Hmloc} and~\ref{le:Hm} to close the estimates in $\e_n$ in~\eqref{eq:BS}.
We still denote $\e_n,F_{0,n},F_{m,n}$ and $G_{m,n}$ simply by $\e,F_0,F_m$ and $G_m$.

Let $s\in [S_n,S_n^*]$. Let $4\leq m_0\leq m_K$.
First, integrating~\eqref{el2} on $[S_n,s]$ and using~\eqref{el1} and $\e(S_n)\equiv 0$, we have,
for some constant $C>0$,
\be \label{bo1}
\|\e(s)\|^2_{\dot H^0_B}\leq C |s|^{-1} F_0(s) \leq C |s|^{-2K-1} \leq \frac12 |s|^{-2K},
\ee
for $|S_0|$ large enough.

\begin{remark}
Note that the use of the functional $F_0$ provides an estimate on $\|\e(s)\|_{\dot H^0_B}$, but not on $\|\e(s)\|_{\dot H^1_B}$
since the term in $\e_y$ has a stronger weight.
This functional, related to Weinstein's stability proof in~\cite{W1986}, decouples the behavior of $b$ and $\e$ at the main order.
No such quantity (even formally) exists at the level of regularity $L^2$ for $\e$, thus one has to use an $H^1$ functional.
The estimate of $\|\e(s)\|_{\dot H^m_B}$ for $1\leq m\leq m_K$ will be obtained by interpolation from $m=0$ to $m=m_0$, where $4\leq m_0\leq m_K$.
\end{remark}

Second, integrating~\eqref{Hm1} on $[S_n,s]$, and using $\lambda^{2m_0}(s) \lesssim |s|^{-m_0}$,
\[
\|\e(s)\|_{\dot H^{m_0}}^2 \lesssim |s|^{-\frac12 -\frac\g 2 -2 \g m_0}.
\]

Third, integrating~\eqref{Hml1} on $[S_n,s]$,
\[
F_{m_0}(s) \lesssim (C^*)^{m_0-1} |s|^{-2K+m_0}.
\]
By~\eqref{bdfm} and~\eqref{eq:BS},
\[
F_{m_0}(s) \geq \|\e(s)\|_{\dot H_B^{m_0}}^2 - C \|\e(s)\|_{\dot H_B^{m_0-1}}^2
\geq \|\e(s)\|_{\dot H_B^{m_0}}^2 - C (C^*)^{m_0-1} |s|^{-2K+m_0-1},
\]
with $C>0$.
Therefore,
\be \label{bo3}
\|\e(s)\|_{\dot H_B^{m_0}}^2\lesssim (C^*)^{m_0-1} |s|^{-2K+m_0}.
\ee

It is clear taking $C^*$ large enough that we have improved the estimates on $\e$ in~\eqref{eq:BS} for $m=0$ and $m=m_0$.

Now, we use simple interpolation to improve the estimates for $1\leq m <m_0$, possibly taking a larger $C^*$.
First, it is clear that
\[
\|\e(s)\|_{\dot H^m}\lesssim \|\e(s)\|_{\dot H^{m_0}}^{\frac m{m_0}}\|\e(s)\|_{L^2}^{1-\frac m{m_0}}
\lesssim |s|^{-\frac m{m_0}(\frac14+\frac\g 4 +\g m_0)-\frac12 \left(1-\frac m{m_0}\right)}.
\]
Since $\frac12 \geq \frac14 + \frac\g 4$, we have $\frac12 (m_0-m) \geq (\frac14 +\frac\g 4) (m_0-m)$, and thus
\[
\frac m{m_0} \left(\frac14+\frac\g 4 +\g m_0 \right)+\frac12 \left(1-\frac m{m_0}\right) \geq \frac14+\frac\g 4+\g m.
\]
Therefore, for $|S_0|$ large enough, we have
\[
\|\e(s)\|_{\dot H^m}\lesssim |s|^{-\frac14-\frac\g 4-\g m}.
\]

We conclude by considering the norm $\|\e(s)\|_{\dot H^m_B}$.
Let $v=\e e^{y/2B}$.
Note first that
\begin{align*}
\|\e(s)\|_{\dot H^m_B}^2
&= \int \left(\partial_y^m (ve^{-y/2B})\right)^2 e^{y/B}
\lesssim \sum_{j=0}^m \|v(s)\|_{\dot H^j}^2
\lesssim \sum_{j=0}^m \|v(s)\|_{\dot H^{m_0}}^{2\frac j{m_0}}\|v(s)\|_{L^2}^{2\left(1-\frac j{m_0}\right)} \\
&\lesssim \sum_{j=0}^m \left(\sum_{k=0}^{m_0}\|\e(s)\|_{\dot H^{k}_B}^{2}\right)^{\frac j{m_0}} \|\e(s)\|_{\dot H_B^0}^{2\left(1-\frac j{m_0}\right)}.
\end{align*}
Thus, using~\eqref{eq:BS},~\eqref{bo1} and~\eqref{bo3}, we have, for $1\leq m<m_0$,
\begin{align*}
\|\e(s)\|_{\dot H^m_B}^2
&\lesssim \sum_{j=0}^m \left( (C^*)^{m_0-1} |s|^{-2K+m_0}\right)^{\frac j{m_0}} \left(|s|^{-2K}\right)^{\left(1-\frac j{m_0}\right)} \\
&\lesssim |s|^{-2K} \sum_{j=0}^m (C^*)^{j\frac{m_0-1}{m_0}} |s|^{j}\lesssim (C^*)^{({m_0-1}) \frac m{m_0}} |s|^{-2K+m}.
\end{align*}

Taking and \emph{definitely fixing} $C^*$ large enough, we have strictly improved all the estimates in~\eqref{eq:BS}, and thus proved $S_n^*=S_0$,
which concludes the proof of Proposition~\ref{proposition:uniforme}.

\subsection{Additional $L^2$ estimates in intermediate region} \label{sec:intermediate}

Let $h$ be a smooth nonnegative function on $\RR$ supported on $[-2,-1]$.
Set $H(y) = \int_{-\infty}^y h^K(y')\,dy'$.
Let $\frac34<\nu<\g$.
We set
\[
J_n(s) = \int \e_n^2(s,y) H(|s|^{-\nu}y)\,dy.
\]

\begin{lemma} \label{l:gch}
There exists $\kappa>0$ such that, for all $S_n\leq s\leq S_0$,
\be \label{e:gch1}
\frac{dJ_n}{ds} \leq \kappa |s|^{-\frac92} + \kappa |s|^{-(1-\nu)K-\frac32}.
\ee
\end{lemma}
Note that, integrating~\eqref{e:gch1} on $[S_n,s]$, using $\e_n(S_n)\equiv 0$ and $H\gtrsim 1$ on $[-1,+\infty)$, we obtain
\be \label{addgch}
\int_{y>-|s|^\nu} \e_n^2(s,y)\, dy \lesssim J_n(s)\lesssim |s|^{-\frac72} + |s|^{-(1-\nu)K-\frac12}.
\ee
\begin{proof}[Proof of Lemma~\ref{l:gch}]
As before, we denote $(\lambda_n,x_n,b_n,\e_n)$ and $J_n$ simply by $(\lambda,x,b,\e)$ and $J$.
For the sake of brevity, we will also omit the variable of $H$ and its derivatives.

First, we have, using~\eqref{eqofeps},
\[
\frac{dJ}{ds} = 2 \int \e_s \e H - \nu s^{-1} \int (|s|^{-\nu}y)\e^2 H'
= j_1+j_2+j_3+j_4+j_5,
\]
where
\begin{align*}
j_1 &= -2 \int \partial_y \left(\e_{yy}-\e\right) \e H,\\
j_2 &= -2 \int \partial_y \left((Q_b+\e)^5-Q_b^5\right) \e H,\\
j_3 &= 2 \frac{\lambda_s}{\lambda} \int (\Lambda \e) \e H,\\
j_4 &= 2 \left(\frac{\lambda_s}{\lambda}+b\right) \int \Lambda Q_b \e H
+ 2 \left(\frac{x_s}{\lambda}-1\right) \int \partial_y (Q_b+\e) \e H
- 2 (b_s+\theta(b)) \int \frac{\partial Q_b}{\partial b} \e H,\\
j_5 &= 2 \int \Psi_b \e H - \nu s^{-1} \int (|s|^{-\nu}y)\e^2 h^K.
\end{align*}

\textbf{Control of $j_1$.}
Integrating by parts, we first find
\[
j_1 = - |s|^{-\nu} \int \left( 3 \e_y^2 + \e^2\right) h^K + |s|^{-3\nu}\int \e^2 (h^K)''.
\]
Since $|(h^K)''|\lesssim h^{K-2}$, we get, from Young's inequality and then~\eqref{eq:BS2},
\[
|s|^{-2\nu}\int \e^2 |(h^K)''| \lesssim |s|^{-2\nu}\int \e^2 h^{K-2}
\leq \frac12 \int \e^2 h^K + C |s|^{-K\nu} \int \e^2
\leq \frac12 \int \e^2 h^K + C |s|^{-K\nu-1},
\]
for some $C>0$.
Thus, since $K>20$ and $\nu>\frac34$,
\[
j_1 \leq -\frac12 |s|^{-\nu} \int \left(\e_y^2 + \e^2\right) h^K + C|s|^{-17}.
\]

\textbf{Control of $j_2$.}
Integrating by parts, we find
\begin{align*}
j_2 &= -\frac13 |s|^{-\nu} \int \left[ (Q_b+\e)^6-Q_b^6-6Q_b^5\e \right]h^K
+ 2 |s|^{-\nu} \int \left[ (Q_b+\e)^5-Q_b^5 \right]\e h^K \\
&\quad - 2 \int (Q_b)_y \left[ (Q_b+\e)^5-Q_b^5-5Q_b^4\e \right] H,
\end{align*}
and so
\[
|j_2| \lesssim |s|^{-\nu} \int (Q_b^4\e^2+\e^6)h^K + \int |(Q_b)_y| (|Q_b|^3\e^2+|\e|^5)H.
\]
To estimate the first term, we rely on~\eqref{eq:Qbj} and the controls
$\|\e\|_{L^\infty} \lesssim \|\e\|_{H^1}\lesssim |s|^{-\frac12}$ and $|b|\lesssim |s|^{-1}$ given by~\eqref{eq:BS} to obtain
\[
\int (Q_b^4\e^2+\e^6)h^K \lesssim \int e^{-2|y|}\e^2h^K + |b|^4 \int \e^2h^K + \|\e\|_{L^\infty}^4 \int \e^2h^K
\lesssim |s|^{-2} \int \e^2h^K.
\]
To estimate the second term, we rely on the same controls and the fact that, since $\nu<\g$,
$\charfunc{[-2,-1]}(|b|^\g y)=0$ when $y>-2|s|^\nu$, and $H(|s|^{-\nu}y)=0$ when $y\leq -2|s|^\nu$
from the definition of~$H$. Hence, we obtain the limiting terms
\[
\int |(Q_b)_y||Q_b|^3 \e^2 H \lesssim \int e^{-|y|}\e^2 + |b|^5 \int \e^2
\lesssim \|\e\|_{\dot H_B^0}^2 +|b|^5\|\e\|_{L^2}^2 \lesssim |s|^{-6}
\]
and
\[
\int |(Q_b)_y| |\e|^5 H \lesssim \|\e\|_{L^\infty}^3 \left( \|\e\|_{\dot H_B^0}^2 +|b|^2\|\e\|_{L^2}^2 \right)
\lesssim |s|^{-\frac32} \left( |s|^{-2K} + |s|^{-3} \right) \lesssim |s|^{-\frac92}.
\]
Thus, for $|S_0|$ large enough, we have obtained, for some $C>0$,
\[
|j_2| \leq \frac1{10} |s|^{-\nu} \int \e^2 h^K + C|s|^{-\frac92}.
\]

\textbf{Control of $j_3$.}
By integration by parts, we get, from~\eqref{eq:2002} and~\eqref{eq:BS},
\[
|j_3| = \left| \frac{\lambda_s}{\lambda}|s|^{-\nu} \int y\e^2 H' \right|
\lesssim |s|^{-1} \int \e^2 h^K \leq \frac1{10} |s|^{-\nu} \int \e^2 h^K.
\]

\textbf{Control of $j_4$.}
Using the smallness of the modulation parameters~\eqref{moduls},
and using also estimates~\eqref{eq:Qbj} and~\eqref{eq:dQb}, we obtain,
after integrating by parts the term $\int \e_y\e H$,
\begin{align*}
|j_4| &\lesssim |s|^{-K} \left( \int \e^2 h^K + \int e^{-|y|/4} |\e| + \int_{-2 |s|^\nu<y<0} |\e| \right) \\
&\lesssim |s|^{-K} \left( \|\e\|_{L^2}^2 + \|\e\|_{\dot H_B^0} + |s|^{\frac\nu2} \|\e\|_{L^2} \right) \\
&\lesssim |s|^{-K} \left( 1 + |s|^{-(1-\nu)/2} \right)
\lesssim |s|^{-K} \lesssim |s|^{-20}.
\end{align*}

\textbf{Control of $j_5$.}
Since $\nu<\g$, we have, by~\eqref{eq:Psibj}, for $y>-2 |s|^{\nu}$,
\begin{align*}
|\Psi_b(y)| &\lesssim |b|^{K+1} e^{-|y|/2} + |b|^{K+1} |y|^{K-1} \charfunc{[-2,0]}(|s|^{-\nu}y) \\
&\lesssim |s|^{-K-1} e^{-|y|/2} + |s|^{-(1-\nu)K-1-\nu} \charfunc{[-2,0]}(|s|^{-\nu}y).
\end{align*}
Thus, similarly to the estimate of $j_4$, we obtain
\begin{align*}
\left|\int \Psi_b \e H\right| &\lesssim |s|^{-K-1} \|\e\|_{\dot H_B^0} + |s|^{-(1-\nu)K-1-\nu} |s|^{\frac\nu2} \|\e\|_{L^2} \\
&\lesssim |s|^{-2K-1} + |s|^{-(1-\nu)K-\frac32-\frac\nu2} \lesssim |s|^{-(1-\nu)K-\frac32}.
\end{align*}
Finally, the last term in $j_5$ is estimated as
\[
\left| s^{-1} \int (|s|^{-\nu}y) \e^2 h^K \right|
\lesssim |s|^{-1} \int \e^2 h^K \leq \frac1{10} |s|^{-\nu} \int \e^2 h^K.
\]

\textbf{Conclusion.}
Gathering all the previous estimates above, we get
\[
\frac {dJ}{ds} \leq -\frac1{10} \int \left(\e_y^2 + \e^2\right) h^K + \kappa |s|^{-\frac92} + \kappa |s|^{-(1-\nu)K-\frac32},
\]
for some $\kappa>0$, and \emph{a fortiori}~\eqref{e:gch1}, which concludes the proof of Lemma~\ref{l:gch}.
\end{proof}

\subsection{Proof of Proposition~\ref{prop:SHm}} \label{sec:proofprop}

Going back to the original variables $(t,x)$, we claim from Proposition~\ref{proposition:uniforme} that there exists $t_0=t_{0,K}>0$ small
and independent of $n$ (but \emph{a priori} dependent of~$K$) such that, for all $0\leq m\leq m_K$, $n$ large enough and $t\in [T_n,t_0]$,
\be \label{gbetan}
\begin{gathered}
\|\e_n(t)\|_{L^2} \lesssim t,\quad \|\e_n(t)\|_{\dot H^m} \lesssim t^{\frac12(1+\g) + 2\g m},\quad
\|\e_n(t)\|_{\dot H^m_B} \lesssim t^{2K-m}, \\
|\lambda_n(t)-t| \lesssim t^2, \quad |b_n(t)+t^2|\lesssim t^3,\quad \left|x_n(t)+\frac 1t\right|\lesssim 1.
\end{gathered}
\ee
Indeed, from~\eqref{eq:BS}, we have $\lambda_n^3(s) = (2|s|)^{-\frac 32} + O(|s|^{-2})$, and so~\eqref{eq:time} rewrites as
\[
t-T_n = \int_{S_n}^s \lambda_n^3(s')\, ds'
= \frac1{\sqrt2} \left(\frac 1{\sqrt{|s|}} - \frac 1{\sqrt{|S_n|}} \right) + O(|s|^{-1}).
\]
By the definitions of $T_n$ and $S_n$, it follows that $t=\frac 1{\sqrt{2|s|}} + O(|s|^{-1})$,
or $\frac 1{\sqrt{2 |s|}} = t+O(t^2)$ equivalently. Thus,~\eqref{eq:BS} and~\eqref{eq:BS2} imply~\eqref{gbetan}.
Moreover,~\eqref{addgch} applied with $K>20$ and $\nu = \frac{17}{20}>\frac45$ gives
\[
\int_{y>-|s|^{\frac{17}{20}}} \e_n^2(s,y)\, dy \lesssim |s|^{-\frac72}.
\]
Since $t^{-\frac 85}<|s|^{\frac{17}{20}}$ by choosing $t_0$ possibly smaller, we also obtain
\be \label{Slocn}
\int_{y>-t^{-\frac 85}} \e_n^2(t,y) \, dy \lesssim t^7.
\ee

Observe that the sequence $(u_n(t_0))$ is bounded in $H^1$,
and also that the sequences $(\lambda_n(t_0))$ and $(x_n(t_0))$ are bounded in $\RR$.
As a consequence, the sequence $(v_n)$ defined by
\[
v_n = \lambda_n^{1/2}(t_0)u_n(t_0,\lambda_n(t_0) \cdot {}+x_n(t_0)) = Q_{b_n(t_0)} +\e_n(t_0)
\]
is also bounded in $H^1$.
Thus, there exist a subsequence $(v_{n_k})$ and $v_\infty\in H^1$ such that $v_{n_k} \rightharpoonup v_\infty$ weakly in $H^1$ as $k\to+\infty$,
and also subsequences $(\lambda_{n_k}(t_0))$, $(x_{n_k}(t_0))$, and $\lambda_\infty>0$, $x_\infty\in\RR$,
such that $\lambda_{n_k}(t_0)\to \lambda_\infty$ and $x_{n_k}(t_0)\to x_\infty$ as $k\to+\infty$.

Let $v_k(t)$ be the maximal solution of~\eqref{kdv} such that $v_k(0)= v_{n_k}$,
and let $(\lambda_k,x_k,b_k,\e_k)$ be its decomposition as given by Lemma~\ref{lemma:decomposition}.
Note that, by uniqueness of such a decomposition, we have the identities
\be \label{tempparam}
\begin{alignedat}{2}
\lambda_k(t) &= \frac{\lambda_{n_k}(t_0+\lambda_{n_k}^3(t_0)t)}{\lambda_{n_k}(t_0)},\quad
&x_k(t) &= \frac{x_{n_k}(t_0+\lambda_{n_k}^3(t_0)t)-x_{n_k}(t_0)}{\lambda_{n_k}(t_0)}, \\
b_k(t) &= b_{n_k}(t_0+\lambda_{n_k}^3(t_0)t),\quad
&\e_k(t) &= \e_{n_k}(t_0+\lambda_{n_k}^3(t_0)t).
\end{alignedat}
\ee
Now let $T_1>0$ be such that $T_1<\frac{t_0}{\lambda_\infty^3}$,
so that we may apply Lemma~\ref{le:weak} to $v_k(t)$ on $[-T_1,0]$,
since the conditions~\eqref{hypoweak} are fulfilled for $k$ large enough on such an interval.
We obtain that the solution $v(t)$ of~\eqref{kdv} such that $v(0)=v_\infty$
exists on $\left(-\frac{t_0}{\lambda_\infty^3},0\right]$,
and its decomposition $(\lambda_v,x_v,b_v,\e_v)$ is obtained as the (weak) limit of $(\lambda_k,x_k,b_k,\e_k)$ as $k\to +\infty$.

Then we define the solution $\wt S(t)$ of~\eqref{kdv}, for all $t\in (0,t_0]$, by
\[
\wt S(t,x) = \frac{1}{\lambda_\infty^{1/2}}\, v\left(\frac{t-t_0}{\lambda_\infty^3},\frac{x-x_\infty}{\lambda_\infty}\right),
\]
and denote $(\wt \lambda,\wt x,\wt b,\wt \e)$ its decomposition.
From~\eqref{gbetan}--\eqref{tempparam} and uniqueness of such a decomposition, we obtain, for all $0\leq m\leq m_K$, $t\in (0,t_0]$,
\begin{gather}
\|\wt \e(t)\|_{L^2} \lesssim t,\quad \|\wt \e(t)\|_{\dot H^m} \lesssim t^{\frac12(1+\g) + 2\g m},\quad
\|\wt \e(t)\|_{\dot H^m_B} \lesssim t^{2K-m},\quad \int_{y>-t^{-\frac 85}} \wt \e^2(t,y)\, dy \lesssim t^7, \nonumber \\
|\wt \lambda(t) - t| \lesssim t^2,\quad |\wt b(t)+t^2|\lesssim t^3,\quad \left|\wt x(t)+\frac 1t\right|\lesssim 1. \label{gbetanlim}
\end{gather}
Since $\wt \lambda(t)\to 0$ as $t\downarrow 0$, the solution $\wt S(t)$ blows up at time $0$.
Moreover, by weak convergence,~\eqref{inittrois} and~\eqref{QbL2}, we have
\begin{multline*}
\|\wt S(t)\|_{L^2} = \|v_\infty\|_{L^2} \leq \liminf \|v_{n_k}\|_{L^2} = \liminf \|u_{n_k}(t_0)\|_{L^2}
= \liminf \|u_{n_k}(T_{n_k})\|_{L^2} \\ = \liminf \|Q_{b_{n_k}(T_{n_k})}\|_{L^2} = \|Q\|_{L^2},
\end{multline*}
and thus $\|\wt S(t)\|_{L^2}=\|Q\|_{L^2}$
(recall that solutions with $\|u_0\|_{L^2}<\|Q\|_{L^2}$ are global and bounded in $H^1$).

The arguments above prove the bounds~\eqref{SHm} on the particular minimal mass blow up solution~$\wt S(t)$,
obtained by the limiting argument above.
Now, we recall how to obtain refined estimates on the parameters $(\wt \lambda, \wt x,\wt b)$ (see proof of Proposition~4.1 in~\cite{MMR1}).
Using~\eqref{gbetanlim}, we rewrite~\eqref{eq:2002} and~\eqref{eq:2003} applied to the solution $\wt S(t)$
in the $t$ variable (recall that $dt = \wt \lambda^3 ds$). We obtain
\be \label{odet}
|\wt \lambda^2 \wt\lambda_t + \wt b | + |\wt \lambda^2 \wt x_t - 1| + |\wt \lambda^3 \wt b_t + \theta(\wt b)|\lesssim t^K.
\ee
In particular, since $\beta_2=2$ in~\eqref{def:Qbthetab}, and using~\eqref{gbetanlim}, we get
\[
\frac{d}{dt}\left(\frac{\wt b}{\wt\lambda^2}\right)
= \frac{\wt\lambda^2 \wt b_t - 2 \wt b\wt\lambda \wt\lambda_t}{\wt\lambda^4}
= \frac {-\theta(\wt b) + 2\wt b^2+ O(t^K)}{\wt\lambda^5} = \beta_3 t + O(t^2).
\]
Since by~\eqref{gbetanlim} we have $\lim_{t\downarrow 0} \frac{\wt b}{\wt\lambda^2}(t)=-1$,
then by integrating on $(0,t)$, we obtain
\be \label{odett}
\left|\frac{\wt b}{\wt\lambda^2}(t) +1 - \frac{\beta_3}{2}t^2\right| \lesssim t^3.
\ee
Using~\eqref{odet} again, this proves $|\wt\lambda_t - 1 +\frac{\beta_3}{2}t^2 | \lesssim t^3$.
Thus, using $\lim_{t\downarrow 0} \wt\lambda(t) =0$, integrating on $(0,t)$, we obtain
$|\wt\lambda(t) - t + \frac{\beta_3}{6}t^3 | \lesssim t^4$.

Inserting this estimate again in~\eqref{odett}, we obtain $|\wt b(t) + t^2 -\frac{5\beta_3}{6}t^4|\lesssim t^5$.
Inserting the estimate on $\wt \lambda(t)$ in~\eqref{odet}, we obtain
\[
\left|\wt x_t - \frac{1}{t^2} -\frac{\beta_3}{3} \right| \lesssim t.
\]
It follows that $\wt x(t)+\frac 1t$ has a finite limit as $t\downarrow 0$, which we denote by $\wt x_0$.
In particular, integrating again on $(0,t)$, we obtain $|\wt x(t) + \frac 1t - \wt x_0 -\frac{\beta_3}{3} t | \lesssim t^2$.

Finally, from~\emph{\ref{uniqueS}} in Theorem~\ref{thm:previous}, any minimal mass blow up solution $S(t)$ which blows up at $t=0$
coincides with $\wt S(t)$ up to time independent sign change, scaling and space translation,
\emph{i.e.}~there exist $\epsilon_1=\pm1$, $\lambda_1>0$ and $x_1\in \RR$ such that, for all $t\in (0,\lambda_1^3 t_0]$,
\[
S(t,x) = \epsilon_1 \frac{1}{\lambda_1^{1/2}}\, \wt S\left(\frac{t}{\lambda_1^3},\frac{x-x_1}{\lambda_1}\right),
\]
and, by uniqueness of the decomposition $(\lambda,x,b,\e)$ of $\epsilon_1 S$ in Lemma~\ref{lemma:decomposition}, we have
\[
\lambda(t) = \lambda_1 \wt \lambda\left(\frac{t}{\lambda_1^3}\right),\quad
x(t) = x_1 +\lambda_1 \wt x\left(\frac{t}{\lambda_1^3}\right),\quad
b(t) = \wt b\left(\frac{t}{\lambda_1^3}\right),\quad
\e(t) = \wt\e\left(\frac{t}{\lambda_1^3}\right).
\]
By setting $T_0=T_{0,K}=\lambda_1^3t_0$, $\epsilon=\epsilon_1$, $\ell_0 = \frac{1}{\lambda_1^2}$ and $x_0= x_1+\lambda_1\wt x_0$,
estimates~\eqref{Spar}--\eqref{Sloc} for any such $S(t)$ decomposed as in~\eqref{eq:decomposition} thus follow.

To conclude, we claim that these parameters $T_0$, $\epsilon$, $\ell_0$ and $x_0$ do not depend on the choice of~$K$,
but only on the solution $S$ itself.
First, for $T_0$, we simply notice that, if~\eqref{Spar}--\eqref{Sloc} hold on $(0,T_{0,\wt K}]$ with $T_{0,\wt K}<T_0$,
then these estimates also hold on the compact set $[T_{0,\wt K},T_0]$
by the continuity of the functions $t\mapsto \lambda(t),x(t),b(t),\e(t)$,
and so on $(0,T_0]$ by choosing larger constants (depending on~$\wt K$) in~\eqref{Spar}--\eqref{Sloc}.
Note that we use the fact that $\delta_0$ given to ensure the existence and uniqueness
of the decomposition in Lemma~\ref{lemma:decomposition}~(i) is independent of $K$, as noticed in its proof.
Second, it is clear for $\epsilon$.
Next, for $\ell_0$, we first notice that $\|\e_y\|_{L^2}^2 +\|\e\|_{\dot H_B^0}^2 + |b|^2 =o(\lambda^2)$ from~\eqref{Spar}--\eqref{SHm}.
Thus, from~\eqref{energie} and then~\eqref{eq:PQ}, we obtain the following expression of $\ell_0$, which depends only on the energy of $S$:
\[
\ell_0 = - \lim_{t\to 0} \frac{b(t)}{\lambda^2(t)} = \frac{E(S)}{\int PQ} = \frac{16E(S)}{\|Q\|_{L^1}^2}.
\]
Finally, for $x_0$, if $(\lambda,x,b,\e)$ and $(\wt \lambda,\wt x,\wt b,\wt \e)$ are two decompositions corresponding
to two different values of $K$ and $\wt K$, then we have the identity
\[
[Q_{\wt b(t)} + \wt \e(t)](y) = (\wt \lambda(t) \lambda^{-1}(t))^{\frac 12} [Q_{b(t)} + \e(t)]
\Bigl(\wt \lambda(t) \lambda^{-1}(t) y + (\wt x(t) - x(t))\lambda^{-1}(t)\Bigl).
\]
Let $t\downarrow 0$. From~\eqref{Spar}--\eqref{SHm}, we have $\wt \lambda(t) \lambda^{-1}(t)\to 1$,
$\e(t),\wt \e(t)\to 0$ and $b(t),\wt b(t)\to 0$, thus, using~\eqref{eq:QbtoQ},
we have necessarily $(\wt x(t) - x(t))\lambda^{-1}(t)\to 0$.
Since $\lambda(t)\to 0$, it forces $\wt x(t)-x(t)\to 0$, and so $x_0=\wt x_0$ from~\eqref{Spar}, which finishes to prove the claim.

\subsection{Proof of Theorem~\ref{thm:maintime}} \label{sec:proofmaintime}

Let $m\geq 0$, and let $K>20$ be such that $0\leq m\leq m_K$, where $m_K=[K/2]-1$.
We also fix $\g=1$ throughout this proof.
Estimate~\eqref{th:timem} follows from the decomposition~\eqref{eq:decomposition},
the control of $\e(t)$ in $\dot H^m$ in Proposition~\ref{prop:SHm},
expansions of $\lambda(t)$, $b(t)$ and $x(t)$ in powers of $t$ as $t\downarrow 0$,
and the properties of the functions $P_k$ as noticed in Remark~\ref{rem:Pk}.

\medskip

\textbf{Step 1.}
First we notice that, from the definition~\eqref{def:locprofile} of $Q_b$
and the decomposition~\eqref{eq:decomposition} of~$S(t)$
with $\epsilon=1$, $\ell_0=1$ and $x_0=0$, we have
\[
S(t,x) = \frac1{\lambda^{1/2}(t)} \left[ Q +\sum_{k=1}^K b^k(t)P_k\chi_{b(t)} +\e(t)\right] \left(\frac{x-x(t)}{\lambda(t)}\right),
\]
and so
\begin{align}
\partial_x^m S(t,x)
&= \frac{1}{\lambda^{1/2}(t)} \partial_x^m \left[ Q\left(\frac{x-x(t)}{\lambda(t)}\right) \right]
+\sum_{k=1}^K \frac{b^k(t)}{\lambda^{1/2}(t)} \partial_x^m \left[ P_k\left(\frac{x-x(t)}{\lambda(t)}\right)
\chi_{b(t)}\left(\frac{x-x(t)}{\lambda(t)}\right) \right] \nonumber \\
&\quad + \frac{1}{\lambda^{1/2+m}(t)} \partial_y^m \e\left(t,\frac{x-x(t)}{\lambda(t)}\right). \label{eq:decompositionbis}
\end{align}
Moreover, the estimates of the parameters~\eqref{Spar} read
\be \label{Sparbis}
|\lambda(t) - t|\lesssim t^3,\quad |x(t) + t^{-1}|\lesssim t,\quad |b(t) + t^2 |\lesssim t^4.
\ee
Also, from~\eqref{SHm}, we have $\|\e(t)\|_{\dot H^m} \lesssim t^{1+2m}$ and thus
\be \label{aneps}
\left\| \frac{1}{\lambda^{1/2+m}(t)} \partial_y^m \e\left(t,\frac{\cdot-x(t)}{\lambda(t)}\right)\right\|_{L^2}
\lesssim t^{1+m} \to 0 \quad \m{as } t\downarrow 0.
\ee
Hence, we just have to estimate the first two terms in~\eqref{eq:decompositionbis} in $L^2$ to conclude.

\medskip

\textbf{Step 2.}
We claim that there exist $\lambda_k$, $b_k$ and $c_k$
such that, for all $0\leq p\leq m_K-2 $, $0<t\leq T_0$,
the parameters $\lambda(t)$, $b(t)$ and $x(t)$ admit the following expansions in powers of $t$:
\begin{align}
\lambda(t) &= t +\lambda_0t^3 +\lambda_1t^5 +\cdots +\lambda_{p-1}t^{2p+1} + O(t^{2p+3}), \label{expansion:lambda} \\
b(t) &= -t^2 +b_0t^4 +b_1t^6 +\cdots +b_{p-1}t^{2p+2} + O(t^{2p+4}), \label{expansion:b} \\
x(t) &= -t^{-1} -c_0t -c_1t^3 -\cdots -c_{p-1}t^{2p-1} + O(t^{2p+1}). \label{expansion:x}
\end{align}
First note that~\eqref{expansion:lambda}--\eqref{expansion:x} hold for $p=0$ from~\eqref{Sparbis}.
Hence, proceeding by induction on the first two expansions,
we may assume that~\eqref{expansion:lambda}--\eqref{expansion:b} hold for some $0\leq p\leq m_K-3$,
and prove that they hold for $p+1$, proceeding as in the proof of Proposition~\ref{prop:SHm} above.

Indeed, we first rely on the approximate equations~\eqref{odet}
satisfied by the functions $\lambda(t)$, $b(t)$ and $x(t)$, which read in this case
\be \label{odetbis}
|\lambda^2\lambda_t+b|+|\lambda^2 x_t-1|+|\lambda^3 b_t+\theta(b)|\lesssim t^K.
\ee
Next, since $\beta_2=2$ in~\eqref{def:Qbthetab}, we obtain similarly
\begin{align*}
\frac d{dt} \left( \frac{b}{\lambda^2}\right)
&= \frac{\lambda^2 b_t - 2 b\lambda \lambda_t}{\lambda^4}
= \frac {-\theta(b) + 2b^2+ O(t^K)}{\lambda^5} \\
&= -\frac{b^3}{\lambda^5} \left[ \beta_3+\beta_4 b +\cdots +\beta_Kb^{K-3} +O(t^{K-6}) \right],
\end{align*}
and so, from~\eqref{expansion:lambda}--\eqref{expansion:b},
there exist some constants $r_1,\ldots,r_p$ such that
\[
\frac d{dt} \left( \frac{b}{\lambda^2}\right)
= \beta_3 t + r_1t^3 +\cdots +r_pt^{2p+1} +O(t^{2p+3}).
\]
By integrating on $(0,t)$, and since by~\eqref{Sparbis}
we have $\lim_{t\downarrow 0} \frac{b}{\lambda^2}(t)=-1$, we obtain
\be \label{odettt}
\frac{b}{\lambda^2}(t) = -1 +\frac{\beta_3}2 t^2 +\frac{r_1}4 t^4 +\cdots +\frac{r_p}{2p+2} t^{2p+2} +O(t^{2p+4}).
\ee
Using $|\lambda^2\lambda_t+b|\lesssim t^K$ from~\eqref{odetbis}, we get an expansion for $\lambda_t$
which gives, after integration,
\[
\lambda(t) = t -\frac{\beta_3}6 t^3 -\cdots -\frac{r_p}{(2p+2)(2p+3)}t^{2p+3} +O(t^{2p+5}).
\]
Hence,~\eqref{expansion:lambda} holds for $p+1$. Using this expansion and again~\eqref{odettt},
we obtain an expansion of~$b(t)$ as in~\eqref{expansion:b} for $p+1$, which concludes the induction.

Finally, to prove~\eqref{expansion:x}, we simply rely on the estimates $|\lambda^2 x_t-1| \lesssim t^K$
and~\eqref{expansion:lambda} to obtain, for some constants $s_1,\ldots,s_{p-1}$,
\[
\frac{dx}{dt} = t^{-2} -2\lambda_0 +s_1t^2 +\cdots +s_{p-1}t^{2p-2} + O(t^{2p}).
\]
By integration, we thus obtain~\eqref{expansion:x} with $c_0=2\lambda_0=-\frac{\beta_3}3$
and $c_k = -\frac{s_k}{2k+1}$ for $1\leq k\leq p-1$.

\medskip

\textbf{Step 3.}
We now give a direct consequence of~\eqref{expansion:lambda}--\eqref{expansion:x}:
for any $A\in\Y$, $k\geq 0$,
there exist functions $A_\ell\in\Y$, for $0\leq \ell \leq \ell_K$, with $\ell_K=m_K-3$ and $A_0=(-1)^kA$, such that
\be \label{mm:paraA}
\frac{b^k(t)}{\lambda^{\frac 12+k}(t)}A\left(\frac{x-x(t)}{\lambda(t)}\right)
=\sum_{\ell=0}^{\ell_K}\frac 1{t^{\frac 12-k-2\ell}}A_\ell\left(\frac{x+\frac 1t}t +c_0\right)
+\E_{A,k}(t,x)
\ee
where $\E_{A,k}$ satisfies, for all $0\leq m\leq m_K$, $\|\E_{A,k}(t)\|_{\dot H^m}\lesssim t^{k-m+K-6}$.
Indeed, we first note that
\[
\frac{b^k(t)}{\lambda^{\frac 12+k}(t)}A\left(\frac{x-x(t)}{\lambda(t)}\right)
= \frac 1{t^{\frac 12-k}} \frac{\tilde b^k(t)}{\tilde\lambda^{\frac12+k}(t)}
A\left[\frac1{\tilde \lambda(t)} \left(\frac{x+\frac 1t}t+c_0\right) -\tilde x(t)\right]
\]
where, using~\eqref{expansion:lambda}--\eqref{expansion:x},
\begin{align*}
\tilde\lambda(t) &= \frac{\lambda(t)}t = 1+\sum_{\ell=1}^{\ell_K} \lambda_{\ell-1} t^{2\ell} +O(t^{K-6}), \\
\tilde b(t) &= \frac{b(t)}{t^2} = -1+\sum_{\ell=1}^{\ell_K} b_{\ell-1} t^{2\ell}+O(t^{K-6}), \\
\tilde x(t) &= \frac1{\lambda(t)} \left(x(t)+t^{-1}+c_0t\right) = -\frac1{\tilde \lambda(t)}\sum_{\ell=1}^{\ell_K} c_\ell t^{2\ell}+O(t^{K-6}).
\end{align*}
Note that, in the expansions of $\tilde \lambda(t)$, $\tilde b(t)$ and $\tilde x(t)$,
the variable $t$ appears systematically with an even exponent.
Moreover, when $t\downarrow 0$, $\tilde \lambda(t)\to 1$, $\tilde b(t)\to -1$ and $\tilde x(t)\to 0$.

Second, the Taylor formula with remainder of integral form gives, for any $A\in \Y$,
the existence of functions $A_{\ell_1,\ell_2}\in \Y$ (with $A_{0,0}=A$) such that,
for all $\tilde x\in\RR$, $\tilde \lambda>0$,
\[
A\left(\frac{y}{\tilde \lambda} -\tilde x\right)
= \sum_{\substack{\ell_1,\ell_2\geq 0 \\ \ell_1+\ell_2\leq \ell_K}} A_{\ell_1,\ell_2}(y)
\tilde x^{\ell_1} \left(\frac1{\tilde \lambda}-1\right)^{\ell_2} +\E_A(y),
\]
where $\|\E_A\|_{\dot H^m} \lesssim (|\tilde x|+ |\tilde \lambda^{-1}-1|)^{\ell_K+1}$ for all $m\geq 0$.
The combination of the two above observations gives~\eqref{mm:paraA}.

\medskip

\textbf{Step 4.}
Let $0\leq k\leq [m/2]$.
By Remark~\ref{rem:Pk}, we know that $P_k^{(k)} \in \Y$.
Thus, applying~\eqref{mm:paraA} to~$P_k^{(k)}$,
there exist functions $P_{k,\ell}\in \Y$, for $0\leq \ell \leq \ell_K$
and with $P_{k,0}=(-1)^k P_k^{(k)}$, such that, for all $0<t\leq T_0$,
\begin{align*}
\frac{b^k(t)}{\lambda^{\frac12}(t)} \partial_x^k \left[P_k\left(\frac{x-x(t)}{\lambda(t)}\right) \right]
&= \frac{b^k(t)}{\lambda^{\frac12+k}(t)} P_k^{(k)}\left(\frac{x-x(t)}{\lambda(t)}\right) \\
&= \sum_{\ell=0}^{\ell_K} \frac1{t^{\frac 12-k-2\ell}} P_{k,\ell}\left(\frac{x+\frac1t}t+c_0\right)+\E_k(t,x),
\end{align*}
where $\E_k$ satisfies $\|\E_k(t)\|_{\dot H^m}\lesssim t^{k-m+K-6}$.
Hence,
\[
\frac {b^{k}(t)}{\lambda^{\frac 12}(t)} \partial_x^m \left[ P_k\left(\frac{x-x(t)}{\lambda(t)}\right) \right]
= \sum_{\ell=0}^{\ell_K} \frac1{t^{\frac 12+m-2k-2\ell}} P_{k,\ell}^{(m-k)}\left(\frac{x+\frac1t}t+c_0\right) +\partial_x^{m-k}\E_k(t,x),
\]
where $\|\partial_x^{m-k}\E_k(t)\|_{L^2}\lesssim t^{2k-m+K-6}$.
Also, note that
\[
\left\| \frac1{t^{\frac 12+m-2k-2\ell}} P_{k,\ell}^{(m-k)}\left(\frac{\cdot+\frac1t}t+c_0\right)\right\|_{L^2} \lesssim t^{2k+2\ell-m}.
\]
At given $k$ and $m$, this term converges to $0$ in $L^2$ if $2k+2\ell-m>0$, \emph{i.e.}~if $\ell>\frac12 m-k$.
Thus, we restrict the sum to $0\leq \ell\leq [m/2] -k$.

First, for $k=0$, we obtain
\be \label{an0}
\frac1{\lambda^{\frac12}(t)} \partial_x^m \left[ Q\left(\frac{\cdot-x(t)}{\lambda(t)}\right) \right]
- \sum_{\ell=0}^{[m/2]} \frac1{t^{\frac12+m-2\ell}} P_{0,\ell}^{(m)}\left(\frac{\cdot+\frac1t}t+c_0\right)
\to 0\quad\m{in } L^2 \m{ as } t\downarrow 0.
\ee

Next, for $1\leq k\leq [m/2]$,
\begin{align*}
&\left\| \frac{b^k(t)}{\lambda^{\frac12}(t)} \partial_x^m \left[ P_k(1-\chi_{b(t)}) \left(\frac{\cdot-x(t)}{\lambda(t)}\right) \right]\right\|_{L^2}
\lesssim t^{2k-m} \sum_{j=0}^m \left\|\partial_y^{m-j}(1-\chi_{b(t)}) P_k^{(j)} \right\|_{L^2} \\
&\lesssim t^{2k-m}\sum_{j=0}^{k-1} \left\|\partial_y^{m-j}\chi_{b(t)} P_k^{(j)} \right\|_{L^2}
+t^{2k-m} \sum_{j=k}^m \left\|\partial_y^{m-j}(1-\chi_{b(t)}) P_k^{(j)} \right\|_{L^2} \\
&\lesssim t^{2k-m} \sum_{j=0}^{k-1} t^{2(m-j-\frac12)} t^{-2(k-j-1)} +e^{-1/t}
\lesssim t^{2k-m+2(m-k)+1} = t^{1+m} \to 0\quad \m{as } t\downarrow 0.
\end{align*}
Therefore,
\begin{multline} \label{an1}
\frac{b^k(t)}{\lambda^{\frac12}(t)} \partial_x^m \left[ P_k\left(\frac{\cdot-x(t)}{\lambda(t)}\right)
\chi_{b(t)}\left(\frac{\cdot-x(t)}{\lambda(t)}\right) \right]
\\ - \sum_{\ell=0}^{[m/2]-k} \frac1{t^{\frac12+m-2k-2\ell}} P_{k,\ell}^{(m-k)}\left(\frac{\cdot+\frac1t}t+c_0\right)
\to 0\quad\m{in } L^2 \m{ as } t\downarrow 0.
\end{multline}

Finally, for $ k\geq [m/2]+1$, one observes that
\be \label{an2}
\left\| \frac{b^k(t)}{\lambda^{\frac12}(t)} \partial_x^m \left[ P_k\left(\frac{\cdot-x(t)}{\lambda(t)}\right)
\chi_{b(t)}\left(\frac{\cdot-x(t)}{\lambda(t)}\right)\right] \right\|_{L^2} \lesssim
t^{2k-m-1}\to 0\quad \m{as } t\downarrow 0.
\ee

\medskip

\textbf{Step 5.}
In conclusion, gathering~\eqref{eq:decompositionbis},~\eqref{aneps} and~\eqref{an0}--\eqref{an2}, we obtain
\begin{multline*}
\left\|\partial_x^m S(t) - \sum_{k=0}^{[m/2]} \sum_{\ell=0}^{[m/2]-k} \frac1{t^{\frac12+m-2k-2\ell}}
P_{k,\ell}^{(m-k)}\left(\frac{\cdot+\frac1t}t+c_0\right) \right\|_{L^2} \\
= \left\|\partial_x^m S(t) - \sum_{k'=0}^{[m/2]} \frac1{t^{\frac12+m-2k'}}
\sum_{\ell'=0}^{k'} P_{k'-\ell',\ell'}^{(m-k'+\ell')}\left(\frac{\cdot+\frac1t}t+c_0\right) \right\|_{L^2}
\to 0\quad\m{as } t\downarrow 0.
\end{multline*}
Thus,~\eqref{th:timem} holds with
\[
Q_{k'} = \sum_{\ell'=0}^{k'} P_{k'-\ell',\ell'}^{(\ell')},\quad \m{and in particular }\
Q_0= P_{0,0} = P_0 = Q,
\]
which concludes the proof of Theorem~\ref{thm:maintime}.
Note that we also have the simple expression
\be \label{Qun}
Q_1=-P'_1-\lambda_0(\Lambda Q)'+c_1Q''.
\ee
Recall that $P'_1\in \mathcal{S}$, but $P'_1$ is not the derivative of a Schwartz function
(see Lemma~\ref{lemma:P} for the properties of $P_1=P$),
and thus $Q_1$ is not the derivative of a Schwartz function.

\begin{remark}\label{rem:akPk}
We have seen in Remark~\ref{rem:uniquePk} that there is a degree of freedom in the definition of each $P_k$ (addition of $a_k Q'$, for any $a_k$):
changing $a_k$ will change the functions $P_{k'}$ for $k'\geq k$.
Thus, the decomposition of $S(t)$ in Proposition~\ref{prop:SHm} is not unique.
In particular, the constants in the expansions of $\lambda$, $b$ and $x$ may also depend on $a_k$.

However, the profile $Q_k$ for any $k$ is uniquely defined by the uniqueness of the minimal mass solution
(obtained in~\cite{MMR2}) and the property~\eqref{th:timem}.
This is the reason why the formulation~\eqref{th:timem} is more canonical that the statement of Proposition~\ref{prop:SHm}.
\end{remark}

\begin{remark}
As a hint of the invariance pointed out in the previous remark, one can check by explicit computations that if $P_1$ is replaced by $P_1^a=P_1+aQ'$,
and $\beta_k^a$, $\lambda_k^a$, $c_k^a$ and $Q_1^a$ are the corresponding values of $\beta_k$, $\lambda_k$, $c_k$ and $Q_1$ respectively,
then $\beta_2^a=\beta_2=2$ (which is related to the universality of the blow up rate $1/t$ of minimal mass blow up solutions)
and $\beta_3^a=\beta_3$ (which is related to the universality of the constant~$c_0$ in the statement of Theorem~\ref{thm:maintime},
since $c_0=-\frac{\beta_3}3$ from the proof of~\eqref{expansion:x} above).

In contrast, the value of $\beta_4$ is changed to $\beta_4^a=\beta_4+30a$ but,
since $\lambda_0=-\frac{\beta_3}6$ and $c_1=\frac{\beta_3^2}{36}+\frac{\beta_4}{30}$
from the proof of~\eqref{expansion:lambda} and~\eqref{expansion:x} above,
we have $\lambda_0^a=\lambda_0$ and $c_1^a=c_1+a$, and so $Q_1^a=Q_1$ from~\eqref{Qun},
which shows the universality of the profile $Q_1$ exhibited in Theorem~\ref{thm:maintime}.
\end{remark}

\section{Space estimates} \label{sec:space-estimates}

\subsection{Integral estimates on the left of the soliton}

Let $S$ be a solution of~\eqref{kdv} which blows up as $t\downarrow 0$ and such that $\|S(t)\|_{L^2}=\|Q\|_{L^2}$.
Apply Proposition~\ref{prop:SHm} to~$S$ and, without loss of generality, assume that $\epsilon=1$, $\ell_0=1$ and $x_0=0$,
after suitable time independent sign change, scaling and space translation.
Note that such a solution $S$ corresponds to the solution considered in~\emph{\ref{existsS}} of Theorem~\ref{thm:previous}.

We now prove the following space decay estimates on the left of such a minimal mass blow up solution $S$.

\begin{proposition} \label{decayR}
Let $0<a_0<\frac3{20}$.
Then, for all $0<t\leq T_0$, $R\geq t^{-1}+2t^{-\frac 12}$,
\be \label{e:L2gche}
\int_{x<-R} S^2(t,x)\,dx = \frac{\|Q\|_{L^1}^2}{8} R^{-2} + O(R^{-2-a_0}),
\ee
and, for all $m\geq 1$,
\be \label{pr:R}
\int_{x<-R} (\partial_x^m S)^2(t,x)\,dx \lesssim R^{-2-2m}.
\ee
\end{proposition}

\begin{proof}
Throughout this proof, we denote $S$ by $u$ for the simplicity of notation.
Let $m_0>1$. We consider $K>20$ so that $m_K>5m_0+5$, and $\g$ in the range~\eqref{def:gamma}.
From the choice of $\epsilon$, $\ell_0$ and~$x_0$ above, estimates~\eqref{Spar} then rewrite as
\be \label{rappel}
\lambda(t)=t+O(t^3),\quad x(t)=-\frac1t + O(t),\quad b(t)=-t^2+O(t^4).
\ee

We now define a smooth cut-off function $g$ as follows:
\[
g\equiv 1 \m{ on } (-\infty,-1], \quad
g\equiv 0 \m{ on } [-1/2,+\infty), \quad
g \m{ nonincreasing on } \RR.
\]
For $k\geq 1$ and $m\geq 0$, $t\in (0,T_0]$, $R>1$, we define
\[
I_{m,k}(t) = \int (\partial_x^m u)^2(t,x) g^k\left(\frac{x+R}{R^{1/2}}\right) dx.
\]
In particular, we have the control
\be \label{controlg}
\int_{x<-R-R^{\frac12}} (\partial_x^m u)^2(t,x)\,dx \leq I_{m,k}(t)
\leq \int_{x<-R-R^{\frac12}/2} (\partial_x^m u)^2(t,x)\,dx.
\ee

\textbf{Step 1.}
Let $R>1$ large. From~\eqref{rappel} and the continuity of $t\mapsto x(t)$,
it follows that there exists $t_R>0$ small such that $x(t_R)=-R$,
with $R = \frac1{t_R}+ O(t_R)$, and so
\be \label{eq:R}
t_R = R^{-1}+O(R^{-3}),\quad \lambda(t_R)= R^{-1}+O(R^{-3}),\quad b(t_R) = -R^{-2} + O(R^{-4}).
\ee
We claim the following \emph{a priori} estimates of the quantities $I_{m,k}$ on $[t_R,T_0]$.

\begin{lemma} \label{le:atR}
For all $k\geq 1$, $R>1$ large, $t_R\leq t\leq T_0$, $0\leq m \leq m_K$,
\be \label{eq:sbd}
I_{m,k}(t) \lesssim 1.
\ee
Let $0<a_0<\frac3{20}$.
Then, for all $k\geq 1$, $R>1$ large,
\be \label{eq:I0atR}
I_{0,k}(t_R) = \frac{\|Q\|_{L^1}^2}{8} R^{-2} + O(R^{-2-a_0}),
\ee
and, for all $1\leq m\leq m_K$,
\be \label{eq:ImatR}
I_{m,k}(t_R) \lesssim R^{-2-2m}.
\ee
\end{lemma}

\begin{proof}
From~\eqref{controlg},~\eqref{rappel} and~\eqref{eq:decomposition}, we have, for $t\in [t_R,T_0]$,
\begin{align*}
I_{m,k}(t)
&\leq \int_{x<-R-R^{\frac12}/2} (\partial_x^m u)^2(t)
\leq \int_{x<x(t)-R^{\frac12}/4} (\partial_x^m u)^2(t) \\
&\lesssim \lambda^{-2m}(t)\left( \int_{y<-R^{\frac12}/4} (\partial_y^m Q_{b(t)})^2 + \int (\partial_y^m \e)^2(t)\right).
\end{align*}
From~\eqref{rappel},~\eqref{eq:Qbj} and~\eqref{SHm}, we obtain, for $0\leq m\leq m_K$,
\[
I_{m,k}(t) \lesssim t^{-2m} \left( e^{-R^{\frac12}/4} + |b(t)|^{2-\g +2\g m} +t^{1+\g +4\g m} \right).
\]
Using~\eqref{rappel} and~\eqref{eq:R}, we first find, for all $t_R\leq t\leq T_0$, all $0\leq m\leq m_K$,
\be \label{pour2}
I_{m,k}(t) \lesssim R^{2m} e^{-R^{\frac12}/4} + t^{2(2-\g) +2(2\g-1)m} + t^{1+\g +2(2\g-1)m},
\ee
and so $I_{m,k}(t) \lesssim 1$, provided that $\g\geq\frac12$, which is granted by the choice~\eqref{def:gamma} of $\g$.
Thus,~\eqref{eq:sbd} is proved.
Second, we evaluate~\eqref{pour2} at $t=t_R$ and obtain, from~\eqref{eq:R}, for all $1\leq m\leq m_K$,
\[
I_{m,k}(t_R) \lesssim R^{2m} e^{-R^{\frac12}/4} + R^{-2(2-\g) -2(2\g-1)m} + R^{-(1+\g) -2(2\g-1)m}
\lesssim R^{- (1+\g) - 2(2\g -1)m}.
\]
Thus, taking $\g=1$,~\eqref{eq:ImatR} is proved.
To prove the more precise asymptotics~\eqref{eq:I0atR}, we rely on~\eqref{eq:PQ} and the $L^2$ norm conservation
for the minimal mass solutions~\eqref{minmass}, which give
\be \label{eq:normeps}
\int \e^2(t_R) = -2 b(t_R) \int PQ + O(|b(t_R)|^{\frac 12(3-\g)})
= \frac{\|Q\|_{L^1}^2}{8} R^{-2} + O(R^{-3+\g}).
\ee
Moreover, since $x(t_R)=-R$ by definition of $t_R$, we obtain, with $y=\frac{x+R}{\lambda(t_R)}$,
\begin{align*}
I_{0,k}(t_R) &= \int u^2(t_R,x)g^k\left( \frac{x+R}{R^{1/2}} \right) dx \\
&= \int Q_{b(t_R)}^2(y) g^k\left( \frac{\lambda(t_R)}{R^{1/2}}y \right) dy
+2 \int [Q_{b(t_R)}\e(t_R)](y) g^k\left( \frac{\lambda(t_R)}{R^{1/2}}y \right) dy \\
&\quad +\int \e^2(t_R,y) g^k\left( \frac{\lambda(t_R)}{R^{1/2}}y \right) dy \\
&= A_1+A_2+A_3.
\end{align*}
From the definition of $g$,~\eqref{eq:Qbj} and~\eqref{eq:R}, we first get
\[
A_1 \leq \int_{y<-\frac12 \frac{R^{1/2}}{\lambda(t_R)}} Q_{b(t_R)}^2(y)\,dy
\lesssim e^{-\frac12 \frac{R^{1/2}}{\lambda(t_R)}} + |b(t_R)|^{2-\g}
\lesssim e^{-\frac14 R^{3/2}} + R^{-4+2\g}.
\]
Using also~\eqref{eq:normeps} and the Cauchy--Schwarz inequality, we get
\[
|A_2| \lesssim \sqrt{A_1}\sqrt{A_3} \lesssim \left( e^{-\frac18 R^{3/2}} + R^{-2+\g} \right) \|\e(t_R)\|_{L^2}
\lesssim e^{-\frac18 R^{3/2}} + R^{-3+\g}.
\]
Finally, to estimate $A_3$, we first notice that, from~\eqref{eq:normeps},
\[ A_3 \leq \int \e^2(t_R) \leq \frac{\|Q\|_{L^1}^2}{8} R^{-2} + CR^{-3+\g},
\]
for some $C>0$.
To obtain the lower bound, we rely on~\eqref{eq:R} and~\eqref{Sloc} applied with $t=t_R$ to obtain
\[
\int_{y>-10 R^{3/2}} \e^2(t_R) \lesssim R^{-7}.
\]
Thus, we have, using again~\eqref{eq:normeps},
\[
A_3 \geq \int_{y<-\frac{R^{1/2}}{\lambda(t_R)}} \e^2(t_R) \geq \int_{y<-2R^{3/2}} \e^2(t_R)
\geq \frac{\|Q\|_{L^1}^2}{8} R^{-2} - CR^{-3+\g} -CR^{-7}.
\]
Gathering the above estimates of $A_1$, $A_2$ and $A_3$, and taking $\frac{17}{20}<\g<1$,
we obtain~\eqref{eq:I0atR} with $a_0=1-\g$,
which concludes the proof of Lemma~\ref{le:atR}.
\end{proof}

\textbf{Step 2.}
We estimate the time derivatives of $I_{m,k}$ as follows.

\begin{lemma} \label{le:Imk}
For all $k\geq 1$, $t\in (0,T_0]$,
\be \label{eq:I0k}
\left| \frac d{dt} I_{0,3k+3} \right| \lesssim R^{-1/2}\left(I_{0,k}+I_{1,k}\right).
\ee
For all $k\geq 1$, $t\in (0,T_0]$,
\be \label{eq:I1k}
\left| \frac d{dt} I_{1,3k+3} \right| \lesssim R^{-1/2}\left(I_{1,k}+I_{2,k}\right)
+I_{0,k}^{\frac34} \left(I_{0,k} + I_{1,k}\right)\left(I_{0,k}+I_{1,k} + I_{2,k} \right)^{\frac14} I_{1,k}.
\ee
For all $2\leq m\leq m_K-1$, $k\geq 1$, $t\in (0,T_0]$,
\be \label{eq:Imk}
\begin{aligned}
\left| \frac d{dt} I_{m,3k+3} \right|
&\lesssim R^{-1/2}\left(I_{m,k}+I_{m+1,k}\right) \\
&\quad +I_{0,k}\left(I_{0,k} + I_{1,k} \right) I_{m,k} +\left(I_{0,k} + I_{1,k} \right)^{\frac 32} \left(I_{1,k} + I_{2,k} \right)^{\frac12} I_{m,k} \\
&\quad +I_{m,k}^{\frac 12} \sum_{\bar m\in \Sigma_m} \left( \prod_{j=1}^4 (I_{m_j,k}+I_{m_j+1,k})^{\frac12} I_{m_5,k}^{\frac 12}\right),
\end{aligned}
\ee
where, for $m\geq 2$,
\[
\Sigma_m =\left\{\bar m=(m_1,\ldots,m_5), \ |\bar m| = \sum_{k=1}^5 m_k= m+1,\ 0\leq m_1\leq \cdots \leq m_5 \leq m-1\right\}.
\]
\end{lemma}

\begin{proof}
Observe that, for $k'\geq k$, $I_{m,k'}\leq I_{m,k}$.
We will use this property to simplify several expressions below.
For the sake of brevity, we will also omit the variable of $g$.
Therefore, using~\eqref{kdv} and integrations by parts, we get, for all $0\leq m\leq m_K-1$ and $k\geq1$,
\begin{align} \label{eq:dIm}
&\frac d{dt} I_{m,3k+3} = 2 \int (\partial_x^m \partial_t u) (\partial_x^m u) g^{3k+3} \nonumber \\
&= -\frac3{R^{1/2}} \int (\partial_x^{m+1} u)^2 (g^{3k+3})'
+ \frac1{R^{3/2}} \int (\partial_x^m u)^2 (g^{3k+3})''' - 2 \int \partial_x^{m+1} (u^5) (\partial_x^m u) g^{3k+3}.
\end{align}

\emph{Case $m=0$.} Using again an integration by parts,~\eqref{eq:dIm} reads as
\[
\frac d{dt} I_{0,3k+3} = -\frac3{R^{1/2}} \int u_x^2 (g^{3k+3})'
+ \frac1{R^{3/2}} \int u^2 (g^{3k+3})''' +\frac1{R^{1/2}} \frac53 \int u^6 (g^{3k+3})'.
\]
Note that, for $k\geq 1$,
\[
|(g^{3k+3})'|\lesssim g^{3k+2},\quad |(g^{3k+3})'''|\lesssim g^{3k}.
\]
Thus,
\[
\left| \int u^6 (g^{3k+3})'\right| \lesssim \int u^6 g^{3k+2} \lesssim \|u^2 g^{k+1}\|_{L^\infty}^2 \int u^2 g^k,
\]
and since $\partial_x(u^2 g^{k+1}) = 2 u_x u g^{k+1} + R^{-1/2} u^2 (g^{k+1})'$, $|(g^{k+1})'|\lesssim g^k$, we have
\be \label{ginfty}
\|u^2 g^{k+1}\|_{L^\infty} \lesssim \int \left| u u_x g^{k+1} \right|+ R^{-1/2} \int \left| u^2 (g^{k+1})'\right|
\leq I_{0,{k+1}}^{\frac12} I_{1,{k+1}}^{\frac12} + R^{-1/2} I_{0,k}.
\ee
Therefore, we obtain
\[
\left|\frac d{dt} I_{0,3k+3}\right| \lesssim R^{-1/2}\left[ I_{1,3k+2} + R^{-1} I_{0,3k}
+ \left(I_{0,{k+1}}^{\frac12} I_{1,{k+1}}^{\frac12} + R^{-1/2} I_{0,k}\right)^2 I_{0,k}\right].
\]
Using $I_{0,k'} \leq \int Q^2$ for all $k'\geq 1$, we obtain the simplified estimate~\eqref{eq:I0k}.

\emph{Case $m=1$.} In this case,~\eqref{eq:dIm} reads as
\[
\frac d{dt} I_{1,3k+3} = -\frac3{R^{1/2}} \int (\partial_x^{2} u)^2 (g^{3k+3})'
+ \frac1{R^{3/2}} \int (\partial_x u)^2 (g^{3k+3})''' - 2 \int \partial_x^{2} (u^5) (\partial_x u) g^{3k+3}.
\]
Note that, integrating by parts,
\begin{align*}
\int \partial_x^{2} (u^5) (\partial_x u) g^{3k+3}
&= 5 \int (\partial_x^2 u) (\partial_x u) u^4 g^{3k+3}
 + 20 \int \left( \partial_x u\right) ^3 u^3 g^{3k+3} \\
&= - \frac5{2R^{1/2}} \int (\partial_x u)^2 u^4 (g^{3k+3})' + 10 \int \left( \partial_x u\right) ^3 u^3 g^{3k+3}.
\end{align*}
First, by~\eqref{ginfty},
\begin{align*}
\left| \int (\partial_x u)^2 u^4 (g^{3k+3})' \right|
&\lesssim \int (\partial_x u)^2 u^4 g^{3k+2} \lesssim \|u^2 g^{k+1}\|_{L^\infty}^2 I_{1,k} \\
&\lesssim \left(I_{0,{k+1}}^{\frac12} I_{1,{k+1}}^{\frac12} + R^{-1/2} I_{0,k}\right)^2 I_{1,k}
\lesssim I_{0,k}\left(I_{0,k} + I_{1,k} \right) I_{1,k}.
\end{align*}
Second,
\begin{align*}
\left| \int (\partial_x u)^3 u^3 g^{3k+3} \right|
&\lesssim \|u^2 g^{k+1}\|_{L^\infty}^{3/2}\|(\partial_x u)^2 g^{k+1}\|_{L^\infty}^{1/2} I_{1,k+1} \\
&\lesssim \left(I_{0,{k+1}}^{\frac12} I_{1,{k+1}}^{\frac12} + R^{-1/2} I_{0,k}\right)^{\frac32}
\left(I_{1,{k+1}}^{\frac12} I_{2,{k+1}}^{\frac12} + R^{-1/2} I_{1,k}\right)^{\frac12} I_{1,k+1} \\
&\lesssim I_{0,k}^{\frac 34} \left(I_{0,k} + I_{1,k} \right)\left(I_{1,k} + I_{2,k} \right)^{\frac14} I_{1,k}.
\end{align*}
Thus,~\eqref{eq:I1k} follows.

\emph{Case $m\geq 2$.} In this case, we decompose the nonlinear term in~\eqref{eq:dIm} as
\begin{align*}
\int \partial_x^{m+1} (u^5) (\partial_x^m u) g^{3k+3}
&= 5 \int (\partial_x^{m+1}u) (\partial_x^m u) u^4 g^{3k+3} + 20 m \int (\partial_x^m u)^2 (\partial_x u) u^3 g^{3k+3} \\
&\quad + \int \left(\partial_x^{m+1} (u^5) - 5 (\partial_x^{m+1}u) u^4- 20 m (\partial_x^m u) (\partial_x u) u^3\right)(\partial_x^m u) g^{3k+3},
\intertext{which gives, after integrating by parts,}
\int \partial_x^{m+1} (u^5) (\partial_x^m u) g^{3k+3}
&= -\frac5{2R^{1/2}} \int (\partial_x^m u)^2 u^4 (g^{3k+3})' + 10 (2m -1) \int (\partial_x^m u)^2 (\partial_x u) u^3 g^{3k+3} \\
&\quad + \int \left(\partial_x^{m+1} (u^5) - 5 (\partial_x^{m+1}u) u^4- 20 m (\partial_x^m u) (\partial_x u) u^3\right)(\partial_x^m u) g^{3k+3}.
\end{align*}
First, by~\eqref{ginfty},
\begin{align*}
\left| \int (\partial_x^m u)^2 u^4 (g^{3k+3})' \right|
&\lesssim \int (\partial_x^m u)^2 u^4 g^{3k+2} \lesssim \|u^2 g^{k+1}\|_{L^\infty}^2 I_{m,k} \\
&\lesssim \left(I_{0,{k+1}}^{\frac12} I_{1,{k+1}}^{\frac12} + R^{-1/2} I_{0,k}\right)^2 I_{m,k}
\lesssim I_{0,k}\left(I_{0,k} + I_{1,k} \right) I_{m,k}.
\end{align*}
Second,
\begin{align*}
\left| \int (\partial_x^m u)^2 (\partial_x u) u^3 g^{3k+3} \right|
&\lesssim \|u^2 g^{k+1}\|_{L^\infty}^{3/2}\|(\partial_x u)^2 g^{k+1}\|_{L^\infty}^{1/2} I_{m,k+1} \\
&\lesssim \left(I_{0,{k+1}}^{\frac12} I_{1,{k+1}}^{\frac12} + R^{-1/2} I_{0,k}\right)^{\frac32}
\left(I_{1,{k+1}}^{\frac12} I_{2,{k+1}}^{\frac12} + R^{-1/2} I_{1,k}\right)^{\frac12} I_{m,k+1} \\
&\lesssim \left(I_{0,k} + I_{1,k} \right)^{\frac 32} \left(I_{1,k} + I_{2,k} \right)^{\frac12} I_{m,k}.
\end{align*}
Now let
\[
\Sigma_m =\left\{\bar m=(m_1,\ldots,m_5), \ |\bar m| = \sum_{k=1}^5 m_k= m+1,\ 0\leq m_1\leq \cdots \leq m_5 \leq m-1\right\}.
\]
Note that
\[
\left|\partial_x^{m+1} (u^5) - 5 (\partial_x^{m+1}u) u^4- 20 m (\partial_x^m u) (\partial_x u) u^3\right|
\lesssim \sum_{\bar m \in \Sigma_m} \left(\prod_{j=1}^5 | \partial_x^{m_j} u| \right),
\]
and thus
\begin{align*}
&\left| \int \left(\partial_x^{m+1} (u^5) - 5 (\partial_x^{m+1}u) u^4- 20 m (\partial_x^m u) (\partial_x u) u^3\right)(\partial_x^m u) g^{3k+3}\right| \\
&\lesssim \sum_{\bar m \in \Sigma_m} \int \left(\prod_{j=1}^5 | \partial_x^{m_j} u|\right) |\partial_x^m u| g^{3k+3}
\lesssim \sum_{\bar m \in \Sigma_m} \prod_{j=1}^4 \|( \partial_x^{m_j} u)^2 g^{k+1} \|_{L^\infty}^{\frac12} I_{m_5,k+1}^{\frac12}I_{m,k+1}^{\frac12} \\
& \lesssim I_{m,k+1}^{\frac12} \sum_{\bar m \in \Sigma_m} \left( \prod_{j=1}^4
\left(I_{m_j,k+1}^{\frac12} I_{m_j+1,k+1}^{\frac12} + R^{-1/2} I_{m_j,k}\right)^{\frac12} I_{m_5,k+1}^{\frac12}\right),
\end{align*}
using~\eqref{ginfty} on $\partial_x^{m_j} u$ with $0\leq m_j\leq m-1$.
In conclusion, we gather and simplify the above estimates to obtain~\eqref{eq:Imk},
which finishes the proof of Lemma~\ref{le:Imk}.
\end{proof}

\textbf{Step 3.}
Now, we apply several times Lemma~\ref{le:Imk}, combined with Lemma~\ref{le:atR},
to improve the \emph{a~priori} estimates~\eqref{eq:sbd} in view of the optimal bounds given by~\eqref{eq:I0atR}--\eqref{eq:ImatR}.
Let $q:k\mapsto 3k+3$ and $p(j)=q^j(1)=\frac12 ( 5\cdot3^j - 3)$.
Let $R>1$ large and consider any $t_R\leq t\leq T_0$.

First, integrating~\eqref{eq:I0k} with $k=1$ on $[t_R,t]$ and using~\eqref{eq:sbd} and~\eqref{eq:I0atR}, we obtain
\[
I_{0,p(1)}(t)\lesssim R^{-1/2} + I_{0,p(1)}(t_R) \lesssim R^{-1/2}.
\]
Second, we integrate~\eqref{eq:I1k} with $k=p(1)$ on $[t_R,t]$ and use~\eqref{eq:sbd},~\eqref{eq:ImatR}
and the previous estimate to get $I_{1,p(2)}(t)\lesssim R^{-3/8} + I_{1,p(2)}(t_R) \lesssim R^{-3/8}$.
Inserting these estimates again in~\eqref{eq:I1k}, with $k=p(2)$, and proceeding similarly, we obtain
\[
I_{1,p(3)}(t)\lesssim R^{-1/2} + I_{1,p(3)}(t_R) \lesssim R^{-1/2}.
\]
Note that, due to the two above estimates, we may now simplify~\eqref{eq:I1k} and~\eqref{eq:Imk}, for $k\geq p(3)$, as
\be \label{eq:I1kb}
\left| \frac d{dt} I_{1,3k+3} \right| \lesssim R^{-1/2}\left(I_{1,k}+I_{2,k}\right)
\ee
and, for $2\leq m\leq m_K-1$,
\be \label{eq:Imkb}
\left| \frac d{dt} I_{m,3k+3} \right| \lesssim R^{-1/2}\left(I_{m,k}+I_{m+1,k}\right)
+I_{m,k}^{\frac 12} \sum_{\bar m\in \Sigma_m} \left( \prod_{j=1}^4 (I_{m_j,k}+I_{m_j+1,k})^{\frac12} I_{m_5,k}^{\frac 12}\right).
\ee
Since
\[
\Sigma_2 = \Bigl\{ (0,0,1,1,1) \Bigl\}\quad \m{and}\quad \Sigma_3 = \Bigl\{ (0,1,1,1,1) , (0,0,1,1,2) , (0,0,0,2,2) \Bigl\},
\]
we obtain similarly, using~\eqref{eq:Imkb} and~\eqref{eq:ImatR},
\[
I_{2,p(4)}(t)\lesssim R^{-1/2},\quad \m{then}\quad I_{3,p(5)}(t)\lesssim R^{-1/2}.
\]
Now note that, if $\bar m\in\Sigma_m$ with $m\geq 4$, then $m_4\leq m-2$, and so we obtain by induction,
for all $0\leq m\leq m_K-1$, $I_{m,p(m+2)}(t)\lesssim R^{-1/2}$.
In other words, we have obtained the first uniform decay estimates, for all $0\leq m\leq m_K-1$:
\[
I_{m,p(m_K+1)}(t) \lesssim R^{-1/2}.
\]

Inserting these estimates in~\eqref{eq:I0k},~\eqref{eq:I1kb} and~\eqref{eq:Imkb}, with $k=p(m_K+1)$,
we obtain $I_{m,p(m_K+2)}(t)\lesssim R^{-1}$, for all $0\leq m\leq m_K-2$.
Boot-strapping again, we obtain next, for all $0\leq m\leq m_K-3$, $I_{m,p(m_K+3)}(t)\lesssim R^{-3/2}$.
With one more step, we obtain the following improved uniform decay estimates, for all $0\leq m\leq m_K-4$:
\[
I_{m,p(m_K+4)}(t) \lesssim R^{-2}.
\]
Note that, in view of~\eqref{eq:I0atR}, this is the optimal bound for $m=0$,
and that we may continue for $m\geq 1$.
Using~\eqref{eq:I1kb} and~\eqref{eq:Imkb}, we obtain similarly, for all $1\leq m\leq m_K-8$,
\[
I_{m,p(m_K+8)}(t)\lesssim R^{-4},
\]
which is the optimal bound for $m=1$ according to~\eqref{eq:ImatR}.
To obtain the optimal bounds for $m\geq 2$, we prove the following proposition by induction,
using now only the estimates~\eqref{eq:Imkb} and~\eqref{eq:ImatR}.

\begin{proposition} \label{prop:inductionIm}
Let $1\leq m\leq m_0$. Then we have, for all $t\in [t_R,T_0]$,
for all $0\leq\ell\leq m$,
\be \label{Ismall}
I_{\ell,p(m_K+4m+4)}(t) \lesssim R^{-2-2\ell};
\ee
and, for all $m+1\leq \ell\leq m_K-4m-4$,
\be \label{Ilarge}
I_{\ell,p(m_K+4m+4)}(t) \lesssim R^{-2-2m}.
\ee
\end{proposition}

\begin{proof}
We proceed by induction on $m\geq1$, noticing that the proposition holds for $m=1$
by the estimates obtained above.
Now let $m\geq 2$, assume that the proposition holds for $m-1$, and we prove that it holds for $m$.
In other words, we assume
\begin{align}
\forall 0\leq\ell\leq m-1,\quad I_{\ell,p(m_K+4m)}(t) &\lesssim R^{-2-2\ell}; \label{ind:Ismall} \\
\forall m\leq \ell\leq m_K-4m,\quad I_{\ell,p(m_K+4m)}(t) &\lesssim R^{-2m}; \label{ind:Ilarge}
\end{align}
and we prove~\eqref{Ilarge}, for all $m\leq \ell\leq m_K-4m-4$.
Note that~\eqref{Ilarge} is relevant provided that $m_K>5m+5$,
which is granted since we assume $m\leq m_0$ and we chose $K>20$ such that $m_K>5m_0+5$.

Let $m\leq \ell\leq m_K-4m$.
From~\eqref{eq:Imkb} with $k\geq p(m_K+4m)$, we obtain
\[
\left| \frac d{dt} I_{\ell,3k+3} \right| \lesssim R^{-1/2}\left(I_{\ell,k}+I_{\ell+1,k}\right)
+I_{\ell,k}^{\frac 12} \sum_{\bar \ell\in \Sigma_\ell} P_{\bar\ell},\quad \m{with}\quad
P_{\bar\ell} = \prod_{j=1}^4 (I_{\ell_j,k}+I_{\ell_j+1,k})^{\frac12} I_{\ell_5,k}^{\frac 12}.
\]
Let $\bar\ell\in\Sigma_\ell$. To estimate $P_{\bar\ell}$, we consider two cases:
\begin{itemize}
\item if $\ell_5\geq m$ then, using~\eqref{ind:Ilarge} and $I_{\ell',k}(t)\lesssim R^{-2}$ for all $0\leq \ell'\leq m_K-4m$,
we obtain
\[
P_{\bar\ell} \lesssim R^{-4-m};
\]
\item if $\ell_5\leq m-1$ then, using~\eqref{ind:Ismall} and also~\eqref{ind:Ilarge} in the case $\ell_j=m-1$, we obtain,
since $|\bar\ell| = \ell+1\geq m+1$,
\[
P_{\bar\ell} \lesssim \prod_{j=1}^5 R^{-1-\ell_j} = R^{-5-|\bar\ell|} \lesssim R^{-5-(m+1)} = R^{-6-m}.
\]
\end{itemize}
Using again~\eqref{ind:Ilarge}, we find, for all $k\geq p(m_K+4m)$,
\[
\left| \frac d{dt} I_{\ell,3k+3} \right| \lesssim R^{-1/2}\left(I_{\ell,k}+I_{\ell+1,k}\right) + R^{-4-2m}.
\]
As before, we integrate this estimate on $[t_R,t]$, use~\eqref{eq:ImatR},
and boot-strap the new estimates of $I_{\ell,3k+3}$ three more times to obtain, for all $m\leq \ell\leq m_K-4m-4$,
\[
I_{\ell,p(m_K+4m+4)}(t) \lesssim R^{-2-2m},
\]
which concludes the proof of the proposition.
\end{proof}

We now may apply Proposition~\ref{prop:inductionIm} with $m=m_0$.
Denoting $k_0=p(m_K+4m_0+4)$, the estimate~\eqref{Ismall} now reads,
for all $t_R\leq t\leq T_0$, $1\leq m\leq m_0$,
\be \label{resultIm}
I_{m,k_0}(t)\lesssim R^{-2-2m}.
\ee
Moreover, integrating~\eqref{eq:I0k} on $[t_R,t]$ with $t_R\leq t\leq T_0$,
we get $|I_{0,k_0}(t)-I_{0,k_0}(t_R)|\lesssim R^{-2-\frac 12}$.
Thus, from~\eqref{eq:I0atR}, we obtain, with $0<a_0<\frac3{20}$,
\be \label{resultI0}
I_{0,k_0}(t) = \frac{\|Q\|_{L^1}^2}{8} R^{-2} + O(R^{-2-a_0}).
\ee

\textbf{Step 4.}
To conclude the proof of Proposition~\ref{decayR}, it is enough to prove
how estimates~\eqref{controlg} and~\eqref{resultIm}--\eqref{resultI0} imply~\eqref{pr:R}--\eqref{e:L2gche}.
Let $1\leq m\leq m_0$ and $0<t\leq T_0$.
Let $R\geq t^{-1}+2t^{-\frac 12}$.
Let $R/2<R'<R$ be such that $R=R'+(R')^{\frac 12}$,
namely $R' = R - \frac{\sqrt{4R+1}-1}2 = R + O(R^{\frac12})$.
Then, by~\eqref{eq:R},
\[
t^{-1}_{R'}+2t^{-\frac 12}_{R'} \geq R'+(R')^{\frac 12}=R \geq t^{-1}+2t^{-\frac 12},
\]
and so $t\geq t_{R'}$.
Thus,~\eqref{resultIm} and~\eqref{controlg} applied to $R'$ and $t_{R'}\leq t\leq T_0$ imply~\eqref{pr:R}, since
\[
\int_{x<-R} (\partial_x^m u)^2(t) = \int_{x<-R'-(R')^{\frac12}} (\partial_x^m u)^2(t) \leq I_{m,k_0}(t)
\lesssim (R')^{-2-2m} \lesssim R^{-2-2m}.
\]
Let $0<a_0<\frac3{20}$.
We obtain similarly, from~\eqref{resultI0} and~\eqref{controlg} applied to $R'$ and $t_{R'}\leq t\leq T_0$,
\begin{align*}
\int_{x<-R} u^2(t) &= \int_{x<-R'-(R')^{\frac12}} u^2(t) \leq I_{0,k_0}(t)
\leq \frac{\|Q\|_{L^1}^2}{8} (R')^{-2} +C(R')^{-2-a_0} \\
&\leq \frac{\|Q\|_{L^1}^2}{8} R^{-2} + C R^{-2-\frac12} + C R^{-2-a_0}
\leq \frac{\|Q\|_{L^1}^2}{8} R^{-2} + C R^{-2-a_0},
\end{align*}
for some $C>0$.
Finally, we have similarly $t\geq t_{R''}$, where $R/2<R''<R$ is such that $R=R''+\frac 12 (R'')^{\frac 12}$.
Thus,~\eqref{resultI0} and~\eqref{controlg} applied to $R''$ and $t_{R''}\leq t\leq T_0$ give
\begin{align*}
\int_{x<-R} u^2(t) = \int_{x<-R''-\frac12 (R'')^{\frac12}} u^2(t) \geq I_{0,k_0}(t)
&\geq \frac{\|Q\|_{L^1}^2}{8} (R'')^{-2} -C(R'')^{-2-a_0} \\
&\geq \frac{\|Q\|_{L^1}^2}{8} R^{-2} - C R^{-2-a_0}.
\end{align*}
Hence,~\eqref{e:L2gche} is proved, which finishes to prove Proposition~\ref{decayR}.
\end{proof}

\subsection{Sharp pointwise asymptotics on the left}

Possibly taking a smaller $T_0$, we obtain the following.

\begin{proposition} \label{prop:ptgche}
Let $0<a<\frac 1{20}$.
For all $t\in (0,T_0]$, for all $R\geq t^{-1}+1$,
\be \label{eq:ptgche}
S(t,-R) = -\frac{\|Q\|_{L^1}}{2} R^{-\frac32}+ O(R^{-\frac32 - a}),
\ee
and, for all $m\geq 1$,
\be \label{eq:diffgche}
\left| \partial_x^m S(t,-R)\right| \lesssim R^{-\frac32 - m}.
\ee
\end{proposition}

To obtain such pointwise estimates from the integral estimates given
by Proposition~\ref{decayR}, we will rely on the following elementary observation.

\begin{lemma} \label{simple}
Let $0<a_0<1$ and $c_0>0$. Assume that a function $v$ satisfies, for all $R$ large,
\[
\int_{x<-R} v^2(x)\, dx = c_0 R^{-2} + O(R^{-2-a_0}),
\]
and
\[
\int_{x<-R} (v')^2(x)\, dx \lesssim R^{-4}.
\]
Then, for all $R$ large,
\[
v^2(-R) = 2c_0 R^{-3} +O(R^{-3-a}),
\]
with $a=\frac13 a_0>0$.
\end{lemma}

\begin{proof}
Fix $R$ large.
Let $0<c<1$ to be chosen. Then
\begin{align*}
\int_{[-R-R^{1-c},-R]}v^2(x)\, dx
&= c_0R^{-2}-c_0(R+R^{1-c})^{-2} + O(R^{-2-a_0}) \\
&= 2c_0 R^{-2-c} + O(R^{-2-2c}) + O(R^{-2-a_0}).
\end{align*}
From the mean value theorem, it follows that there exists $x_R\in [-R-R^{1-c},-R]$ such that
\[
v^2(x_R) = 2c_0 R^{-3} + O(R^{-3-c}) + O(R^{-3-a_0+c}).
\]
Next, by the Cauchy--Schwarz inequality,
\begin{align*}
|v^2(-R)-v^2(x_R)|
&\lesssim \int_{[-R-R^{1-c},-R]} |v(x)| |v'(x)|\, dx \\
&\lesssim \left(\int_{[-R-R^{1-c},-R]} v^2(x) \,dx \right)^{\frac 12}
	\left(\int_{(-\infty,-R]} (v')^2(x) \,dx \right)^{\frac 12} \\[2mm]
& \lesssim R^{-\frac 12 (2+c)}R^{-2} = R^{-3- \frac c2}.
\end{align*}
Thus,
\[
v^2(-R) = 2c_0 R^{-3} + O(R^{-3-\frac c2}) + O(R^{-3-a_0+c}).
\]
To conclude, it is enough to choose $c$ such that $c<a_0$ and $c>0$.
With $c=\frac23 a_0$, we reach the desired estimate.
\end{proof}

\begin{proof}[Proof of Proposition~\ref{prop:ptgche}]
The proof of~\eqref{eq:ptgche}--\eqref{eq:diffgche} proceeds in four steps,
making the distinction between the regions $R\geq t^{-1}+2 t^{-\frac 12}$ and $R\in [t^{-1}+1,t^{-1}+2t^{-\frac 12}]$.

\textbf{Step 1.}
First, we prove that, for all $0<a<\frac1{20}$, $t\in (0,T_0]$ and $R\geq t^{-1}+2 t^{-\frac 12}$,
\be \label{eq:ptgchesquare}
S^2(t,-R) = \frac{\|Q\|_{L^1}^2}{4}R^{-3}+ O(R^{-3-a}).
\ee
Indeed,~\eqref{e:L2gche} with $0<a_0<\frac3{20}$ implies
\[
\int_{x<-R} S^2(t,x)\,dx = \frac{\|Q\|_{L^1}^2}{8}R^{-2} + O(R^{-2-a_0}),
\]
and~\eqref{pr:R} with $m=1$ implies
\[
\int_{x<-R} S_x^2(t,x)\, dx \lesssim R^{-4}.
\]
Applying Lemma~\ref{simple} with $c_0=\frac18 \|Q\|_{L^1}^2$,
we obtain~\eqref{eq:ptgchesquare} for all $0<a<\frac1{20}$.

\medskip

\textbf{Step 2.}
Let $t\in (0,T_0]$ and $m\geq 1$.
For all $R\geq t^{-1}+2t^{-\frac12}$,~\eqref{eq:diffgche} holds.
Indeed, from~\eqref{pr:R}, we have, using the Cauchy--Schwarz inequality,
\begin{align*}
(\partial_x^m S)^2(t,-R) &\leq \|(\partial_x^m S)^2(t)\|_{L^\infty(x<-R)} \\
&\lesssim \left(\int_{x<-R} (\partial_x^m S)^2(t,x)\,dx \right)^{\frac12}
\left(\int_{x<-R} (\partial_x^{m+1} S)^2(t,x)\,dx \right)^{\frac12} \\
&\lesssim R^{-1-m} R^{-1-(m+1)} = R^{-3-2m}.
\end{align*}

\textbf{Step 3.}
Let $t\in (0,T_0]$, $m\geq 1$, and $R\in [t^{-1}+1,t^{-1}+2t^{-\frac 12}]$.
We now prove~\eqref{eq:ptgche}--\eqref{eq:diffgche} for such an $R$, by a straightforward calculation.
First, we set
\[
y_R=\frac {-R-x(t)}{\lambda(t)},\quad \m{which satisfies, by~\eqref{rappel},}\quad
-4 t^{-\frac 32}<y_R<-\frac 12 t^{-1}.
\]
Since $b(t)= -t^2 + O(t^4)$ and $\frac{17}{20}<\g \leq 1$, the definition~\eqref{def:locprofile} of $Q_b$
gives $Q_{b(t)}(y_R)=Q(y_R)+R_{b(t)}(y_R)$.
In particular, we deduce, from the definition of $R_b$ and the properties of the $P_k$
given in Lemma~\ref{lemma:Qbt},
\[
|Q_{b(t)}(y_R)-b(t)P_1(y_R)| \lesssim Q(y_R) + \sum_{k=2}^{K} |b(t)|^k |P_k(y_R)|
\lesssim e^{-\frac 12 t^{-1}} + \sum_{k=2}^K t^{2k} t^{-\frac32(k-1)} \lesssim t^{\frac 52}.
\]
But, from~\eqref{eq:Pkinf} in Remark~\ref{rem:Pk}, we have
\[ \left| P_1(y_R) -\frac12 \|Q\|_{L^1} \right| \lesssim e^{-|y_R|/2} \lesssim e^{-\frac 12 t^{-1}},
\]
and so
\[
\left| Q_{b(t)}(y_R) + \frac 12 \|Q\|_{L^1} t^2 \right| \lesssim t^{\frac 52}.
\]
Moreover, by~\eqref{Sloc} and~\eqref{SHm}, we obtain, for all $y>-t^{\frac85}$,
\be \label{eq:pteps}
|\e(t,y)|\lesssim \left(\int_{y>-t^{\frac85}} \e^2(t,y)\,dy \right)^{\frac14}
\left( \int \e_y^2(t,y)\,dy \right)^{\frac14} \lesssim t^{2+\frac 54 \g},
\ee
and in particular $|\e(t,y_R)| \lesssim t^{2+\frac 54 \g}$.
Since $S(t,-R) = \lambda^{-1/2}(t)[Q_{b(t)}(y_R) +\e(t,y_R)]$ from~\eqref{eq:decomposition},
we obtain, using $\lambda(t)=t+O(t^3)$ and $t = R^{-1} +O(R^{-\frac 32})$,
\[
S(t,-R) = -\frac 12 \|Q\|_{L^1} R^{-\frac 32} +O(R^{-2}),
\]
and \emph{a fortiori}~\eqref{eq:ptgche}.
To prove~\eqref{eq:diffgche}, we first notice that, from~\eqref{eq:Qbj},
\[
|\partial_y^m Q_{b(t)}(y_R)|\lesssim e^{-|y_R|/2}+ |b(t)|^{1+m}\lesssim t^{2+2m},
\]
and then, from~\eqref{SHm},
\[
\|\partial_y^m \e(t)\|_{L^\infty}\lesssim \|\e(t)\|_{\dot H^m}^{\frac 12} \|\e(t)\|_{\dot H^{m+1}}^{\frac 12}
\lesssim t^{\frac 12 (1+3\g) + 2\g m}.
\]
Using again $\lambda(t)=t+O(t^3)$, $t = R^{-1} +O(R^{-\frac 32})$ and~\eqref{eq:decomposition}, the two last estimates imply
\[
\left| \partial_x^m S(t,-R)\right| \lesssim
\lambda^{-m-\frac 12}(t) \left(|\partial_y^m Q_{b(t)}(y_R)|+\|\partial_y^m \e(t)\|_{L^\infty}\right) \lesssim
R^{-\frac 32 -m} + R^{-\frac32\g - (2\g-1)m},
\]
which leads to~\eqref{eq:diffgche} by taking $\g=1$.

\medskip

\textbf{Step 4.}
To conclude the proof of Proposition~\ref{prop:ptgche}, we note that, from the results of Steps~2 and~3,
we just need to prove that~\eqref{eq:ptgche} holds for all $R\geq R(t)$, where $R(t)=t^{-1}+2t^{-1/2}$.
To do so, we first notice that, taking $T_0$ small enough, $S(t,-R)$ cannot vanish for $R\geq R(t)$ from~\eqref{eq:ptgchesquare}.
But, from Step~3, taking possibly $T_0$ even smaller, $S(t,-R(t))$ is negative, and so $S(t,-R)$ is also negative for $R\geq R(t)$.
We can conclude that~\eqref{eq:ptgche} holds for all $R\geq R(t)$ from~\eqref{eq:ptgchesquare}.
\end{proof}

\subsection{Pointwise exponential decay on the right}

We prove~\eqref{th:ptwiser}, which is a direct consequence of the estimates of Proposition~\ref{prop:SHm}.
Let $m\geq 0$ and $K>20$ such that $0\leq m\leq m_K-1$, where $m_K=[K/2]-1$.
First, by~\eqref{SHm}, we have, for some $B>1$, for all $t\in (0,T_0]$, $y\in\RR$,
\[
e^{\frac y{B}} (\partial_y^m\e(t,y))^2 =\left(e^{\frac y{2B}} \partial_y^m\e(t,y)\right)^2
\lesssim \|\e\|_{\dot H_B^m}(\|\e\|_{\dot H_B^m}+\|\e\|_{\dot H_B^{m+1}})
\lesssim t^{4K-2m-1} \lesssim t^{3K+3} \lesssim 1.
\]
Using~\eqref{eq:Qbj} and~\eqref{Spar}, we deduce, for all $t\in (0,T_0]$, for all $y\in\RR$,
\[
|\partial_y^m Q_{b(t)}(y)|+| \partial_y^m\e(t,y)|\lesssim e^{-\frac y{2B}}.
\]
Thus, by~\eqref{eq:decomposition} and~\eqref{Spar}, with $\epsilon=1$, $\ell_0=1$ and $x_0=0$, we obtain, for all $x\in\RR$,
\[
|\partial_x^mS(t,x)|\leq \frac{C_m'}{t^{\frac 12+m}} \exp\left(- {\frac{x+\frac1t}{4Bt}}\right).
\]
Hence, setting $\g_m=(4B)^{-1}$,~\eqref{th:ptwiser} is proved, and we notice that this control, for $t\in (0,T_0]$ fixed, is only interesting
for $x\geq -\frac1t$, that is to say on the right of the soliton.

\subsection{$L^1$ norm and integral of $S(t)$}

\begin{proposition} \label{p:L1}
Let $0<a<\frac 1{20}$.
For all $t\in (0,T_0]$,
\be \label{e:L1}
\int |S(t,x)|\,dx = 2\|Q\|_{L^1}t^{\frac12} + O(t^{\frac12 +a}).
\ee
\end{proposition}

\begin{proof}
Let $0<a<\frac 1{20}$, $t\in (0,T_0]$, $R_1(t)=t^{-1}+1$ and $R_2(t)=t^{-1}-1$.
We estimate the $L^1$~norm of $S(t)$ on the three regions
$(-\infty,-R_1(t)]$, $[-R_1(t),-R_2(t)]$ and $[-R_2(t),+\infty)$.

First, we integrate~\eqref{eq:ptgche} on $[R_1(t),+\infty)$ and obtain
\[
\int_{-\infty}^{-R_1(t)} |S(t,x)|\, dx = \|Q\|_{L^1} (R_1(t))^{-\frac 12} + O((R_1(t))^{-\frac12-a})
= \|Q\|_{L^1} t^{\frac 12} + O(t^{\frac12+a}).
\]

Next, from the decomposition~\eqref{eq:decomposition} of~$S$, we have
\[
\int_{-R_1(t)}^{-R_2(t)} |S(t,x)|\,dx = \lambda^{\frac12}(t) \int_{y_1(t)}^{y_2(t)} |Q_{b(t)}(y) +\e(t,y)|\,dy,
\]
where $y_i(t)=\frac{-R_i(t)-x(t)}{\lambda(t)}$ for $i=1,2$ satisfy, from~\eqref{rappel},
\[
-\frac94 t^{-1} \leq y_1(t)\leq -\frac14 t^{-1}\quad \m{and}\quad \frac14 t^{-1} \leq y_2(t)\leq \frac94 t^{-1}.
\]
But, for $y\in [-3t^{-1},3t^{-1}]$, we have, using~\eqref{eq:QbtoQ} and~\eqref{eq:pteps},
\[
|Q_{b(t)}(y) +\e(t,y)| \leq Q(y) +C|b(t)| + Ct^{2+\frac54 \g} \leq Q(y) + Ct^2,
\]
for some $C>0$, and so, by integration,
\[
\int_{y_1(t)}^{y_2(t)} |Q_{b(t)}(y) +\e(t,y)|\,dy \leq \int_{-\frac94 t^{-1}}^{\frac94 t^{-1}} |Q_{b(t)}(y) +\e(t,y)|\,dy
\leq \|Q\|_{L^1} +Ct.
\]
Similarly, we obtain
\[
\int_{y_1(t)}^{y_2(t)} |Q_{b(t)}(y) +\e(t,y)|\,dy \geq \int_{-\frac14 t^{-1}}^{\frac14 t^{-1}} |Q_{b(t)}(y) +\e(t,y)|\,dy
\geq \|Q\|_{L^1} - C e^{-\frac14 t^{-1}} - C t.
\]
Since $\lambda(t)=t+O(t^3)$, we have obtained
\[
\int_{-R_1(t)}^{-R_2(t)} |S(t,x)|\,dx = \|Q\|_{L^1} t^{\frac 12} + O(t^{\frac32}).
\]

Finally, to estimate the term
\[
\int_{-R_2(t)}^{+\infty} |S(t,x)|\,dx = \lambda^{\frac12}(t) \int_{y>y_2(t)} |Q_{b(t)}(y) +\e(t,y)|\,dy,
\]
we use~\eqref{eq:Qbj} and the estimate of $\|\e(t)\|_{\dot H_B^0}$ in~\eqref{SHm} to obtain
\begin{align*}
\int_{-R_2(t)}^{+\infty} |S(t,x)|\,dx &\lesssim \int_{y>\frac14 t^{-1}} e^{-y/2}\,dy + \int_{y>0} |\e(t,y)|\,dy \\
&\lesssim e^{-\frac18 t^{-1}} + \left(\int_{y>0} \e^2(t,y)e^{y/B}\,dy \right)^{\frac12} \left( \int_{y>0} e^{-y/B}\,dy \right)^{\frac12} \\
&\lesssim e^{-\frac18 t^{-1}} + \|\e(t)\|_{\dot H_B^0} \lesssim t^{2K} \lesssim t^{40}.
\end{align*}

Thus, gathering the three above estimates of the $L^1$ norm of $S(t)$ for each region described previously, we obtain~\eqref{e:L1}.
\end{proof}

\begin{corollary}
For all $t\in (0,T_0]$,
\be \label{e:int}
\int S(t,x)\,dx=0.
\ee
\end{corollary}

\begin{proof}
From Proposition~\ref{p:L1}, we deduce that, for all $t\in (0,T_0]$, $S(t)\in L^1(\RR)$ and in particular $\int S(t)$ is well defined.
Moreover, if $u$ is a solution of~\eqref{kdv} defined on some interval $I$ and $u(t)\in L^1(\RR)$ for all $t\in I$,
then $\int u(t)$ does not depend on $t$.
Thus, for some constant $C_0\in\RR$, we have $\int S(t)=C_0$ for all $t\in (0,T_0]$, and passing to the limit $t\to0$ in~\eqref{e:L1},
we obtain $C_0=0$, and so~\eqref{e:int}.
\end{proof}

\end{document}